\documentclass[10pt,oneside,a4paper,reqno]{amsbook}  

\usepackage{amssymb}
\usepackage[colorlinks=true, linkcolor=magenta, citecolor=cyan, urlcolor=blue]{hyperref}
\usepackage{graphicx}
\usepackage{color}
\definecolor{sap}{RGB}{129,36,51}
\usepackage{bbm}

\def\ni{\noindent}

\numberwithin{equation}{chapter}\newtheorem{theorem}{Theorem}[chapter]
\newtheorem{corollary}[theorem]{Corollary}\newtheorem{lemma}[theorem]{Lemma}

\newtheorem{proposition}[theorem]{Proposition}\theoremstyle{remark}
\newtheorem{remark}{Remark}[chapter]
\newtheorem*{notation}{Notation}\theoremstyle{definition}
\newtheorem{definition}[theorem]{Definition}

\newcommand{\bra}[1]{\langle #1 \rangle}

\newcommand{\p}{\widetilde{p}}
\newcommand{\q}{\widetilde{q}}
\newcommand{\s}{\widetilde{s}}
\newcommand{\R}{\widetilde{r}}
\newcommand{\Lqqtilde}{L^{q}_{|x|}L^{\widetilde{q}}_{\theta}}
\newcommand{\Lpptilde}{L^{p}_{|x|}L^{\widetilde{p}}_{\theta}}

\newcommand{\Lpptildealphazero}{L^{p_{0}}_{|x|^{\alpha_{0} p_{0}}d|x|}L^{\widetilde{p}_{0}}_{\theta}}

\newcommand{\Lqqtildebetazero}{L^{q_{0}}_{|x|^{\beta_{0} q_{0}}d|x|}L^{\widetilde{q}_{0}}_{\theta}}
\newcommand{\Lpptildealphauno}{L^{p_{1}}_{|x|^{\alpha_{1} p_{1}}d|x|}L^{\widetilde{p}_{1}}_{\theta}}
\newcommand{\Lqqtildebetauno}{L^{q_{1}}_{|x|^{\beta_{1} q_{1}}d|x|}L^{\widetilde{q}_{1}}_{\theta}}
\newcommand{\Lpptildealphaxi}{L^{p_{\xi}}_{|x|^{\alpha_{\xi} p_{\xi}}d|x|}L^{\widetilde{p}_{\xi}}_{\theta}}
\newcommand{\Lqqtildebetaxi}{L^{q_{\xi}}_{|x|^{\beta_{\xi} q_{\xi}}d|x|}L^{\widetilde{q}_{\xi}}_{\theta}}
\newcommand{\Lrq}{L^{r}_{t}L^{q}_{x}}
\newcommand{\Lsp}{L^{s}_{t}L^{p}_{x}}
\newcommand{\Lrqqtilde}{L^{r}_{t}L^{q}_{x}L^{\widetilde{q}}_{\theta}}
\newcommand{\Lspptilde}{L^{s}_{t}L^{p}_{x}L^{\widetilde{p}}_{\theta}}
\newcommand{\Rn}{\mathbb{R}^{n}}
\newcommand{\Rpiu}{\mathbb{R}^{+}}

\title[Angular Integrability and Navier-Stokes equation]
{Tesi phd}
\date{\today}    

\makeindex

\begin{document}

\pagestyle{plain}

\begin{titlepage}
\begin{center}

\vspace{0.7cm}
{\large {\LARGE S}CUOLA DI {\LARGE D}OTTORATO {\LARGE "V}ITO {\LARGE
V}OLTERRA{\LARGE"}}\\
\vspace{0.2cm}
{\large{\LARGE D}OTTORATO DI {\LARGE R}ICERCA IN {\LARGE M}ATEMATICA --
{\LARGE XXV} CICLO}\\
\vspace{6cm}
{\textbf{\Huge Inequalities with angular integrability and applications}}

\end{center}
\vspace{6cm}
\large \textbf{Dottorando}\hfill \textbf{Relatore}\\
\Large Renato Luc\`{a} \hfill Prof. Piero D'Ancona\\
\vspace{3cm}
\begin{center}
\textit{\normalsize{\Large A}NNO {\Large A}CCADEMICO {\Large2012--2013}}
\end{center}
\end{titlepage}

\tableofcontents

\chapter*{Introduction}
We study the improvements due to the angular regularity in the context of Sobolev embeddings and PDEs. It is well known that many fundamental inequalities in mathematical analysis get improvements under some additional symmetry assumptions. Such improvements are related to the geometric nature of the space and in particular to the action of a certain group of symmetry. 
This is not surprising because a symmetric function on a differentiable manifold can be considered as a function defined on lower dimensional manifold on wich stronger estimates are often available. Then such improved estimates can be extended on the whole manifold by the action of a certain group. In particular we work in a very simple setting by considering radially symmetric functions defined on $\Rn$. Such functions are indeed defined on $\Rpiu$ and $SO(n)$ acts to go back to $\Rn$.

\ni For instance the Hardy-Littelwood (\cite{Vilela01-a}, \cite{HidanoKurokawa08-a}, \cite{Rubin83-a}, \cite{DenapoliDrelichmanDuran11-a}), Caffarelli-Kohn-Nirenberg (\cite{DenapoliDrelichmanDuran10-a}) and Strichartz (\cite{DanconaCacciafesta11-a}, \cite{MachiharaNakamuraNakanishi05-a}, \cite{Sterbenz05-a}) inequality on $\Rn$ get improvements if one restricts to consider just radially symmetric functions. A natural question is if and how this phenomenon occurs if we replace the symmetry hypotesis with merely higher integrability in the angular variables. We show that improvements can be obtained by working with the norms:
\begin{equation*}
\begin{array}{lcl}
  \|f\|_{L^{p}_{|x|}L^{\p}_{\theta}}&=&
  \left(
    \int_{0}^{+\infty}
    \|f(\rho\ \cdot \ )\|^{p}_{L^{\p}(\mathbb{S}^{n-1})}
    \rho^{n-1}d \rho
  \right)^{\frac1p}, \\
  \|f\|_{L^{\infty}_{|x|}L^{\p}_{\theta}}&=&
  \sup_{\rho>0}\|f(\rho\ \cdot \ )\|_{L^{\p}(\mathbb{S}^{n-1})};
\end{array}
\end{equation*}

\ni we observe that such results interpolate beetween the improved versions for radially symmetric functions and the classical ones. 

\ni We develop widely this technology in the first chapter. The main results proved are extensions of Hardy-Littelwood-Sobolev (theorem \ref{the:Our1Thm}, corollary \ref{cor:nonhom}) and Caffarelli-Kohn-Nirenberg inequality (theorem \ref{the:Our2Thm}). Another interesting aspect is that the inequalities on which we focus are fundamental tools in the study of PDEs. For instance the Caffarelli-Kohn-Nirenberg inequality is really important in the study of regularity of the Navier-Stokes equation's solutions, because it provides a priori estimates through interpolation of quantities related to the energy dissipation. So it is expectable that the technology developed can lead to succesfully results also in this areas. We focus basically on the small data theory and on regularity criteria. We actually extended the criteria in \cite{YongZhou} and we made a conjecture on a possible extension of the small data result in \cite{CKN}. 

\ni As well known the well posedness of the Cauchy Problem for the Navier-Stokes equation
\begin{equation*}
\left \{
\begin{array}{rcccl}
\partial_{t}u + (u \cdot \nabla) u +\nabla p  & = & \Delta u & \quad \mbox{in}& \quad \Rpiu \times \mathbb{R}^{n}  \\
\nabla \cdot u & = & 0 & \quad \mbox{in} & \quad \Rpiu \times \mathbb{R}^{n} \\
u & = & u_{0} & \quad \mbox{in}& \quad \{0\} \times \mathbb{R}^{n}, 
\end{array}\right.
\end{equation*}
is a big mathematical challenge, and only partial results have been obtained.
Leray has proved global existence of weak solutions for $L^{2}$ initial data in his pioneering work \cite{Ler}. On the other hand the uniqeness of Leray's solutions is still open, as it is the propagation of regularity of the initial datum $u_{0}$. The well posedness theory is well developed for short times, or if one restricts to small $u_{0}$. In this scenario is useful to estabilish at least regularity or uniqeness criteria. We mean to find some a priori assumptions on the solutions under which the regularity (or the uniqeness) is guaranteed. We focus on the regularity in space variables, being the time regularity a different and more difficult problem. This is actually expectable by a physical viewpoint, because of the assumption of incompressibility of the fluid; see \cite{Ser}. It turns out that the a priori assumptions requested in order to get regularity are basically boundedness assumptions on $u, \nabla u$, or $\nabla \times u$. 

\ni As mentioned we refer basically to the criteria in \cite{YongZhou}, in which the main novelty is to consider boundedness in weighted $L^{p}$ spaces with weights $|x|^{\alpha}$. More precisely the author works with solutions $u:[0,T]\times \Rn \rightarrow \Rn$, equipped with the norms:
\begin{equation}\label{IntrodScaling}
\||x|^{\alpha}u\|_{L^{s}_{T}L^{p}_{x}}, \qquad \frac{2}{s}+\frac{n}{p}=1-\alpha,
\end{equation}      
where, of course, the indexes relations follow by scaling considerations. Under boundedness assumptions the regularity in the segment $(0,T)\times \{ 0 \}$ is achieved, more precisely u is $C^{\infty}$ in the space variables. So the introduction of weights allows to get local regularity criteria, where we mean only in a neighborhood of the origin. At first we show how for some choices of the indexes the criteria in \cite{YongZhou} are indeed global (the regularity is achieved in $(0,T)\times \Rn$). Then we get improvements by working with spaces with different integrability in radial and angular directions; so instead of the norms (\ref{IntrodScaling}) we use:     
\begin{equation}\nonumber
  \| |x|^{\alpha} u \|_{L^{s}_{T}L^{p}_{|x|}L^{\p}_{\theta}} = 
   \left( \int_{0}^{T} \left(
    \int_{0}^{+\infty}
    \|u(t,  \rho \ \cdot \ )\|^{p}_{L^{\p}(\mathbb{S}^{n-1})}
     \ \rho^{\alpha p + n-1} \ d \rho  
  \right)^{\frac{s}{p}} \ dt \right)^{\frac{1}{s}}, 
\end{equation}
with $\frac{2}{s}+\frac{n}{p}=1-\alpha$. In this setting we get global regularity if sufficiently high values of $\p$ are considered. We observe, as expectable, two different behaviour in the ranges $\alpha <0$ and $\alpha > 0$; we show regularity in the case $\p \geq \p_{G}$ where
$$
\p_{G}= \frac{(n-1)p}{\alpha p +n -1};
$$  
and of course:
\begin{equation*}
p < \p_{G}, \qquad \mbox{if} \quad \alpha <0, 
\end{equation*}
\begin{equation*}
\p_{G} < p, \qquad \mbox{if} \quad \alpha > 0.
\end{equation*}
A similar analysis has been performed about the well posedness with small data in mixed angular-radial weighted spaces equipped with the norms:
\begin{equation}
  \||x|^{\alpha} u_{0} \|_{L^{p}_{|x|}L^{\p}_{\theta}} =    
   \left(
    \int_{0}^{+\infty}
    \|u_{0}(  \rho \ \cdot \ )\|^{p}_{L^{\p}(\mathbb{S}^{n-1})}
    \rho^{\alpha p + n-1} \ d \rho
  \right)^{\frac1p},
\end{equation}
with again the critical scaling relationship $\alpha = 1 + \frac{n}{p}$. Here we find out the critical value
\begin{equation}\nonumber
\p_{G} = \frac{(n-1)p}{p-1},
\end{equation}
and the well posedness is achieved for small data, with $\p \geq \p_{G}$.
Actually we have closely looked at the following heuristic: the weights $|x|^{\alpha}$, $\alpha<0$ localize, in some sense, the norms of the data (or of the solutions) near to the origin. In such a way local results are still available, but a loss of informations far from the origin occurs. These informations can be recovered every times by a suitbale amount of angular integrability. 

\ni On the other hand the local (in the sense of localized near to the origin) results have an intrinsic interest, and we also look at this problem. In theorem \ref{OurYZTheoremLoc} we prove local regularity for bounded solutions in 
$$
\| |x|^{\alpha} u \|_{L^{s}_{T}L^{p}_{|x|}L^{\p}_{\theta}}, 
\qquad \mbox{with} \quad \p \leq p. 
$$
This improves in a different direction the results in \ref{YZTheorem}. We get local regularity under the assumption of a sufficiently high angular regularity, i.e for $\p \geq \p_{L}$ where 

\begin{equation*}
\p_{L} = \left \{
\begin{array}{lcr}
 \frac{2(n-1)p}{(2 \alpha +1)p +2(n-1)}  &  \mbox{if} & -\frac{1}{2} \leq \alpha < 0  \\
 && \\
 \frac{2(n-1)p}{p +2(n-1)}  &  \mbox{if} & 0 < \alpha < 1. 
\end{array}\right.
\end{equation*}
The main technical tools we use consist in decay estimates for convolutions with the heat and Oseen kernels in the context of weighted $L^{p}_{|x|}L^{\p}_{\theta}$ spaces. Estimates in weighted spaces have been considered in literature, but the information provided by the angular integrability leads to a really satisfactory admissibility range for the weights. The precise relation between the weights and the angular integrability is basically contained in the relation (\ref{eq:condDL}) in the corollary \ref{cor:nonhom}.

\ni The small data theory in the context of weighted $L^{p}_{|x|}L^{\p}_{\theta}$ spaces with $\p < p$ is more delicate and we just make a conjecture about a possible improvement of theorem \ref{CKNSmallData}, in which the authors show regularity of the Leray's solutions in the interior of the space-time parabola:
$$
\Pi = \left\{ (t,x) \quad \mbox{s.t.} \quad t >  \frac{|x|^{2}}{\varepsilon_{0} - \varepsilon}  \right\},
$$
for a sufficiently small $\varepsilon_{0} > 0$ and data with $\||x|^{-1/2} \cdot\|_{L^{2}_{x}} = \varepsilon < \varepsilon_{0}$. The conjecture in section \ref{smallDataMixed} is made in order to cover the gap between this localized result and the classical well posedness results (in particular we refer to theorem \ref{WeighKato} that's a particular case of the Koch-Tataru theorem \ref{TatTheo}).  

\begin{remark}
Of course by translations all the results are still valid if the norms and the weights are centered in a point $\bar{x} \neq 0$. So if we consider
 \begin{equation*}
 \begin{array}{lcl}
  \|  f \|_{L^{p}_{|x-\bar{x}|}L^{\p}_{\theta}}&=&
  \left(
    \int_{0}^{+\infty}
    \|f( \bar{x} + \rho \theta )\|^{p}_{L^{\p}(\mathbb{S}^{n-1})}
    \rho^{n-1}d \rho
  \right)^{\frac1p}, \\
  \| |x-\bar{x}|^{\alpha} u \|_{L^{s}_{T}L^{p}_{|x-\bar{x}|}L^{\p}_{\theta}} &=& 
   \left( \int_{0}^{T} \left(
    \int_{0}^{+\infty}
    \|u( \bar{x} + \rho \theta )\|^{p}_{L^{\p}(\mathbb{S}^{n-1})}
     \ \rho^{\alpha p + n-1} \ d \rho  
  \right)^{\frac{s}{p}} \ dt \right)^{\frac{1}{s}}, \\
  \||x-\bar{x}|^{\alpha} u_{0} \|_{L^{p}_{|x-\bar{x}|}L^{\p}_{\theta}} &=&    
   \left(
    \int_{0}^{+\infty}
    \|u_{0}( \bar{x} + \rho \theta )\|^{p}_{L^{\p}(\mathbb{S}^{n-1})}
    \rho^{\alpha p + n-1} \ d \rho
  \right)^{\frac1p},
\end{array}
\end{equation*}  
and so on.
\end{remark}
\ni The content of the Chapter \ref{SectInequality} is taken from \cite{DL}.

\chapter{Classical inequalities with angular integrability}\label{SectInequality}

\noindent The goal of this section is to extend some classical estimates in the context of $L^{p}$ spaces to a setting in which the role of the angular and radial integrability is well distinguished. In order of explaining our purpose we start by a well known estimate of Walter Strauss \cite{Strauss77-a} that proves    
\begin{equation}\label{eq:strauss}
  |x|^{\frac{n-1}{2}}|u(x)|\le 
  C\|\nabla u\|_{L^{2}},\qquad|x|\ge1,
\end{equation}
for radial functions
$u\in \dot H^{1}(\mathbb{R}^{n})$, $n\ge2$,
This is an example of a well known general phenomenon:
under suitable assumptions of symmetry, notably
radial symmetry, classical estimates and embeddings of spaces
admit substantial improvements. In the case of \eqref{eq:strauss},
a control on the $H^{1}$ norm of $u$ gives a pointwise
bound and decay of $u$. Both are false in the general case.
Radial and more general symmetric estimates have been extensively
investigated, in view of their relevance for applications,
especially to differential equations.

\noindent This phenomenon is quite natural; indeed, symmetric functions
can be regarded as functions defined on lower dimensional manifolds,
on which stronger estimates are available, then extended by the action of
some group of symmetries. Radial functions are essentially functions
on $\mathbb{R}^{+}$, while the norms on $\mathbb{R}^{n}$ are recovered by the action of $SO(n)$ that introduces
suitable dimensional weights connected to the volume form.

\noindent In view of the gap between the symmetric and the non symmetric
case, an
interesting question arises: is it possible to quantify the
defect of symmetry of functions and prove more general
estimates which encompass all cases, and in particular
reduce to radial estimates when applied to radial functions?
Heuristically, one should be
able to improve on the general case by introducing some
measure of the distance from the maximizers of the
inequality, which typically have the greatest symmetry.

\noindent The aim of this paper is to give a partial 
positive answer to this question, through the use of
the following type of mixed radial-angular norms:
\begin{equation*}
\begin{array}{lcl}
  \|f\|_{L^{p}_{|x|}L^{\p}_{\theta}}&=&
  \left(
    \int_{0}^{+\infty}
    \|f(\rho\ \cdot\ )\|^{p}_{L^{\p}(\mathbb{S}^{n-1})}
    \rho^{n-1}d \rho
  \right)^{\frac1p}, \\
  \|f\|_{L^{\infty}_{|x|}L^{\p}_{\theta}}&=&
  \sup_{\rho>0}\|f(\rho\ \cdot\ )\|_{L^{\p}(\mathbb{S}^{n-1})}.
\end{array}
\end{equation*}
When the context is clear we shall write simply $L^{p}L^{\p}$.
For $p=\p$ the norms reduce to the usual $L^{p}$ norms
\begin{equation*}
  \|u\|_{L^{p}_{|x|}L^{p}_{\theta}}\equiv
  \|u\|_{L^{p}(\mathbb{R}^{n})},
\end{equation*}
while for radial functions the value of $\p$ is irrelevant:
\begin{equation*}
  \text{$u$ radial}\quad\implies\quad 
  \|u\|_{L^{p}L^{\p}}\simeq \|u\|_{L^{p}(\mathbb{R}^{n})}
  \quad \forall p,\p\in[1,\infty].
\end{equation*}
Notice also that the norms are increasing in $\p$.
The idea of distinguishing radial and angular directions 
is not new and
has proved successful in the context of
Strichartz estimates and dispersive equations
(see \cite{MachiharaNakamuraNakanishi05-a},
\cite{Sterbenz05-a},
\cite{DanconaCacciafesta11-a}; see also
\cite{ChoOzawa09-a}). To give a flavour of the results which can
be obtained, Strauss' estimate \eqref{eq:strauss} can be extended
as follows:
\begin{equation*}
  |x|^{\frac np-\sigma}|u(x)|\lesssim
  \||D|^{\sigma}u\|_{L^{p}L^{\p}},\qquad
  \frac {n-1}\p+\frac1p<\sigma <\frac np
\end{equation*}
for arbitrary non radial functions $u$ and all $1<p<\infty$,
$1\le\p\le \infty$.

\begin{remark}
Of course by translations all the results we will prove hold with the norm
\begin{equation*}
\begin{array}{lcl}
  \|f\|_{L^{p}_{|x-\bar{x}|}L^{\p}_{\theta}}&=&
  \left(
    \int_{0}^{+\infty}
    \|f(  \bar{x} + \rho \theta )\|^{p}_{L^{\p}(\mathbb{S}^{n-1})}
    \rho^{n-1}d \rho
  \right)^{\frac1p}, \\
  \|f\|_{L^{\infty}_{|x-\bar{x}|}L^{\p}_{\theta}}&=&
  \sup_{\rho>0}\|f(  \bar{x} + \rho \theta )\|_{L^{\p}(\mathbb{S}^{n-1})};
\end{array}
\end{equation*}
\end{remark}

\section{The Stein-Weiss inequality}

\noindent A central role in our approach will be
played by the fractional integrals
$$
(T_{\gamma}\phi)(x)=\int_{\mathbb{R}^{n}}
  \frac{\phi(y)}{|x-y|^{\gamma}}dy, \qquad 0<\gamma<n.
$$
Weighted $L^{p}$ estimates for $T_{\gamma}$
are a fundamental problem of harmonic analysis, with a wide
range of applications. Starting from the
classical one dimensional case studied by Hardy and Littlewood,
an exhaustive analysis has been made of the admissible classes of
weights and ranges of indices
(see \cite{Stein93-a} and the references therein).
In the special case of power weights the optimal result is
due to Stein and Weiss:

\begin{theorem}[\cite{SteinWeiss58-b}]\label{SteinWeissThm}
Let $n\geq 1$ and $1< p\le q<\infty$.
Assume $\alpha,\beta,\gamma$ satisfy the set of conditions
($1=1/p+1/p'$)
\begin{equation}\label{eq:condSW}
\begin{split}
  &\beta<\frac nq,\quad \alpha<\frac{n}{p'},\quad 0<\gamma<n
    \\
  &\alpha+\beta+\gamma=n+\frac nq-\frac np
    \\
  &\alpha+\beta\ge0.
\end{split}
\end{equation}
Then the following inequality holds 
\begin{equation}\label{eq:stw}
  \||x|^{-\beta}T_{\gamma}\phi\|_{L^{q}}\le 
  C(\alpha,\beta,p,q)\cdot
  \||x|^{\alpha}\phi\|_{L^{p}}.
\end{equation}
\end{theorem}

\noindent Conditions in
the first line of \eqref{eq:condSW} are necessary to ensure
integrability,
while the necessity of 
the condition on the second line is due to scaling.
On the other hand, the sharpness of
$\alpha+\beta\ge0$ is less obvious and follows from the results
of \cite{SawyerWheeden92-a}.

\noindent In the radial case the last condition
can be relaxed and $\alpha+\beta$ is allowed to assume negative
values. Radial improvements were noticed in
\cite{Vilela01-a}, 
\cite{HidanoKurokawa08-a}, and the sharp result was obtained
by Rubin \cite{Rubin83-a} and more
recently by De Napoli, Dreichman and Dur\'an:

\begin{theorem}[\cite{Rubin83-a},\cite{DenapoliDrelichmanDuran09-a}]
\label{DeNapoli1Thm}
  Let $n,p,q,\alpha,\beta,\gamma$ be as in the statement
  of Theorem \eqref{SteinWeissThm} but with the condition
  $\alpha+\beta\ge0$ relaxed to
  \begin{equation}\label{eq:condDDD}
    \alpha+\beta\ge(n-1)\left(\frac1q-\frac1p\right).
  \end{equation}
  Then estimate \eqref{eq:stw} is valid for all
  radial functions $\phi=\phi(|x|)$.
\end{theorem}

\noindent Using the $L^{p}_{|x|}L^{\p}_{\theta}$ norms we are able prove
the following general result which extends both theorems:

\begin{theorem}\label{the:Our1Thm}
  Let $n \geq 2$ and $1<p\le q<\infty$, $1\le\p\le\q\le\infty$. Assume
  $\alpha,\beta,\gamma$ satisfy the set of conditions
  \begin{equation}\label{eq:cDL}
  \begin{split}
    &\beta<\frac nq,\quad \alpha<\frac{n}{p'},\quad 0<\gamma<n
      \\
    &\alpha+\beta+\gamma=n+\frac nq-\frac np
      \\
    &\alpha+\beta\ge(n-1)
      \left(\frac1q-\frac1p+\frac{1}{\p}-\frac{1}{\q}\right).
  \end{split}
  \end{equation}
  Then the following estimate holds: 
  \begin{equation}\label{oHLS}
    \||x|^{-\beta}T_{\gamma} \phi\|
        _{L^{q}_{|x|}L^{\q }_{\theta}} 
    \le C
    \| |x|^{\alpha} \phi\|_{L^{p}_{|x|}L^{\p }_{\theta}}.
  \end{equation}
  The range of admissible $p,q$ indices 
  can be relaxed to $1\le p\le q\le \infty$ in two cases:
  \begin{enumerate}\setlength{\itemindent}{-0pt}
  \renewcommand{\labelenumi}{\textit{(\roman{enumi})}}
    \item when the third inequality in \eqref{eq:cDL} is strict, or
    \item when the Fourier transform $\widehat{\phi}$ has support
    contained in an annulus $c_{1}R\le|\xi|\le c_{2} R$
    ($c_{2}\ge c_{1}>0$, $R>0$);
    in this case \eqref{oHLS} holds with a constant independent 
    of $R$.
  \end{enumerate}
\end{theorem}

\begin{remark}
Notice that:
\begin{enumerate}\setlength{\itemindent}{-0pt}
\renewcommand{\labelenumi}{(\alph{enumi})}
  \item with the choices $q=\q $ and $p=\p $ (i.e.~in the usual 
  $L^{p}$ norms)
  Theorem \ref{the:Our1Thm} reduces to Theorem \ref{SteinWeissThm};
  \item if $\phi$ is radially symmetric, with the choice $\q =\p $,
  Theorem \ref{the:Our1Thm} reduces to Theorem \ref{DeNapoli1Thm}.
  Indeed, if $\phi$ is radially symmetric then
  $T_{\gamma}\phi$ is radially symmetric too, so that all choices
  for $\q,\p $ are equivalent;
  \item obviously, the same estimate is true for general 
  operators $T_{F}$ with nonradial kernels $F(x)$ satisfying
  \begin{equation*}
    T_{F}\phi(x)=\int F(x-y) \phi(y)dy,
    \qquad |F|\le C|x|^{-\gamma}.
  \end{equation*}
\end{enumerate}
\end{remark}

\noindent The proof of Theorem \ref{the:Our1Thm} is based on
two successive applications of Young's inequality 
for convolutions on suitable
Lie groups: first we use the strong inequality on the rotation
group $SO(n)$; then we use a Young inequality in the radial
variable, which in some cases must be replaced by 
the weak Young-Marcinkewicz
inequality on the multiplicative group $(\mathbb{R}^{+},\cdot)$
with the Haar measure $d\rho/\rho$. The convenient idea of using
convolution in the measure $d\rho/\rho$ was introduced in
\cite{DenapoliDrelichmanDuran09-a}.  

\begin{remark}\label{rem:nonhomogeneous}
  The operator $T_{\gamma}$ is a convolution with the homogenous
  kernel $|x|^{-\gamma}$. Consider instead
  the convolution with a nonhomogeneous kernel
  \begin{equation*}
    S_{\gamma}\phi(x)=\int \frac{\phi(y)}{\bra{x-y}^{\gamma}}dy.
  \end{equation*}
  By the obvious pointwise bound
  \begin{equation*}
    |S_{\gamma}\phi(x)|\le T_{\gamma}|\phi|(x)
  \end{equation*}
  it is clear that $S_{\gamma}$ satisfies the same estimates as
  $T_{\gamma}$. However the scaling invariance of the estimate is
  broken, and indeed something more can be proved, thanks to the
  smoothness of the kernel (see Lemma \ref{lem:xmu}):
\end{remark}

\begin{corollary}\label{cor:nonhom}
  Let $n \geq 2$ and $1\le p\le q\le\infty$, 
  $1\le\p\le\q\le\infty$. Assume
  $\alpha,\beta,\gamma$ satisfy the set of conditions
  \begin{equation}\label{eq:condDL}
    \beta<\frac nq,\qquad \alpha<\frac{n}{p'},\qquad
    \alpha+\beta \ge (n-1)
      \left(\frac1q-\frac1p+\frac{1}{\p}-\frac{1}{\q}\right),
  \end{equation}
  \begin{equation}\label{eq:condabg}
    \alpha+\beta+\gamma>n\left(1+\frac1q-\frac1p\right).
  \end{equation}
  Then the following estimate holds: 
  \begin{equation}\label{ourHLS}
    \||x|^{-\beta}S_{\gamma} \phi\|
        _{L^{q}_{|x|}L^{\q }_{\theta}} 
    \le C
    \| |x|^{\alpha} \phi\|_{L^{p}_{|x|}L^{\p }_{\theta}}.
  \end{equation}
\end{corollary}

\ni The first result we need is an explicit estimate of 
the angular part of the fractional integral
$T_{\gamma}\phi$. Notice that a similar analysis
in the radial case was done in
\cite{DenapoliDrelichmanDuran09-a} (see Lemma 4.2 there).
The following estimates are sharp:

\begin{lemma}\label{lem:singint}
  Let $n\ge2$, $\nu>0$, and write
  $\bra{x}=(1+|x|^{2})^{1/2}$.
  Then the integral
  \begin{equation*}
    I_{\nu}(x)=\int_{\mathbb{S}^{n-1}}|x-y|^{-\nu }dS(y)
    \qquad x\in \mathbb{R}^{n}
  \end{equation*}
  satisfies
  \begin{equation}\label{eq:stima0}
    |I_{\nu}(x)|\simeq\bra{x}^{-\nu }\qquad
    \text{for}\quad|x|\ge2,
  \end{equation}
  while for $|x|\le2$ we have
  \begin{equation}\label{stimaI}
    |I_{\nu}(x)| \simeq
    \left\{ 
    \begin{array}{cc}
      1& 
            \mbox{if} \ \ \nu <n-1 \\
      |\log{||x|-1|}| + 1& 
            \mbox{if} \ \ \nu =n-1 \\
     ||x|-1|^{n-1- \nu }  & 
             \mbox{if} \ \ \nu  > n-1.
    \end{array} 
    \right. 
  \end{equation}
\end{lemma}

\begin{proof}
We consider four different regimes according to the size of
$|x|$. We write for brevity $I$ instead of $I_{\nu}$.

\subsection*{First case: $|x|\geq 2$} 
For $x$ large and $|y|=1$
we have $|x-y| \simeq |x|$, 
hence $|I(x)| \simeq |x|^{-\nu } 
\simeq \langle x \rangle^{-\nu }$.
This proves \eqref{eq:stima0}.

\subsection*{Second case: $0 \leq |x| \leq \frac{1}{2}$} 
Clearly we have $|x-y| \simeq 1$ when $|y|=1$, 
and this implies 
$|I(x)| \simeq 1 \simeq \langle x \rangle^{-\nu }$.
This is equivalent to \eqref{stimaI} when $|x|\le1/2$.

\subsection*{Third case: $1 \leq |x| \leq 2$} 
This is the bulk of the computation since it contains the
singular part of the integral, as $|x|\to1$.
We write the integral in polar coordinates using the
spherical angles $(\theta_{1},\theta_{2},...,\theta_{n-1})$ 
on $\mathbb{S}^{n-1}$, oriented in such a way that
$\theta_{1}$ is the angle between $x$ and $y$.
Using the notation $\sigma=|x-y|$,
by the symmetry of $I(x)$ in $(\theta_{2},...,\theta_{n-1})$ we have
$$
  |I(x)| \simeq \int_{0}^{\pi}
  \sigma^{-\nu }(\sin{\theta_{1}})^{n-2}d\theta_{1}.
$$
In order to rewrite the integral using $\sigma$ as a new
variable, we compute 
$$
  2\sigma d\sigma = d(|x-y|^{2})=d(|x|+1 -2|x| \cos{\theta_{1}})
  =2|x|\sin{\theta_{1}}d\theta_{1}
$$
so we have
$$
  (\sin{\theta_{1}})^{n-2}d\theta_{1} = 
  \frac{\sigma (\sin{\theta_{1}})^{n-3}}{|x|}d\sigma
$$
and, noticing that $0\le|x|-1\le|x-y|=\sigma\le|x|+1$,
$$
  |I(x)| \simeq \int_{|x|-1}^{|x|+1}
  \sigma^{1-\nu }\frac{(\sin{\theta_{1}})^{n-3}}{|x|}d\sigma.
$$
Now let $A$ be the area of the triangle with vertices
$0,x$ andd $y$: we have $2A=|x|\sin{\theta_{1}}$ so that
$$
  |I(x)| \simeq  |x|^{2-n}\int_{|x|-1}^{|x|+1}
  \sigma^{1-\nu }A^{n-3}d \sigma.
$$
Recalling Heron's formula for the area of a triangle as a function
of the length of its sides we obtain
$$
  |I(x)| \simeq |x|^{2-n}\int_{|x|-1}^{|x|+1}
  \sigma^{1-\nu }
  \Bigl[(|x|+\sigma +1)(|x|+\sigma -1)
  (|x|+1 -\sigma)(\sigma +1 -|x|)\Bigr]^{\frac{n-3}{2}}d\sigma.
$$
Notice that this formula is correct for all dimensions $n\ge2$.

Now we split the integral as $I \simeq I_{1} + I_{2}$
with
$$
  I_{1}(x)= |x|^{2-n}\int_{|x|-1}^{|x|}
  \sigma^{1-\nu }
  \Bigl[
  (|x|+\sigma +1)
  (|x|+\sigma -1)
  (|x|+1 -\sigma)
  (\sigma +1 -|x|)
  \Bigr]^{\frac{n-3}{2}}
  d\sigma
$$
and
$$
  I_{2}(x) = |x|^{2-n}\int_{|x|}^{|x|+1}
  \sigma^{1-\nu }
  \Bigl[
  (|x|+\sigma +1)
  (|x|+\sigma -1)
  (|x|+1 -\sigma)
  (\sigma +1 -|x|)
  \Bigr]^{\frac{n-3}{2}}
  d\sigma.
$$
In the second integral $I_{2}$, recalling that $1\le|x|\le2$,
we have
\begin{equation*}
  |x|\simeq
  \sigma \simeq 
  |x|+\sigma+1 \simeq
  |x|+\sigma-1 \simeq
  \sigma+1-|x| \simeq
  1
\end{equation*}
so that
\begin{equation*}
  I_{2}\simeq
  \int_{|x|}^{|x|+1}(|x|+1-\sigma)^{\frac{n-3}{2}}d \sigma
  = \int_{0}^{1}(1-\sigma)^{\frac{n-3}{2}}d \sigma
  \simeq 1.
\end{equation*}
In the first integral $I_{1}$, using that
$1 \leq |x| \leq 2$ and $|x|-1 \leq \sigma \leq |x|$,
we see that 
\begin{equation*}
  |x|\simeq
  (|x|+\sigma +1)\simeq
  (|x|+1 -\sigma)\simeq
  1;
\end{equation*}
moreover,
\begin{equation*}
  1\le \frac{|x|+\sigma-1}{\sigma}\le 2
  \quad\text{so that}\quad |x|+\sigma-1 \simeq \sigma
\end{equation*}
and we have
$$
  I_{1}(x) \simeq \int_{|x|-1}^{|x|}
  \sigma^{1-\nu  + \frac{n-3}{2}}
  (\sigma+1 -|x|)^{\frac{n-2}{2}}d\sigma
$$
or, after the change of variable $\sigma\to\sigma(|x|-1)$,
$$
  I_{1}(x)\simeq(|x|-1)^{n-1-\nu   }
  \int_{1}^{1+\frac{1}{|x|-1}}
  (\sigma -1)^{\frac{n-3}{2}}
  \sigma ^{\frac{n-1}{2}-\nu }d\sigma .
$$
Now split the last integral as $A+B$ where
$$
   A= (|x|-1)^{n-1-\nu }\int_{1}^{2}
   (\sigma -1)^{\frac{n-3}{2}}\sigma ^{\frac{n-1}{2}-\nu }d\sigma 
$$
and
$$
  B=(|x|-1)^{n-1-\nu }
  \int_{2}^{1+\frac{1}{|x|-1}}(\sigma -1)^{\frac{n-3}{2}}
  \sigma ^{\frac{n-1}{2}-\nu }d\sigma;
$$
we have immediately
$$  
  A \simeq (|x|-1)^{n-1 - \nu }
$$
while, keeping into account that $\sigma  \simeq \sigma -1$ 
for $\sigma$ in $(2,1+\frac{1}{|x|-1})$,
$$
  B=(|x|-1)^{n-1-\nu }
  \int_{2}^{1+\frac{1}{|x|-1}}
  \sigma ^{n-2-\nu }d\sigma
$$
which gives
\begin{equation}\label{andamentoB}
  B\simeq 
    \left\{ \begin{array}{cc}
     1& \mbox{if} \ \ \nu <n-1 \\
     |\log{||x|-1|}| + 1& \mbox{if} \ \ \nu =n-1 \\
     ||x|-1|^{n-1- \nu }  & \mbox{if} \ \ \nu  > n-1 
  \end{array} \right.
\end{equation}

\subsection*{Fourth case: $\frac{1}{2} \leq |x| \leq 1$} 
Using the change of variable $|x'|=1/|x|$, we see that
$|I(x)|\simeq |I(1/|x'|)|$, 
thus the fourth case follows immediately from the third one,
and this concludes the proof of the Lemma.
\end{proof}

\ni We shall also need the following estimate which is
proved in a similar way:

\begin{lemma}\label{lem:singint2}
  Let $n\ge2$, $\nu>0$. Then the integral
  \begin{equation*}
    J_{\nu}(x,\rho)=\int_{\mathbb{S}^{n-1}}
    \bra{x-\rho\theta}^{-\nu}dS(\theta)
    \qquad x\in \mathbb{R}^{n},\ \rho\ge0
  \end{equation*}
  satisfies:
  \begin{equation}\label{eq:stimab0}
    |J_{\nu}(x,\rho)|\simeq\bra{x}^{-\nu }\qquad
    \text{for $\rho\le1$ or $|x|\ge 2\rho$},
  \end{equation}
  \begin{equation}\label{eq:stimac0}
    |J_{\nu}(x,\rho)|\simeq\bra{\rho}^{-\nu }\qquad
    \text{for $|x|\le1$ or $\rho\ge 2|x|$},
  \end{equation}
  while in the remaining case, i.e.~when $|x|\ge1$ and
  $\rho\ge1$ and $2^{-1}|x|\le\rho\le2|x|$,
  \begin{equation}\label{eq:stimabI}
    |J_{\nu}(x,\rho)| \simeq
    \left\{ 
    \begin{array}{cc}
      \bra{\rho}^{-\nu}& 
            \mbox{if} \ \ \nu <n-1 \\
      \bra{\rho}^{-\nu}
          \log\left(\frac{2\bra{\rho}}{\bra{|x|-\rho}}\right)
      & 
            \mbox{if} \ \ \nu =n-1 \\
      \bra{\rho}^{1-n}\bra{|x|-\rho}^{n-1-\nu}  & 
             \mbox{if} \ \ \nu  > n-1.
    \end{array} 
    \right. 
  \end{equation}
  As a consequence, one has
  $J_{\nu}\lesssim \bra{\rho+|x|}^{-\nu}$ when $\nu<n-1$ and
  $J_{\nu}\lesssim \bra{\rho+|x|}^{-\nu}\log(2\bra{\rho}+|x|)$ 
  when $\nu=n-1$.
\end{lemma}

\begin{proof}
  The proof is similar to the proof of Lemma \ref{lem:singint};
  we sketch the main steps. Estimates \eqref{eq:stimab0} and
  \eqref{eq:stimac0} are obvious, thus we focus on \eqref{eq:stimabI}.
  Write $r=|x|$, so that we are in the region $1/2\le r/\rho\le2$;
  we shall consider in detail the case
  \begin{equation*}
    1\le \frac{r}{\rho}\le 2,
  \end{equation*}
  the remaining region being similar.
  Using the same coordinates as before, the integral is reduced to
  \begin{equation*}
    J_{\nu}(|x|,\rho)=
    |x|^{2-n}\int_{|x|-1}^{|x|+1}\bra{\rho \sigma}^{-\nu}
    A^{n-3}\sigma\ d \sigma
  \end{equation*}
  where $A$ is given by Heron's formula
  \begin{equation*}
    A(|x|,\sigma)^{2}=
    (|x|+\sigma+1)
    (|x|+\sigma-1)
    (|x|+1-\sigma)
    (\sigma+1-|x|).
  \end{equation*}
  We split the integral on the intervals $|x|\le\sigma\le|x|+1$ 
  and $|x|-1\le\sigma\le|x|$. The first piece gives 
  \begin{equation*}
    I_{1}\simeq 
     \bra{\rho}^{-\nu}
      \int_{|x|}^{|x|+1}(|x|+1-\sigma)^{\frac{n-3}{2}}d \sigma
  \end{equation*}
  and by the change of variable $\sigma\to \sigma(|x|+1)$
  we obtain
  \begin{equation*}
    I_{1}(|x|,\rho)\simeq \bra{\rho}^{-\nu}.
  \end{equation*}
  For the second integral on $|x|-1\le\sigma\le|x|$, noticing that
  \begin{equation*}
    1\le \frac{|x|+\sigma-1}{\sigma}\le2
  \end{equation*}
  we have
  \begin{equation*}
  \begin{split}
    I_{2}\simeq &
    \int_{|x|-1}^{|x|}
      \bra{\rho \sigma}^{-\nu} \sigma^{\frac{n-1}{2}}
      (\sigma+1-|x|)^{\frac{n-3}{2}}d \sigma
    \\
    = &
    (|x|-1)^{n-1}
    \int_{1}^{\frac{|x|}{|x|-1}}
    \bra{(r-\rho)\sigma}^{-\nu}\sigma^{\frac{n-1}{2}}
    (\sigma-1)^{\frac{n-3}{2}}d \sigma
  \end{split}
  \end{equation*}
  via the change of variables $\sigma\to \sigma(|x|-1)$
  which gives $\rho \sigma\to(r-\rho)\sigma$.
  The part of the integral bewteen 1 and 2 produces
  \begin{equation*}
    \simeq
    (|x|-1)^{n-1}\bra{r-\rho}^{-\nu}=
    \rho^{1-n}(r-\rho)^{n-1}\bra{r-\rho}^{-\nu}
  \end{equation*}
  while the remaining part between 2 and $|x|/(|x|-1)$ gives
  \begin{equation*}
  \begin{split}
    \simeq &
    (|x|-1)^{n-1}
    \int_{2}^{\frac{r}{r-\rho}}
    \bra{(r-\rho)\sigma}^{-\nu} 
    \sigma^{n-2}d \sigma
    \\
    = &
    \rho^{1-n}
    \int_{2(r-\rho)}^{r}\bra{\sigma}^{-\nu}\sigma^{n-2}d \sigma
    \\
    \simeq\ &
    \rho^{1-n}
    \int_{2(r-\rho)}^{r}
    \frac{\sigma^{n-2}}{1+\sigma^{\nu}}d \sigma
  \end{split}
  \end{equation*}
  which can be computed explicitly. Summing up we obtain
  \eqref{eq:stimabI}.
\end{proof}

\ni We are ready for the main part of the proof.
By the isomorphism 
\begin{equation*}
  \mathbb{S}^{n-1}\simeq SO(n)/SO(n-1)
\end{equation*}
we can represent integrals on $\mathbb{S}^{n-1}$ in the form
\begin{equation*}
  \int_{\mathbb{S}^{n-1}}g(y)dS(y)= c_{n}
  \int_{SO(n)}g(Ae)dA,\qquad n\ge2
\end{equation*}
where $dA$ is the left Haar measure on
$SO(n)$, and $e\in\mathbb{S}^{n-1}$ is a fixed arbitrary
unit vector. Thus, via polar coordinates, 
a convolution integral can be written as follows
(apart from inessential constants depending only on the
space dimension $n$):
\begin{equation*}
\begin{split}
  F*\phi(x)=
  \int_{\mathbb{R}^{n}}F(x-y)\phi(y)dy
  &=
  \int_{0}^{\infty}
  \int_{\mathbb{S}^{n-1}}
     F(x-\rho \omega)\phi(\rho \omega)dS_{\omega}\rho^{n-1}d \rho
    \\
  &\simeq
  \int_{0}^{\infty}
  \int_{SO(n)}F(x-\rho Be)\phi(\rho Be)dB\rho^{n-1}d \rho
\end{split}
\end{equation*}
Hence the $L^{\q}$ norm of the convolution
on the sphere can be written as
\begin{equation*}
\begin{split}
  \|F*\phi(|x|\theta)\|_{L^{\q}_{\theta}(\mathbb{S}^{n-1})}
  &\simeq
  \|F*\phi(|x|Ae)\|_{L^{\q}_{A}(SO(n))}
    \\
  &\le 
  \int_{0}^{\infty}
  \left\|
    \int_{SO(n)}F(|x|Ae-\rho Be)\phi(\rho Be)dB
  \right\|_{L^{\q}_{A}(SO(n))}
  \rho^{n-1}d \rho
\end{split}
\end{equation*}
where $e$ is any fixed unit vector. By the change of variables
$B\to AB^{-1}$ in the inner integral
(and the invariance of the measure) 
this is equivalent to
\begin{equation*}
  =\int_{0}^{\infty}
  \left\|
    \int_{SO(n)}F(AB^{-1}(|x|Be-\rho e))\phi(\rho AB^{-1}e)dB
  \right\|_{L^{\q}_{A}(SO(n))}
  \rho^{n-1}d \rho
\end{equation*}
If $F$ satisfies
\begin{equation}\label{eq:rad}
  |F(x)|\le C f(|x|)
\end{equation}
for a radial function $f$, we can write
\begin{equation*}
  |F(AB^{-1}(|x|Be-\rho e))|\le 
  C f\left(\bigl||x|Be-\rho e\bigr|\right)
\end{equation*}
and we notice that the integral
\begin{equation*}
  \int_{SO(n)}f\left(\bigl||x|Be-\rho e\bigr|\right)
     |\phi(\rho AB^{-1}e)|dB=
  g*h(A)
\end{equation*}
is a convolution on $SO(n)$ of the functions
\begin{equation*}
  g(A)=f\left(\bigl||x|Ae-\rho e\bigr|\right),\qquad
  h(A)=|\phi(\rho Ae)|.
\end{equation*}
We can thus apply the Young's inequality on $SO(n)$
(see e.g.~Theorem 1.2.12 in 
\cite{Grafakos08-a}) and we obtain, for any
\begin{equation*}
  \q,\R,\p \in[1,+\infty]
  \quad\text{with}\quad
  1+\frac{1}{\q}=\frac1\R+\frac1\p,
\end{equation*}
the estimate
\begin{equation}\label{eq:firstest}
  \|F*\phi(|x|\theta)\|_{L^{\q}_{\theta}(\mathbb{S}^{n-1})}
  \lesssim 
  \int_{0}^{\infty}
  \|f(||x|e-\rho\theta|)\|_{L^{\R}_{\theta}(\mathbb{S}^{n-1})}
  \|\phi(\rho\theta)\|_{L^{\p}_{\theta}(\mathbb{S}^{n-1})}
  \rho^{n-1}d \rho
\end{equation}
where we switched back to the coordinates of $\mathbb{S}^{n-1}$.
Notice that the conditions on the indices imply in particular
\begin{equation*}
  \q\ge\p.
\end{equation*}
Specializing $f$ to the choice
\begin{equation*}
  f(|x|)=|x|^{-\gamma}
\end{equation*}
we get
\begin{equation*}%
  \|F*\phi(|x|\theta)\|_{L^{\q}_{\theta}}
  \lesssim 
  \int_{0}^{\infty}
  \rho^{-\gamma}
  \||\rho^{-1}|x|e-\theta|^{-\gamma}\|_{L^{\R}_{\theta}}
  \|\phi(\rho\theta)\|_{L^{\p}_{\theta}}
  \rho^{n-1}d \rho
\end{equation*}
which can be written in the form
\begin{equation*}
  =
  |x|^{n-\alpha-\frac np-\gamma}
  \int_{0}^{\infty}
  \left(
  \frac{|x|}{\rho}
  \right)^{\alpha+\frac np-n+\gamma}
  \||\rho^{-1}|x|e-\theta|^{-\gamma}\|_{L^{\R}_{\theta}}
  \rho^{\alpha+\frac np}
  \|\phi(\rho\theta)\|_{L^{\p}_{\theta}}
  \frac{d\rho}{\rho}
\end{equation*}
or equivalently, recalling \eqref{eq:condSW},
\begin{equation*}
  =
  |x|^{\beta-\frac nq}
  \int_{0}^{\infty}
  \left(
  \frac{|x|}{\rho}
  \right)^{-\beta+\frac nq}
  \||\rho^{-1}|x|e-\theta|^{-\gamma}\|_{L^{\R}_{\theta}}
  \rho^{\alpha+\frac np}
  \|\phi(\rho\theta)\|_{L^{\p}_{\theta}}
  \frac{d\rho}{\rho}
\end{equation*}
Following \cite{DenapoliDrelichmanDuran09-a},
we recognize that the
last integral is a convolution
in the multiplicative group $(\mathbb{R},\cdot)$ with the Haar measure
$d \rho/\rho$, which implies
\begin{equation*}
  |x|^{-\beta+\frac nq}
  \|F*\phi(|x|\theta)\|_{L^{\q}_{\theta}}
  \lesssim
  g_{1}*h_{1}(|x|),
\end{equation*}
with
\begin{equation*}
  g_{1}(\rho)=\rho^{-\beta+\frac nq}
  \||\rho e-\theta|^{-\gamma}\|_{L^{\R}_{\theta}},\qquad
  h_{1}(\rho)=
  \rho^{\alpha+\frac np}
  \|\phi(\rho\theta)\|_{L^{\p}_{\theta}}.
\end{equation*}
By the weak Young's inequality in the measure $d\rho/\rho$
(Theorem 1.4.24 in \cite{Grafakos08-a})
we obtain
\begin{equation*}
\begin{split}
  \||x|^{-\beta}F*\phi\|_{L^{q}L^{\q}}\equiv
  &\left\||x|^{-\beta+\frac nq}
     \|F*\phi(|x|\theta)\|_{L^{\q}_{\theta}}
  \right\|_{L^{q}(\rho^{-1}d\rho)}
    \\
  \lesssim &
  \|h_{1}\|_{L^{p}(\rho^{-1}d\rho)}
  \|g_{1}\|_{L^{r,\infty}(\rho^{-1}d\rho)}
\end{split}
\end{equation*}
that is to say
\begin{equation}\label{eq:almostfin}
  \||x|^{-\beta}F*\phi\|_{L^{q}L^{\q}}
  \lesssim
  \|\phi\|_{L^{p}L^{\p}}
  \left\|
    \rho^{-\beta+\frac nq}\||\rho e-\theta|^{-\gamma}\|
       _{L^{\R}_{\theta}}
  \right\|_{L^{r,\infty}(\rho^{-1}d\rho)}.
\end{equation}
provided
\begin{equation*}
  q,r,p\in(1,+\infty)\qquad
  1+\frac1q=\frac1r+\frac1p.
\end{equation*}
In particular this implies
\begin{equation}\label{eq:qp}
  q>p.
\end{equation}
In order to achieve the proof, it remains to check that the
last norm in \eqref{eq:almostfin} is finite.
Notice that, when $\R<\infty$,
\begin{equation*}
  \||\rho e-\theta|^{-\gamma}\|_{L^{\R}_{\theta}}=
  I_{\gamma \R}(\rho e)^{\frac{1}{\R}}
\end{equation*}
where $I_{\nu}$ was defined and estimated in Lemma \ref{lem:singint}.
On the other hand, when $\R=\infty$ one has directly
\begin{equation}\label{eq:rinf}
  \R=\infty \quad\implies\quad
  \||\rho e-\theta|^{-\gamma}\|_{L^{\R}_{\theta}}\simeq
  |\rho-1|^{-\gamma}.
\end{equation}

\ni Using cutoffs, we split the $L^{r,\infty}$ norm in three regions
$0\le \rho\le 1/2$, $\rho\ge2$ and $1/2\le \rho\le2$.

\ni In the region $0\le\rho\le1/2$, recalling 
\eqref{eq:stima0}-\eqref{stimaI} or \eqref{eq:rinf},
we have
\begin{equation*}
  I_{\gamma \R}(\rho e)^{\frac{1}{\R}}\simeq 1
  \quad\implies\quad
  \rho^{-\beta+\frac nq}I_{\gamma \R}(\rho e)^{\frac{1}{\R}}
  \in L^{1}(0,1/2; d\rho/\rho)
\end{equation*}
since by assumption $\beta<n/q$; thus the contribution of this
part to the $L^{r,\infty}(d\rho/\rho)$ norm is finite.

In the region $\rho\ge2$ we have
\begin{equation*}
  I_{\gamma \R}(\rho e)^{\frac{1}{\R}}\simeq 
    \rho^{-\gamma}
  \quad\implies\quad 
  \rho^{-\beta+\frac nq}I_{\gamma \R}(\rho e)^{\frac{1}{\R}} \simeq
  \rho^{-\beta-\gamma+\frac nq} 
  \in L^{1}(2,\infty; d\rho/\rho)
\end{equation*}
since the condition
\begin{equation*}
  -\beta-\gamma+\frac nq<0 \quad\iff \quad
  \alpha<\frac{n}{p'}
\end{equation*}
is satisfied by \eqref{eq:condDL}, and again the
contribution to the $L^{r,\infty}$ norm is finite.

For the third region $1/2\le \rho\le2$, by estimate
\eqref{stimaI}, we see that in the case
$\gamma\R\le n-1$ one has again, for some $\sigma\ge0$,
\begin{equation*}
  I_{\gamma \R}(\rho e)^{\frac{1}{\R}}\simeq 
    |\log||\rho|-1|^{\sigma}
  \quad\implies\quad 
  \rho^{-\beta+\frac nq}I_{\gamma \R}(\rho e)^{\frac{1}{\R}}
  \in L^{1}(1/2,2; d\rho/\rho)
\end{equation*}
On the other hand, in the case $\gamma\R>n-1$ (which includes
the choice $\R=\infty$), we see that
\begin{equation*}
  \rho^{-\beta+\frac nq}I_{\gamma \R}(\rho e)^{\frac{1}{\R}}\simeq
  |\rho-1|^{\frac{n-1}{\R}-\gamma}
  \in L^{r,\infty}(1/2,2; d\rho/\rho)\quad
  \iff \frac{n-1}{\R}-\gamma\ge-\frac1r.
\end{equation*}
Recalling the relation between $q,r,p$ (resp.~$\q,\R,\p$) the last
condition is equivalent to
\begin{equation*}
  -\gamma
  \ge 
  (n-1)
  \left(
    \frac1q-\frac1p-\frac1\q+\frac1\p
  \right)-\frac nq+\frac np-n
\end{equation*}
which is precisely the third of conditions \eqref{eq:condDL}.

\ni The weak Young inequality can be used in \eqref{eq:almostfin}
only in the range $q,r,p\in(1,+\infty)$, which forces
\begin{equation*}
  1<p<q<\infty.
\end{equation*}
To cover the cases
\begin{equation*}
  1\le p<q\le\infty
\end{equation*}
we use instead the strong Young inequality: we can write
\begin{equation}\label{eq:almostfin2}
  \||x|^{-\beta}F*\phi\|_{L^{q}L^{\q}}
  \lesssim
  \|\phi\|_{L^{p}L^{\p}}
  \left\|
    \rho^{-\beta+\frac nq}\||\rho e-\theta|^{-\gamma}\|
       _{L^{\R}_{\theta}}
  \right\|_{L^{r}(\rho^{-1}d\rho)}
\end{equation}
for the full range 
$q,r,p\in[1,+\infty]$. The previous arguments
are still valid apart from the last step which must be replaced by
\begin{equation*}
  \rho^{-\beta+\frac nq}I_{\gamma \R}(\rho e)^{\frac{1}{\R}}\simeq
  |\rho-1|^{\frac{n-1}{\R}-\gamma}
  \in L^{r}(1/2,2; d\rho/\rho)\quad
  \iff \frac{n-1}{\R}-\gamma>-\frac1r
\end{equation*}
and this implies that the inequality in the last condition
\eqref{eq:condDL} must be strict. 

\ni The case
\begin{equation*}
  1<p=q<\infty
\end{equation*}
has already been covered. Indeed, in this case
the scaling condition \eqref{eq:condDL} implies
\begin{equation*}
  \alpha+\beta+\gamma=n
  \quad\implies\quad
  \alpha+\beta>0
\end{equation*}
since $\gamma<n$. Thus when $\p=\q$ the last inequality in
\eqref{eq:condDL} is strict and we can apply the second part of
the proof; the cases $\p\le\q$ follow from the case $\p=\q$.

\ni To complete the proof, it remains to consider the case (ii) where
we assume that the support of the Fourier transform
$\widehat{\phi}$ is contained in an annular region of size $R$.
By scaling invariance of the inequality, it is sufficient to
consider the case $R=1$. Now let $\psi(x)$ be such that
$\widehat{\psi}\in C^{\infty}_{c}$ and precisely
\begin{equation*}
  \widehat{\psi}(\xi)=1 \quad\text{for $c_{1}'\le|\xi|\le c_{2}'$},
  \qquad
  \widehat{\psi}(\xi)=0 \quad\text{for $|\xi|>2c_{1}'$ and 
       $|\xi|<\frac12 c_{2}'$},
\end{equation*}
for some constants $c'_{2}>c_{2}\ge c_{1}>c'_{1}>0$. This implies
\begin{equation*}
  \phi =\mathsf{F}^{-1}(\widehat{\psi}\widehat{\phi})
  =\psi*\phi
\end{equation*}
and we can write
\begin{equation*}
  T_{\gamma}\phi=|x|^{-\gamma}*\psi*\phi=
  (T_{\gamma}\psi)*\phi.
\end{equation*}
Since $T_{\gamma}\psi=c\mathsf{F}^{-1}
(|\xi|^{\gamma-n}\widehat{\psi}(\xi))$ is a Schwartz class function,
we arrive at the estimates
\begin{equation}\label{eq:allN}
  |T_{\gamma}\phi(x)|\le C_{\mu,\gamma}\bra{x}^{-\mu}*|\phi|\qquad
  \forall \mu\ge1.
\end{equation}
Here we can take $\mu$ arbitrarily large. Thus the proof of case
(ii) is concluded by applying the following Lemma:

\begin{lemma}\label{lem:xmu}
  Let $n\ge2$.
  Assume $1\le p\le q\le \infty$,
  $1\le \p\le \q\le \infty$ and $\alpha,\beta, \mu$ satisfy
  \begin{equation}\label{eq:condabmu}
    \beta<\frac nq,\qquad
    \alpha<\frac{n}{p'},\qquad
    \alpha+\beta\ge(n-1)
    \left(\frac1q-\frac1p+\frac1\p-\frac1\q\right),
  \end{equation}
  \begin{equation}\label{eq:condmu}
    \mu>
    -\alpha-\beta+n\left(1+\frac1q-\frac1p\right).
  \end{equation}
  Then the following estimate holds:
  \begin{equation}\label{eq:estmu}
    \||x|^{-\beta}\bra{x}^{-\mu}*\phi\|
       _{L^{q}_{|x|}L^{\q}_{\theta}}\lesssim
    \|\phi\|_{L^{p}_{|x|}L^{\p}_{\theta}}.
  \end{equation}
\end{lemma}

\begin{proof}
  Notice that, by \eqref{eq:condabmu},
  the right hand side in \eqref{eq:condmu} is always strictly
  positive and never larger than 
  $n-1$, thus it is sufficient to prove the lemma for $\mu$
  in the range
  \begin{equation*}
    0<\mu\le n.
  \end{equation*}
  By \eqref{eq:firstest} we have, for all $\p,\q,\R\in[1,+\infty]$ with
  $1+1/\q=1/\R+1/\p$,
  \begin{equation}\label{eq:quant}
    \|\bra{\cdot}^{-\mu}*|\phi|(|x|\theta)\|
       _{L^{\q}_{\theta}(\mathbb{S}^{n-1})}\lesssim
    \int_{0}^{\infty}
       J_{\mu\R}(|x|,\rho)^{\frac{1}{\R}}
    \|\phi(\rho \theta)\|
       _{L^{\p}_{\theta}(\mathbb{S}^{n-1})}\rho^{n-1} d\rho.
  \end{equation}
  Notice that when $\R=\infty$ we have
  \begin{equation*}
    \|\bra{|x|e-\rho \theta}^{-\mu}\|_{L^{\infty}_{\theta}}\lesssim
    \bra{|x|-\rho}^{-\mu}.
  \end{equation*}
  We write for brevity
  \begin{equation*}
    Q(|x|)\equiv
    |x|^{-\beta+\frac{n-1}{q}}
    \|\bra{\cdot}^{-\mu}*|\phi|(|x|\theta)\|
       _{L^{\q}_{\theta}},\qquad
    P(\rho)=\rho^{\alpha+\frac{n-1}{p}}
      \|\phi(\rho \theta)\|
         _{L^{\p}_{\theta}}
  \end{equation*}
  \begin{equation*}
    J(|x|,\rho)=J_{\mu\R}^{\frac{1}{\R}}(x,\rho)
    \quad\text{(resp. $\bra{|x|-\rho}^{-\mu}$ if $\R=\infty$).}
  \end{equation*}
  Thus \eqref{eq:quant} becomes
  \begin{equation}\label{eq:quantt}
    Q(\sigma)\lesssim
    \sigma^{-\beta+\frac{n-1}{q}}
    \int_{0}^{\infty}J(\sigma,\rho)
    \rho^{\frac{n-1}{p'}-\alpha}
    P(\rho)
    d\rho
  \end{equation}
  and the estimate to be proved \eqref{eq:estmu}
  can be written as
  \begin{equation}\label{eq:thesis}
    \|Q\|_{L^{q}(0,+\infty)}\lesssim
    \|P\|_{L^{p}(0,+\infty)}
  \end{equation}
  Recall that the integrals of the form $J(\sigma,\rho)$ have
  been estimated in Lemma \ref{lem:singint2}.
  
\ni  We split $Q$ into the sum of several terms corresponding
  to different regions of $\rho,\sigma$.
  In the region $\sigma\le1$ we have 
  $J(\sigma,\rho)\lesssim \bra{\rho}^{-\mu}$
  so that
  \begin{equation}\label{eq:quant2}
    Q_{1}(\sigma)\lesssim
    \sigma^{-\beta+\frac{n-1}{q}}
    \int_{0}^{\infty}\bra{\rho}^{-\mu}\rho^{\frac{n-1}{p'}-\alpha}
    P(\rho)
    d \rho
  \end{equation}
  Thus we see that in this region \eqref{eq:thesis} follows simply from
  H\"older's inequality and the fact that $\alpha<n/p'$
  and $\beta<n/q$. Similarly, it is easy to handle the part of the
  integral with $\rho\le1$ since we have then
  $J(\sigma,\rho)\lesssim \bra{\sigma}^{-\mu}$. Thus in the
  following we can restrict to $\sigma \gtrsim 1,\rho \gtrsim 1$.

  \ni When $1 \lesssim \sigma \le \rho/2$ we have again 
  $J(\sigma,\rho)\lesssim \bra{\rho}^{-\mu}$ 
  and \eqref{eq:quantt} becomes
  \begin{equation}\label{eq:quant3}
    Q_{2}(\sigma)\lesssim
    \sigma^{-\beta+\frac{n-1}{q}}
    \int_{\sigma}^{\infty}\bra{\rho}^{-\mu}\rho^{\frac{n-1}{p'}-\alpha}
    P(\rho)
    d \rho
  \end{equation}
  If we assume
  \begin{equation}\label{eq:conda}
    \mu>\frac{n}{p'}-\alpha
  \end{equation}
  we can apply H\"older's inequality and we get
  \begin{equation*}
    Q_{3}(\sigma)\lesssim
    \sigma^{-\beta+\frac{n-1}{q}}
    \sigma^{\frac{n}{p'}-\mu-\alpha}\|P\|_{L^{p}}.
  \end{equation*}
  Now the right hand side is in $L^{q}(\sigma\ge1)$ provided
  \begin{equation}\label{eq:condb}
    \mu>\frac{n}{p'}-\alpha+\frac nq-\beta \equiv
    -\alpha-\beta+n\left(1+\frac1q-\frac1p\right)
  \end{equation}
  and we see that \eqref{eq:condb} implies \eqref{eq:conda}
  since $\beta<n/q$ by assumption.

  When $1 \lesssim \rho \le \sigma/2$ we have 
  $J(\sigma,\rho)\lesssim \bra{\sigma}^{-\mu}$ 
  and \eqref{eq:quantt}
  becomes
  \begin{equation}\label{eq:quant3-4}
    Q_{4}(\sigma)\lesssim
    \sigma^{-\beta+\frac{n-1}{q}}\sigma^{-\mu}
    \int_{0}^{\sigma}\rho^{\frac{n-1}{p'}-\alpha}
    P(\rho)
    d \rho
  \end{equation}
  and by H\"older's inequality we have as before
  \begin{equation*}
    \lesssim
    \sigma^{-\beta+\frac{n-1}{q}}
    \sigma^{\frac{n}{p'}-\mu-\alpha}
    \|P\|_{L^{p}}
  \end{equation*}
  so that \eqref{eq:condb} is again sufficient to obtain 
  \eqref{eq:thesis}.

\ni  Finally, let $\sigma \gtrsim 1$, $\rho \gtrsim 1$ and
  $2^{-1}\sigma\le \rho\le 2 \sigma$.
  In this region we must treat differently the values of
  $\mu\R$ larger or smaller than $n-1$, and the case
  $\R=\infty$ is considered at the end.
  Assume that $n-1<\mu\R\le n$; then
  $J(\sigma,\rho)\lesssim \bra{\rho}^{1-n}
  \bra{\sigma-\rho}^{\frac{n-1}{\R}-\mu}$,
  and using the relations
  \begin{equation*}
    \sigma \simeq \rho,\qquad
    \frac1\R=1+\frac1\p-\frac1\q
  \end{equation*}
  we see that \eqref{eq:quantt} reduces to
  \begin{equation}\label{eq:quant4}
    Q_{5}(\sigma)\lesssim
    \sigma^{-\alpha-\beta+(n-1)
      \left(
        \frac1q-\frac1p+\frac1\p-\frac1\q
      \right)}
    \int_{\sigma/2}^{2\sigma}
    \bra{\sigma-\rho}^{\frac{n-1}{\R}-\mu}
    P(\rho)
    d \rho.
  \end{equation}
  The last integral is (bounded by)
  a convolution of $P(\rho)$ with the function
  $\bra{\rho}^{\frac{n-1}{\R}-\mu}$.
  In order to estimate the $L^{q}(\sigma\ge1)$ norm of $Q_{5}$,
  we use first H\"older's then Young's inequality:
  \begin{equation*}
    \|Q_{5}\|_{L^{q}}\lesssim
    \|\bra{\sigma}^{-\epsilon}\|_{L^{q_{0}}}
    \|\bra{\rho}^{\frac{n-1}{\R}-\mu}\|_{L^{q_{1}}}
    \|P\|_{L^{p}}
  \end{equation*}
  where
  \begin{equation*}
    \epsilon=-\alpha-\beta+(n-1)
      \left(
        \frac1q-\frac1p+\frac1\p-\frac1\q
      \right),\qquad
    \frac1q=\frac{1}{q_{0}}+\frac{1}{q_{1}}+\frac1p-1.
  \end{equation*}
  By assumption we have $\epsilon\ge0$.
  When $\epsilon>0$,
  in order for the norms to be finite we need
  \begin{equation*}
    \epsilon q_{0}> 1,\qquad
    \frac{n-1}{\R}-\mu<-\frac{1}{q_{1}}
  \end{equation*}
  which can be rewritten
  \begin{equation*}
    (n-1)\left(
      1+\frac1\q-\frac1\p
    \right)-\mu +1+\frac1q-\frac1p
    < \frac1{q_{0}}
    <\epsilon
  \end{equation*}
  and we see that we can find a suitable $q_{0}$ provided
  the first side is strictly smaller than the last side; this
  condition is precisely equivalent to \eqref{eq:condb} again
  (recall also that $n-1<\mu\le n$).
  The argument works also in the case $\epsilon=0$ by choosing
  $q_{0}=\infty$.

\ni  If on the other hand $0<\mu<n-1$, we have
  $J_{\mu\R}^{\frac1\R}\lesssim \bra{\rho}^{-\mu}$ also
  in this region, so that
  \begin{equation*}
    Q_{5}(\sigma)\lesssim
    \sigma^{-\beta+\frac{n-1}{q}}
    \sigma^{\frac{n-1}{p'}-\alpha-\mu}
    \int_{\sigma/2}^{\sigma}
    P(\rho)
    d \rho
  \end{equation*}
  by $\sigma \simeq \rho$. H\"older's inequality gives
  \begin{equation*}
    Q_{5}(\sigma)\lesssim
    \sigma^{-\beta+\frac{n-1}{q}}
    \sigma^{\frac{n-1}{p'}-\alpha-\mu}
    \sigma^{\frac1{p'}}\|P\|_{L^{p}}
  \end{equation*}
  which leads to exactly the same computations as above and in the
  end to \eqref{eq:condb}. The case $\mu=n-1$ introduces a
  logarithmic term which does not change the integrability properties
  used here.  
  
\ni  It remains the last region when $\R=\infty$ so that
  $J(\sigma,\rho)=\bra{\sigma-\rho}^{-\mu}$
  and $1/\p-1/\q=1$. Then
  \begin{equation*}
    Q_{5}(\sigma)\lesssim 
    \sigma^{-\beta+\frac{n-1}{q}}
    \int_{\sigma/2}^{2 \sigma}\bra{\sigma-\rho}^{-\mu}
    \rho^{\frac{n-1}{p'}-\alpha}
    P(\rho)d \rho
  \end{equation*}
  which is identical with \eqref{eq:quant4} with $\R=\infty$,
  thus the same computations apply and the proof is concluded.
\end{proof}

\section{Weighted Sobolev embeddings}\label{SobEmb}

\noindent In this section we write estimate (\ref{oHLS}) in the form of a Sobolev embedding. In this way we get also critical estimates in Besov Spaces.

\noindent Recalling the pointwise bound
\begin{equation}\label{eq:derivT}
  |u(x)|\le C T_{\lambda}(||D|^{n-\lambda}u|),\qquad
  0<\lambda<n
\end{equation}
where $|D|^{\sigma}=(-\Delta)^{\frac s2}$, we see that an immediate
consequence of \eqref{ourHLS} is the weighted Sobolev
inequality
\begin{equation}\label{eq:weightS}
  \||x|^{-\beta}u\|_{L^{q}L^{\q}}\lesssim
  \||x|^{\alpha}|D|^{\sigma}u\|_{L^{p}L^{\p}}
\end{equation}
provided $1<p\le q<\infty$, $1\le\p\le\q\le\infty$ and
\begin{equation}\label{eq:condDLsob}
\begin{split}
  &\beta<\frac nq,\quad \alpha<\frac{n}{p'},\quad 0<\sigma<n
    \\
  &\alpha+\beta=\sigma+\frac nq-\frac np
    \\
  &\alpha+\beta\ge(n-1)
    \left(\frac1q-\frac1p+\frac{1}{\p}-\frac{1}{\q}\right).
\end{split}
\end{equation}
As usual, if the last condition is strict we can take $p,q$
in the full range $1\le p\le q\le \infty$. For instance, this
implies the inequality
\begin{equation}\label{eq:infsob}
  |x|^{-\beta}|u(x)|\lesssim
  \||x|^{\alpha}|D|^{\sigma}u\|_{L^{p}L^{\p}}
\end{equation}
provided $1\le p\le \infty$ and
\begin{equation*}
  \begin{split}
    &\beta<0,\quad \alpha<\frac{n}{p'},\quad 0<\sigma<n
      \\
    &\alpha+\beta=\sigma-\frac np
      \\
    &\alpha+\beta>(n-1)
      \left(\frac{1}{\p}-\frac1p\right).
  \end{split}
\end{equation*}
If we choose $\alpha=0$ we have in particular
for $p\in(1,\infty)$, $\p\in[1,\infty]$
\begin{equation}\label{eq:partsob}
  |x|^{\frac np-\sigma}|u(x)|\lesssim
  \||D|^{\sigma}u\|_{L^{p}L^{\p}},\qquad
  \frac {n-1}\p+\frac1p<\sigma <\frac np.
\end{equation}
This extends
to the non radial case the radial inequalities in
\cite{Strauss77-a},
\cite{Ni82-a},
\cite{ChoOzawa09-a}
(see also \cite{ChoNakanishi10-a})
and many others; notice that in the radial case we can choose
$\p=\infty$ to obtain the largest possible range.
When $\sigma$ is an integer we can replace
the fractional operator $|D|^{\sigma}$ with usual derivatives;
see Corollary \ref{cor:integers} below for a similar argument.

\noindent By similar techniques it is possible to derive nonhomogeneous
estimates in terms of norms of type $\|\bra{D}^{\sigma}u\|_{L^{p}}$;
we omit the details.


\subsection*{Critical estimates in Besov spaces}\label{sub:besov}  

\noindent Case (ii) in Theorem \ref{the:Our1Thm} is suitable for
applications to spaces defined via Fourier decompositions,
in particular Besov spaces. We recall the standard
machinery:
fix a $C^{\infty}_{c}$ radial function $\psi_{0}(\xi)$
equal to 1 for $|\xi|<1$ and vanishing for $|\xi|>2$,
define a Littlewood-Paley partition of unity via
$\phi_{0}(\xi)=\psi(\xi)-\psi(\xi/2)$,
$\phi_{j}(\xi)=\phi_{0}(2^{-j}\xi)$,
and decompose $u$ as $u=\sum_{j\in \mathbb{Z}}u_{j}$
where $u_{j}=\phi_{j}(D)u=\mathbb{F}^{-1}\phi_{j}(\xi)\mathbb{F}u$.
Then the homogeneous Besov norm $\dot B^{s}_{p,1}$ is defined as
\begin{equation}\label{eq:besovn}
  \|u\|_{\dot B^{s}_{p,1}}=
  \sum_{j\in \mathbb{Z}}2^{js}\|u_{j}\|_{L^{p}}.
\end{equation}
We can apply Theorem \ref{the:Our1Thm}-(ii)
to each component $u_{j}$ 
in the full range of indices $1\le p\le q\le \infty$,
with a constant independent of $j$.
By the standard trick
$\widetilde{u}_{j}=u_{j-1}+u_{j}+u_{j+1}$, 
$u_{j}=\phi_{j}(D)\widetilde{u}_{j}$ we obtain the estimate
\begin{equation}\label{eq:besovest}
  \||x|^{-\beta}T_{\gamma} u\|
      _{L^{q}_{|x|}L^{\q }_{\theta}} 
  \le C
  \sum_{j\in \mathbb{Z}}
  \| |x|^{\alpha} \widetilde{u}_{j}\|_{L^{p}_{|x|}L^{\p }_{\theta}}
\end{equation}
for the full range $1\le p\le q\le \infty$, $1\le \p\le \q\le \infty$,
with $\alpha,\beta,\gamma$ satisfying \eqref{eq:condDLsob}.
The right hand side can be interpreted as a weighted norm of
Besov type with different radial and angular integrability;
this kind of spaces were already considered in
\cite{ChoNakanishi10-a}.
In the special case $\alpha=0$, $p=\p>1$ we obtain a standard
Besov norm \eqref{eq:besovn} and hence the estimate 
(with the optimal choice $\q=\p=p$)
reduces to
\begin{equation}\label{eq:besovtrue}
  \||x|^{-\beta}T_{\gamma} u\|
      _{L^{q}_{|x|}L^{p }_{\theta}} 
  \le C
  \|u\|_{\dot B^{0}_{p,1}}.
\end{equation}
This estimate is weaker than \eqref{ourHLS} when the
third condition in \eqref{eq:condDL} is strict, but
in the case of equality it gives a new estimate:
recalling \eqref{eq:derivT}, we have proved the following

\begin{corollary}\label{cor:besov}
  For all $1< p\le q\le \infty$ we have
  \begin{equation}\label{eq:truesob}
    \||x|^{\frac {n-1}p-\frac {n-1}q} u\|
      _{L^{q}_{|x|}L^{p }_{\theta}} \le C
    \|u\|_{\dot B^{\frac1p-\frac1q}_{p,1}}.
  \end{equation}
\end{corollary}

\noindent If we restrict \eqref{eq:truesob} to radial functions and $q=\infty$, we
obtain the well known radial pointwise estimate
\begin{equation}\label{eq:radialbes}
  |x|^{\frac {n-1}p}|u|\le C
  \|u\|_{\dot B^{1/p}_{p,1}}\qquad
  1< p<\infty
\end{equation}
(see \cite{ChoOzawa09-a}, \cite{SickelSkrzypczak00-a}).

\section{Caffarelli-Kohn-Nirenberg weighted interpolation inequalities}

\ni In this section we use the technology outlined before toghether with interpolation. In this way we can extend to $L^{p}L^{\p}$ setting also the Caffarelli-Kohn-Nirenberg inequalities. Such inequalities are fundamental tools in mathematical analysis and in PDEs theory. In particular are really useful in the context of Navier-Stokes equation because provide a priori estimate for weak solutions by interpolation of quantities related to the energy dissipation. Let's start by the family of inequalities on $\mathbb{R}^{n}$, $n\ge1$
\begin{equation}\label{eq:CKN}
  \||x|^{-\gamma}u\|_{L^{r}}\le C
  \||x|^{-\alpha}\nabla u\|^{a}_{L^{p}}
  \||x|^{-\beta}u\|^{1-a}_{L^{q}}.
\end{equation}
for the range of parameters
\begin{equation}\label{eq:rangeCKN}
  n\ge1,\qquad
  1\le p<\infty,\qquad
  1\le q<\infty,\qquad
  0<r<\infty,\qquad
  0< a\le1.
\end{equation}
Some conditions are immediately seen to be necessary for the validity
of \eqref{eq:CKN}:
to ensure local integrability we need
\begin{equation}\label{eq:intCKN}
  \gamma<\frac nr\qquad
  \alpha<\frac np\qquad
  \beta<\frac nq
\end{equation}
and by scaling invariance we need to assume
\begin{equation}\label{eq:scaCKN}
  \gamma-\frac nr=
  a\left(\alpha+1-\frac np\right)+
  (1-a)\left(\beta-\frac nq\right).
\end{equation}
In \cite{CaffarelliKohnNirenberg84-a} the following remarkable
result was proved, which improves and extends a number of earlier
estimates including weighted Sobolev and Hardy inequalities:

\begin{theorem}[\cite{CaffarelliKohnNirenberg84-a}]\label{the:CKN}
  Consider the inequalities \eqref{eq:CKN} 
  in the range of parameters given by
  \eqref{eq:intCKN}, \eqref{eq:rangeCKN}, \eqref{eq:scaCKN}.
  Denote with $\mathbf\Delta$ the quantity
  \begin{equation}\label{eq:Delta}
    \mathbf\Delta=\gamma-a \alpha-(1-a)\beta \equiv
    a +n
      \left(
        \frac1r-\frac{1-a}{q}-\frac ap
      \right)
  \end{equation}
  (the identity in \eqref{eq:Delta} is a reformulation of
  the scaling relation \eqref{eq:scaCKN}).
  Then the inequalities \eqref{eq:CKN} are true if and only if
  both the following conditions are satisfied:
  \begin{enumerate}\setlength{\itemindent}{-0pt}
  \renewcommand{\labelenumi}{\textit{(\roman{enumi})}}
    \item $\mathbf\Delta\ge0$
    \item $\mathbf\Delta\le a$ when
    $\gamma-n/r=\alpha+1-n/p$.
  \end{enumerate}
\end{theorem}

\begin{remark}\label{rem:original}
  Notice that in the original formulation of 
  \cite{CaffarelliKohnNirenberg84-a} also the case $a=0$
  was considered, but with the introduction of an additional parameter
  forcing $\beta=\gamma$ when $a=0$. Thus the case $a=0$
  becomes trivial
  in the original formulation; however, at least for $r>1$,
  a much larger range $0\le \gamma-\beta<n$ can be obtained
  by a direct application of the Hardy-Littlewood-Sobolev inequality,
  so strictly speaking the additional requirement $\beta=\gamma$ is not
  necessary. We think the formulation adopted here is cleaner.
  
  \noindent On the other hand, the necessity of (i) follows from
  the uniformity of the estimate w.r.to translations, while the
  necessity of (ii) is proved by testing the inequality on
  the spikes $|x|^{\gamma-n/r}\log|x|^{-1}$ truncated near $x=0$.
\end{remark}

\noindent In \cite{DenapoliDrelichmanDuran09-a} 
the authors prove the
following radial improvement
of Theorem \ref{the:CKN}:

\begin{theorem}[\cite{DenapoliDrelichmanDuran09-a}]
\label{the:CKNDDD}
  Let $n\ge2$,
  let $\alpha,\beta,\gamma,r,p,q,a$ be in the range determined by
  \eqref{eq:intCKN}, \eqref{eq:rangeCKN}, \eqref{eq:scaCKN},
  define $\mathbf\Delta$ as in \eqref{eq:Delta}, and assume that
  \begin{equation}\label{eq:CKNDDD}
    a\left(1-\frac np\right)\le \mathbf\Delta\le a,\qquad
    \alpha<\frac np-1,
  \end{equation}
  the first inequality being strict
  when $p=1$. Then estimate \eqref{eq:CKN} is true for
  all radial functions $u\in C^{\infty}_{c}(\mathbb{R}^{n})$.
\end{theorem}

\noindent We somewhat simplified the statement of Theorem 1.1 in
\cite{DenapoliDrelichmanDuran09-a}, and in particular conditions
(1.8)-(1.10) in that paper are equivalent to \eqref{eq:CKNDDD} 
here, as it is
readily seen. Notice that 
the condition $\mathbf\Delta\le a$ forces 
$r$ to be larger than 1.

\noindent Using the $L^{p}L^{\p}$ norms we can extend both
Theorems \ref{the:CKN} and \ref{the:CKNDDD}. For greater
generality we prove an estimate with fractional derivatives
\begin{equation*}
  |D|^{\sigma}=(-\Delta)^{\frac \sigma2},\qquad \sigma>0.
\end{equation*}
Our result is the following:

\begin{theorem}\label{the:Our2Thm}
  Let $n\ge2$, $r,\R,p,\p,q,\q\in[1,+\infty)$, $0<a\le1$,
  $0<\sigma<n$ with
  \begin{equation}\label{eq:Ico}
    \gamma<\frac nr,\qquad
    \beta<\frac nq,\qquad
    \frac np-n<\alpha<\frac np-\sigma
  \end{equation}
  satisfying the scaling condition
  \begin{equation}\label{eq:scalingCKN}
    \gamma-\frac nr=
    a\left(\alpha+\sigma-\frac nr\right)+
    (1-a)\left(\beta-\frac nq\right).
  \end{equation}
  Define the quantities
  \begin{equation}\label{eq:Deltas}
    \mathbf\Delta=a \sigma+n
      \left(
        \frac1r-\frac{1-a}{q}-\frac ap
      \right),\qquad
    \widetilde{\mathbf\Delta}=a \sigma+n
      \left(
        \frac1\R-\frac{1-a}{\q}-\frac a\p
      \right).
  \end{equation}
  and assume further that
  \begin{equation}\label{eq:IIco}
    \mathbf\Delta+(n-1)\widetilde{\mathbf\Delta}\ge0,
  \end{equation}
  \begin{equation}\label{eq:IIIco}
    1<p,\qquad
    a\left(\sigma-\frac np\right)<\mathbf\Delta\le a \sigma, \qquad 
    a\left(\sigma-\frac n\p\right)\le \widetilde{\mathbf\Delta}\le a \sigma.
  \end{equation}
  Then the following interpolation inequality holds:
  \begin{equation}\label{eq:CKNnostra}
    \||x|^{-\gamma}u\|_{L^{r}_{|x|}L^{\R}_{\theta}}\le C
    \||x|^{-\alpha}|D|^{\sigma} u\|^{a}_{L^{p}_{|x|}L^{\p}_{\theta}}
    \||x|^{-\beta}u\|^{1-a}_{L^{q}_{|x|}L^{\q}_{\theta}}.
  \end{equation}
  If one assumes strict inequality in \eqref{eq:IIco}, then
  the inequalities in \eqref{eq:IIIco} can be relaxed to 
  non strict inequalities.
\end{theorem}

\noindent When $\sigma$ is an integer, the condition on $\alpha$
from below can be dropped, and a slightly stronger estimate 
can be proved.
We introduce the notation
\begin{equation*}
  \||x|^{-\alpha}D^{\sigma} u\|_{L^{p}L^{\p}}=
  \sum_{|\nu|=\sigma}
  \||x|^{-\alpha}D^{\nu} u\|_{L^{p}L^{\p}},\qquad
  \nu=(\nu_{1},\dots,\nu_{n})\in \mathbb{N}^{n}.
\end{equation*}
Then we have:

\begin{corollary}\label{cor:integers}
  Assume $\sigma=1,\dots,n-1$ is an integer. Then
  the following estimate holds
  \begin{equation}\label{eq:CKNnostraint}
    \||x|^{-\gamma}u\|_{L^{r}_{|x|}L^{\R}_{\theta}}\le C
    \||x|^{-\alpha}D^{\sigma} u\|^{a}_{L^{p}_{|x|}L^{\p}_{\theta}}
    \||x|^{-\beta}u\|^{1-a}_{L^{q}_{|x|}L^{\q}_{\theta}}.
  \end{equation}
  provided the parameters satisfy the same conditions as in the
  previous theorem, with the exception of the condition
  $\alpha>-n+n/p$
  which is not necessary.
\end{corollary}

\begin{remark}\label{rem:comparCKN}
  If $\sigma=1$, Corollary \ref{cor:integers} contains both the
  original result of \cite{CaffarelliKohnNirenberg84-a}
  (for $\mathbf\Delta\le a$)
  and the radial improvement of \cite{DenapoliDrelichmanDuran09-a}.
  
  \noindent Indeed,
  if we choose $p=\p$, $q=\q$, $r=\R$ in Corollary \ref{cor:integers}
  we get of course $\mathbf\Delta=\widetilde{\mathbf\Delta}$, and selecting
  $\sigma=1$ we reobtain the original inequality 
  \eqref{eq:CKN} in the range $0\le \mathbf\Delta\le a$.
  
  \noindent On the other hand, if $u$ is a radial function, estimate 
  \eqref{eq:CKNnostraint} does not
  depend on the choice of $\p,\q,\R$ and we can let $\widetilde{\mathbf\Delta}$
  assume an arbitrary value in the range \eqref{eq:IIIco}.
  Thus if $\mathbf\Delta>a(\sigma-n/p)$ we can choose $\widetilde{\mathbf\Delta}=0$,
  while if $\mathbf\Delta=a(\sigma-n/p)$ we can choose
  $\widetilde{\mathbf\Delta}=\epsilon>0$ arbitrarily small, recovering
  the results of Theorem \ref{the:CKNDDD}.
\end{remark}  

\noindent Again we can get theorem \ref{the:Our2Thm} as a consequence of \ref{the:Our1Thm}.

\begin{proof}
 We begin by taking
$0<a\le1$, and indices $r,\R,s,\s,q,\q\in[1,+\infty]$ such that
\begin{equation}\label{eq:rsq}
  \frac1r=\frac as+\frac{1-a}q,\qquad
  \frac1\R=\frac a\s+\frac{1-a}\q.
\end{equation}
Then by two applications of H\"older's inequality we obtain
the interpolation inequality
\begin{equation}\label{eq:holder}
\begin{split}
  \||x|^{-\gamma}u\|_{L^{r}L^{\R}}
  =&
  \|(|x|^{-\delta}u)^{a}(|x|^{-\beta}u)^{1-a}\|_{L^{r}L^{\R}}
    \\
  \le&
  \|(|x|^{-\delta}u)^{a}\|_{L^{s/a}L^{\s/a}}
  \|(|x|^{-\beta}u)^{1-a}\|_{L^{q/(1-a)}L^{\q/(1-a)}}
    \\
  =&
  \||x|^{-\delta}u\|^{a}_{L^{s}L^{\s}}
  \||x|^{-\beta}u\|^{1-a}_{L^{q}L^{\q}}
\end{split}
\end{equation}
provided the exponents $\gamma,\delta,\beta$ are related by
\begin{equation}\label{eq:expon}
  \gamma=a \delta+(1-a)\beta.
\end{equation}
Now the main step of the proof. By Theorem \eqref{the:Our1Thm}
we know that
\begin{equation*}
  \||x|^{-\delta}T_{\lambda}u\|_{L^{s}L^{\s}}\lesssim
  \||x|^{-\alpha}u\|_{L^{p}L^{\p}}
\end{equation*}
under suitable conditions on the indices. 
Now using the well known estimate
\begin{equation}\label{eq:fractest}
  |u(x)|\le C_{\lambda,n}
    T_{\lambda}\left(
      \left|
        |D|^{n-\lambda}u
      \right|
    \right)
\end{equation}
the previous inequality can be equivalently written
\begin{equation*}
  \||x|^{-\delta}u\|_{L^{s}L^{\s}}\lesssim
  \||x|^{-\alpha}|D|^{\sigma}u\|_{L^{p}L^{\p}},\qquad
  \sigma=n-\lambda
\end{equation*}
which together with \eqref{eq:holder} gives
\begin{equation}\label{eq:final}
  \||x|^{-\gamma}u\|_{L^{r}L^{\R}}\lesssim
  \||x|^{-\alpha}|D|^{\sigma}u\|_{L^{p}L^{\p}}^{a}
  \||x|^{-\beta}u\|^{1-a}_{L^{q}L^{\q}}.
\end{equation}
The conditions on the indices are those given by \eqref{eq:rsq},
\eqref{eq:expon}, plus those listed in the statement of
Theorem \eqref{the:Our1Thm} (notice that we are using $-\alpha$
instead of $\alpha$). The complete list is the following:
\begin{equation}\label{eq:tot1}
  r,s,q,\R,\s,\q\in[1,+\infty],\qquad 
  a<0\le1,\qquad
  0<\sigma<n,
\end{equation}
\begin{equation}\label{eq:tot2}
  \frac1r=\frac as+\frac{1-a}q,\qquad
  \frac1\R=\frac a\s+\frac{1-a}\q.
\end{equation}
\begin{equation}\label{eq:tot3}
  1<s\le p<\infty,\qquad
  1\le\s\le\p\le \infty,
\end{equation}
\begin{equation}\label{eq:tot4}
  \gamma<\frac nr,\qquad  \beta<\frac nq,\qquad
  -\alpha<\frac n{p'},\qquad
  \delta<\frac ns,
\end{equation}
\begin{equation}\label{eq:tot5}
  \gamma=a \delta+(1-a)\beta,
\end{equation}
\begin{equation}\label{eq:tot6}
  -\alpha+\delta+n-\sigma=n+\frac ns-\frac np,
\end{equation}
\begin{equation}\label{eq:tot7}
  -\alpha+\delta\ge
  (n-1)\left(\frac1s-\frac1p+\frac1\p-\frac1\s\right).
\end{equation}
Recall also that, when the last inequality \eqref{eq:tot7}
is strict, we can allow the full range
\begin{equation*}
  1\le s\le p\le \infty.
\end{equation*}
Our final task is to rewrite this set of conditions in a
compact form, eliminating the redundant parameters
$\delta,s,\s$. Define the two quantities
\begin{equation*}
  \mathbf\Delta=a \sigma+n
    \left(
      \frac1r-\frac{1-a}{q}-\frac ap
    \right),\qquad
  \widetilde{\mathbf\Delta}=a \sigma+n
    \left(
      \frac1\R-\frac{1-a}{\q}-\frac a\p
    \right).
\end{equation*}
Then \eqref{eq:tot2} are equivalent to
\begin{equation}\label{eq:tot2b}
  \mathbf\Delta=a \left(\sigma+\frac ns-\frac np\right),\qquad
  \widetilde{\mathbf\Delta}=a \left(\sigma+\frac n\s-\frac n\p\right)
\end{equation}
while \eqref{eq:tot6} is equivalent to
\begin{equation}\label{eq:tot6b}
  \delta=\alpha+\frac{\mathbf\Delta}{a}
\end{equation}
and we can use \eqref{eq:tot2b}, \eqref{eq:tot6b} to replace
$\delta,s,\s$ in the remaining relations. Condition \eqref{eq:tot5}
becomes
\begin{equation}\label{eq:tot5b}
  \mathbf\Delta=\gamma-a \alpha-(1-a)\beta,
\end{equation}
which is precisely the scaling condition,
while \eqref{eq:tot7} becomes
\begin{equation}\label{eq:tot7b}
  \mathbf\Delta+(n-1)\widetilde{\mathbf\Delta}\ge0.
\end{equation}
The last inequality in \eqref{eq:tot4}, $\delta<n/s$, can be written
\begin{equation*}%
  \alpha<\frac np-\sigma
\end{equation*}
so that \eqref{eq:tot4} is replaced by
\begin{equation}\label{eq:tot4b}
  \gamma<\frac nr,\qquad
  \beta<\frac nq,\qquad
  \frac np-n<\alpha<\frac np-\sigma.
\end{equation}
Finally, conditions \eqref{eq:tot3} translate to
\begin{equation}\label{eq:tot3b}
  1<p,\qquad
  a\left(\sigma-\frac np\right)<\mathbf\Delta\le a \sigma, \qquad 
  a\left(\sigma-\frac n\p\right)\le \widetilde{\mathbf\Delta}\le a \sigma.
\end{equation}
When the inequality in \eqref{eq:tot7b} is strict, the last
condition can be relaxed to
\begin{equation}\label{eq:tot3b2}
  1\le p,\qquad
  a\left(\sigma-\frac np\right)\le\mathbf\Delta\le a \sigma, \qquad 
  a\left(\sigma-\frac n\p\right)\le \widetilde{\mathbf\Delta}\le a \sigma.
\end{equation}

\ni We pass now to the proof of Corollary \ref{cor:integers}.
Assume now $\sigma$ is integer, and the inequality
\begin{equation*}
  \||x|^{-\gamma}u\|_{L^{r}L^{\R}}\le C
  \||x|^{-\alpha}|D|^{\sigma} u\|^{a}_{L^{p}L^{\p}}
  \||x|^{-\beta}u\|^{1-a}_{L^{q}L^{\q}}
\end{equation*}
is true for a certain choice of the parameters as in the theorem,
so that in particular
\begin{equation*}
  \alpha< \frac np-\sigma<\frac np.
\end{equation*}
Then we shall prove that also the following inequalities are true
\begin{equation}\label{eq:goalk}
  \||x|^{k-\gamma}u\|_{L^{r}L^{\R}}\le C
  \||x|^{k-\alpha}D^{\sigma} u\|^{a}_{L^{p}L^{\p}}
  \||x|^{k-\beta}u\|^{1-a}_{L^{q}L^{\q}}
\end{equation}
for all integers $k\ge0$, where we are using the
shorthand notation
\begin{equation*}
  \||x|^{k-\alpha}D^{\sigma} u\|_{L^{p}L^{\p}}=
  \sum_{|\nu|=\sigma}
  \||x|^{k-\alpha}D^{\nu} u\|_{L^{p}L^{\p}},\qquad
  (\nu=(\nu_{1},\dots,\nu_{n})\in \mathbb{N}^{n}).
\end{equation*}
This in particular implies that the condition on $\alpha$ from
below can be dropped when $\sigma$ is integer.

\ni When $k=0$, \eqref{eq:goalk} is obtained just by
replacing $|D|^{\sigma}$ with $D^{\sigma}$ in the 
original inequality. 
The proof of this estimate is identical to the previous one;
the only modification is to use, instead of \eqref{eq:fractest},
the stronger pointwise bound
\begin{equation}\label{eq:fractest2}
  |u(x)|\le C_{\lambda,n}
    T_{\lambda}\left(
      \left|
        D^{n-\lambda}u
      \right|
    \right)
\end{equation}
which is valid for all $\lambda=1,\dots,n-1$.

\ni Now if we apply \eqref{eq:goalk} (with $k=0$) to a function of the
form $|x|^{k}u$ for some $k\ge1$, we obtain
\begin{equation*}
  \||x|^{k-\gamma}u\|_{L^{r}L^{\R}}\le C
  \||x|^{-\alpha}D^{\sigma}(|x|^{k}u)\|^{a}_{L^{p}L^{\p}}
  \||x|^{k-\beta}u\|^{1-a}_{L^{q}L^{\q}}
\end{equation*}
and to conclude the proof we see that it is sufficient to prove the
inequality
\begin{equation}\label{eq:interm}
  \||x|^{-\alpha}D^{\sigma}(|x|^{k}u)\|_{L^{p}L^{\p}}\lesssim
  \||x|^{k-\alpha}D^{\sigma} u\|_{L^{p}L^{\p}}
\end{equation}
for all $\alpha<n/p$, $1\le p,\p<\infty$,
and integers $\sigma=1,\dots,n-1$, $k\ge1$. Notice indeed that all
the conditions on the parameters 
(apart from $\alpha>-n+n/p$)
are unchanged if we decrease
$\gamma,\alpha,\beta$ by the same quantity.

\ni By induction on $k$ (and writing $\delta=-\alpha$),
we are reduced to prove that
for all $p,\p\in[1,\infty)$ and $1\le\sigma\le n-1$
\begin{equation}\label{eq:interm2}
  \||x|^{\delta}D^{\sigma}(|x|u)\|_{L^{p}L^{\p}}\lesssim
  \||x|^{1+\delta}D^{\sigma} u\|_{L^{p}L^{\p}},\qquad
  \delta>\sigma-\frac np.
\end{equation}
Using Leibnitz' rule we reduce further to
\begin{equation}\label{eq:interm3}
  \||x|^{1+\delta-\ell}u\|_{L^{p}L^{\p}}\lesssim
  \||x|^{1+\delta}D^{\ell} u\|_{L^{p}L^{\p}},\qquad
  \delta>\ell-\frac np
\end{equation}
for $\ell=1,\dots,n-1$, and by induction on $\ell$
this is implied by
\begin{equation}\label{eq:interm4}
  \||x|^{\delta}u\|_{L^{p}L^{\p}}\lesssim
  \||x|^{1+\delta}\nabla u\|_{L^{p}L^{\p}},\qquad
  \delta>1-\frac np.
\end{equation}
In order to prove \eqref{eq:interm4}, consider first
the radial case.
When $u=\phi(|x|)$ is a radial (smooth compactly supported)
function, we have
\begin{equation*}
  \||x|^{\delta}u\|_{L^{p}L^{\p}}^{p}\simeq
  \int_{0}^{\infty}\rho^{\delta p+n-1}|\phi(\rho)|^{p}d\rho.
\end{equation*}
Integrating by parts we get
\begin{equation*}
\begin{split}
  =&-\frac{p}{\delta p+n}\int_{0}^{\infty}
    \rho^{\delta p+n}|\phi|^{p-1}|\phi(\rho)|'d\rho
    \\
  \lesssim &
  \int_{0}^{\infty}(\rho^{\delta p+n-1}|\phi|^{p})^{\frac{p-1}{p}}
  (\rho^{\delta p+p+n-1}|\phi'|^{p})^{\frac1p}d\rho
    \\
  \simeq &
  \||x|^{\delta}u\|^{\frac{p-1}{p}}_{L^{p}L^{\p}}
  \||x|^{1+\delta}\nabla u\|_{L^{p}L^{\p}}
\end{split}
\end{equation*}
which implies \eqref{eq:interm4} in the radial case. If $u$ is not
radial, define
\begin{equation*}
  \phi(\rho)=\|u(\rho \theta)\|_{L^{\p}_{\theta}(\mathbb{S}^{n-1})}
  =
  \left(
    \int_{\mathbb{S}^{n-1}}
    |u(\rho \theta)|^{\p}dS_{\theta}
  \right)^{\frac1\p}
\end{equation*}
so that
\begin{equation*}
  \||x|^{\delta}u\|_{L^{p}L^{\p}}\simeq
  \left(\int_{0}^{\infty}\rho^{\delta p+n-1}|\phi(\rho)|^{p}d\rho
  \right)
  ^{\frac1p}.
\end{equation*}
The proof in the radial case implies
\begin{equation*}
  \||x|^{\delta}u\|_{L^{p}L^{\p}}\le
  \||x|^{\delta+1}\phi'(|x|)\|_{L^{p}};
\end{equation*}
moreover we have
\begin{equation*}
\begin{split}
  |\phi'(\rho)|\lesssim &\ 
  \phi^{1-\p}
  \int_{\mathbb{S}^{n-1}}
  |u(\rho \theta)|^{\p-1}|\theta \cdot \nabla u|\ 
  dS_{\theta}
    \\
  \le &\ 
  \phi^{1-\p}
  \left(\int_{\mathbb{S}}|u|^{\p}\right)^{\frac{\p-1}\p}
  \left(\int_{\mathbb{S}}|\nabla u|^{\p}\right)^{\frac1\p}
  =\|\nabla u(\rho \theta)\|_{L^{\p}_{\theta}(\mathbb{S}^{n-1})}
\end{split}
\end{equation*}
and in conclusion we obtain
\begin{equation*}
  \||x|^{\delta}u\|_{L^{p}L^{\p}}\le
  \||x|^{\delta+1}\nabla u\|_{L^{p}L^{\p}}
\end{equation*}
as claimed.
\end{proof}

\section{Strichartz estimates for the wave equation}

\ni As a last example, we
mention an application of our result to Strichartz
estimates for the wave equation; a more detailed
analysis will be conducted elsewhere.
The wave flow $e^{it|D|}$ on $\mathbb{R}^{n}$, $n\ge2$,
satisfies the estimates,
which are usually called \emph{Strichartz estimates}:
\begin{equation}\label{eq:strich}
  \||D|^{\frac nr+\frac1p-\frac n2}e^{it|D|}f\|_{L^{p}_{t}L^{r}_{x}}
  \lesssim \|f\|_{L^{2}}
\end{equation}
provided the indices $p,r$ satisfy
\begin{equation}\label{eq:admWE}
  p\in[2,\infty],\qquad
  0<\frac1r\le\frac12-\frac{2}{(n-1)p}.
\end{equation}
Here the $L^{p}_{t}L^{r}_{x}$ norms are defined as
\begin{equation*}
  \|u(t,x)\|_{L^{p}_{t}L^{r}_{x}}=
  \left\|
     \|u(t,\cdot)\|_{L^{r}_{x}}
  \right\|_{L^{p}_{t}}.
\end{equation*}
In their most general version, the estimates were proved in 
\cite{GinibreVelo95-b},
\cite{KeelTao98-a}. Notice that in \eqref{eq:strich} we included
the extension of the estimates which can be obtained via
Sobolev embedding on $\mathbb{R}^{n}$. 

\ni If the initial value $f$ is a radial function, the estimates admit
an improvement in the sense that conditions \eqref{eq:admWE}
can be relaxed to
\begin{equation}\label{eq:admradial}
  p\in[2,\infty],\qquad
  0<\frac1r<\frac12-\frac{1}{(n-1)p}.
\end{equation}
This phenomenon is connected with the finite speed of
propagation for the wave equation and is usually deduced 
using the space-time decay properties of the equation.
For a thourough discussion and a comprehensive history
of such estimates see 
e.g.~\cite{JiangWangYu10-a} and the references therein.

\ni A different set of estimates are the \emph{smoothing estimates},
also known as Morawetz-type or weak dispersion estimates.
These appear in a large number of versions;
a particularly sharp one is the following, from
\cite{FangWang08-a}:
\begin{equation}\label{eq:smooWE}
  \||x|^{-\zeta}|D|^{\frac12-\zeta}
     e^{it|D|}f\|_{L^{2}_{t}L^{2}_{x}}
  \lesssim \|\Lambda^{\frac12-\zeta} f\|_{L^{2}},\qquad
  \frac12<\zeta <\frac n2.
\end{equation}
Here the operator
\begin{equation*}
  \Lambda=(1-\Delta_{\mathbb{S}^{n-1}})^{1/2}
\end{equation*}
is a function of the Laplace-Beltrami operator on the
sphere and acts only on angular variables, thus we see
that the flow improves the angular regularity.
Morawetz-type estimates are conceptually simpler than
\eqref{eq:strich}, being related to more basic properties
of the operators; indeed $L^{2}$ estimates of this type
can be proved 
for quite large classes of equations
via multiplier methods.

\ni Corresponding estimates are known for the Schr\"odinger flow
$e^{it \Delta}$, and
M.C.~Vilela \cite{Vilela01-a} noticed that in the radial case
they can be used to
deduce Strichartz estimates via the radial Sobolev embedding. 
Following a similar idea for the wave flow, 
in combination with our precised 
estimates \eqref{eq:weightS}, gives an even better result,
which strengthens the standard Strichartz estimates
\eqref{eq:strich}-\eqref{eq:admWE} in
terms of the mixed $L_{|x|}^{p}L_{\theta}^{\p}$ norms.
Indeed, a special case of \eqref{eq:weightS} gives, 
for arbitrary functions $g(x)$,
\begin{equation}\label{eq:special}
  \|g\|_{L^{q}_{|x|}L^{\q}_{\theta}}\lesssim
  \||x|^{\alpha}|D|^{\alpha+\frac n2-\frac nq}g\|_{L^{2}},\qquad
  q,\q\in[2,\infty),\qquad
  \frac n2>\alpha\ge(n-1)\left(\frac1q-\frac1\q\right)
\end{equation}
with the exclusion of the case $\alpha=0$, $q=\q=2$. Then by
\eqref{eq:special} and \eqref{eq:smooWE}
we obtain the precised Strichartz estimates
\begin{equation}\label{eq:prestrich}
  \||x|^{-\delta}|D|^{\frac nq+\frac12-\frac n2-\delta}e^{it|D|}f\|
      _{L^{2}_{t}L^{q}_{|x|}L^{\q}_{\theta}}
  \lesssim \|\Lambda^{-\epsilon} f\|_{L^{2}}
\end{equation}
provided
\begin{equation}\label{eq:condpre}
  q,\q\in[2,+\infty),\qquad
  \delta<\frac nq,\qquad
  0<\epsilon<\frac{n-1}{2},\qquad
  0<\frac1q<\frac1\q-\frac{1}{2(n-1)}
\end{equation}
and
\begin{equation}\label{eq:condpre2}
  \epsilon\le \delta+(n-1)\left(\frac1\q-\frac{1}{2(n-1)}-\frac1q\right).
\end{equation}

\ni We will not exploit the consequence of this particular section in the thesis. The results contained can be consider just as an application of our point of view to a different class of problems. We hope to came back on Strichartz estimates in $L^{p}L^{\p}$ spaces in future works.

\chapter{Introduction to the regularity problem for the Navier-Stokes equation}\label{chap2}

\ni In this chapter we introduce the Cauchy problem for the Navier-Stokes approximation of the fluid motion in the whole space, that's
\begin{equation}\label{CauchyNS}
\left \{
\begin{array}{rcccl}
\partial_{t}u + (u \cdot \nabla) u +\nabla p -\Delta u & = & 0 & \quad \mbox{in}& \quad [0,T)\times \mathbb{R}^{n}  \\
\nabla \cdot u & = & 0 & \quad \mbox{in} & \quad [0,T)\times \mathbb{R}^{n} \\
u & = & u_{0} & \quad \mbox{in}& \quad \{0\} \times \mathbb{R}^{n}. 
\end{array}\right.
\end{equation}
Here $u = (u_{1},u_{2},u_{3})$ is the velocity field, $p$ is the pressure and the viscosity have been set to one. No external forces is working.
The first equation is the Newton law while the second guarantees the incompressibility of the fluid. To require incompressibility also at time $t=0$ have to be considered just initial data $u_{0}$ such that $\nabla \cdot u_{0}=0$. So is useful to define the space
\begin{equation}
L^{2}_{\sigma}(\mathbb{R}^{n}) = \overline{ \left\{ u_{0} \in C^{\infty}_{0}(\mathbb{R}^{n})\quad \mbox{s.t.} \quad \int_{\mathbb{R}^{n}} |u_{0}|^{2} \ dx < +\infty, \ \nabla \cdot u_{0} =0 \right\} }. 
\end{equation}
We will use the same notation for the norm of scalar, vector or tensor quantities, the meaning will be clear by the situation; for instance we will use $\|p\|_{L^{2}(\mathbb{R}^{n})}=\int_{\mathbb{R}^{n}}|p| \ dx$, $\|u\|_{L^{2}(\mathbb{R}^{n})} = \int_{\mathbb{R}^{n}}\sum_{i=1}^{3}u_{i}^{2} \ dx$, $\|\nabla u\|_{L^{2}(\mathbb{R}^{n})} = \int_{\mathbb{R}^{n}}\sum_{i,j=1}^{3}(\partial_{i}u_{j})^{2} \ dx$. We will use also the notation $u \in L^{2}(\mathbb{R}^{n})$ instead of $u \in ({L^{2}(\mathbb{R}^{n})})^{3}$, and so on.
The well posedness of (\ref{CauchyNS}) is a well-known mathematical challenge and just partial results have been obtained. The main question is: if we consider an initial datum $u_{0}$ in the Schwartz class, there exists a unique global solution of the problem (\ref{CauchyNS})?

\ni In this chapter we will give an excursus on classical theorems starting by the pioneering work of Leray, Hopf, Serrin and Kato \cite{Ler,Hopf,Ser}. This classical results have been improved in many different directions and in the thesis we will focus basically on the weighted norm approach, which also has a wide reference literature, see \cite{YongZhou}, \cite{Kukavica}, et al. We will also briefly focus on \cite{Tat,CKN}. The first is particulary relevant because seems to be a sharp version for the existence with small data. The second is a celebrated landmark for the regularity theory.

\section{Equivalence between the differential and integral formulation}

\ni In this secion we give the integral formulation of problem (\ref{CauchyNS}). This formulation is very useful in order to study both local (in time) well posedness and global well posedness with small data. In such a case, starting by the integral formulation is immediate to look at the equation (\ref{CauchyNS}) as a perturbed heat equation, and fixed point techniques are available. Of course this is irrelevant when we look for global solutions with large initial data. We follow basically \cite{Lem}, and we omit, almost completely, the proofs to get a compact presentation. Anyway all the proofs are classical and can be easily found in literature.
Let come back on the system:
$$
\left \{
\begin{array}{rcl}
\partial_{t}u + (u \cdot \nabla) u +\nabla p  & = \Delta u & 0   \\
\nabla \cdot u & = & 0  \\
u & = & u_{0}. 
\end{array}\right.
$$
or in components $(i=1,...,n)$:
$$
\left \{
\begin{array}{rcl}
\partial_{t}u_{i} + \sum_{j=1}^{n}u_{j} \partial_{j} u_{i} +\partial_{i} p  & = &  \sum_{j=1}^{n}\partial_{jj}u_{i}   \\
\sum_{i=1}^{n}\partial_{i} u_{i} & = & 0.
\end{array}\right.
$$
By taking the divergence of the first equation and using $\nabla \cdot u =0$ we get:
\begin{eqnarray}\label{PressureRie}
0&=&  \sum_{i=1}^{n} \partial_{i} \sum_{j=1}^{n}u_{j}\partial_{j}u_{i} + \Delta p \\ 
&=& \sum_{i,j=1}^{n}\partial_{i}\partial_{j}(u_{i}u_{j}) + \Delta p,
\end{eqnarray}
so $p$ can, at least formally, be recovered by $u$ througt:
\begin{equation}
p= - \frac{\sum_{i,j=1}^{n}\partial_{i}\partial_{j}(u_{i}u_{j})}{\Delta}.
\end{equation} 
By using his relation the system can be written as:
\begin{equation}\label{CauchyNSInt}
\left \{
\begin{array}{rclcl}
 u & = & e^{t\Delta}u_{0} + \int_{0}^{t}e^{(t-s)\Delta}\mathbb{P} \nabla \cdot (u \otimes u) ds &\quad \mbox{in} \quad & [0,T)\times \mathbb{R}^{n}  \\
\nabla \cdot u & = & 0 &\quad \mbox{in} \quad & [0,T)\times \mathbb{R}^{n}. 
\end{array}\right.
\end{equation}
where $(u\otimes u)_{i,j} = u_{i}u_{j}$ and $\mathbb{P}$ is formally:
\begin{equation}\label{LerProj}
\mathbb{P}f = f - \nabla \frac{1}{\Delta}\nabla \cdot f,
\end{equation}
or in components
$$
(\mathbb{P} f)_{i} = f_{i} - \frac{ \partial_{i}\partial_{j}}{\Delta} f_{j}.
$$
The operator $\mathbb{P}$, that's a really useful tool in the study of the Navier-Stokes problem, has been introduced by Leray in \cite{Ler}. It is a projection operator on the subspace of the divergence free vector fields. It is infact easy to show that $\mathbb{P}f=f \Leftrightarrow \nabla \cdot f =0$.
In order to give precise definition of the formal computation above at first we have to make sense to the operator $\mathbb{P}$. This is easy if we restrict to $f \in L^{2}(\mathbb{R}^{n})$, in this case $\mathbb{P}f$ is defind by:
$$
\mathbb{P} f = f - (R \otimes R) f
$$
or in components:
$$
(\mathbb{P}f)_{i} = f_{i} - \sum_{j}R_{i}R_{j}f_{j},
$$
where $R_{j}$ is the Riesz transform in the direction $j$ defined by the symbol $i\frac{\xi_{j}}{|\xi|}$. This is a simple way to define $\mathbb{P}$, even if it can be defined on larger Banach spaces as a Calderon-Zygmund operator; details can be found in \cite{Lem}. Anyway we are interested basically in the operator $\mathbb{P}(\nabla \cdot f)$, that appears in the integral formulation (\ref{CauchyNSInt}); it is defined through components by:
$$
(\mathbb{P}\nabla \cdot (u \otimes u))_{i} = \partial_{j} (u \otimes u)_{i,j}
- \frac{1}{\Delta} \partial_{i}\partial_{j}\partial_{k} (u \otimes u)_{j,k}.
$$
In such a case the differentiation allows to extend the definition to a larger class of Banach spaces. To give a precise definition (again in \cite{Lem}) we need the auxiliary spaces:

\begin{definition}

Let define the dual spaces $WL^{\infty}(\mathbb{R}^{n}), L^{1}_{uloc}(\mathbb{R}^{n})$:
\begin{itemize}
\item the space $WL^{\infty}(\mathbb{R}^{n})$ is the Banach space of the Lebesgue measurable functions $\phi$ on $\mathbb{R}^{n}$ such that
\begin{equation}\label{aaa}
\sum_{k \in \mathbb{Z}^{n}} \sup_{x \in \{ k+[0,1]^{n} \}} | \phi(x)| < +\infty, 
\end{equation}
equipped with the norm (\ref{aaa});
\item the space $L^{1}_{uloc}(\mathbb{R}^{n})$ is the space of locally intgrable functions equipped with the norm:
$$
L^{1}_{uloc}(\mathbb{R}^{n}) = \sup_{[0,1]^{n}} \| \mathbbm{1}_{[0,1]^{n}} f \|_{L^{1}(\Rn)}. 
$$
\end{itemize}
\end{definition}
\ni It holds the following: 
\begin{lemma}[\cite{Lem}]\label{OseenDef}
The operator $\frac{1}{\Delta}\partial_{i}\partial_{j}\partial_{k}$ is a convolution operator with a kernel $T_{i,j,k}$ such that the following decomposition holds:
$$
T_{i,j,k}=\alpha_{i,j,k}+\partial_{i}\partial_{j}\beta_{k}
$$
where $\alpha_{i,j,k} \in WL^{\infty}(\mathbb{R}^{n})$ and $\beta_{k} \in L^{1}_{loc}(\mathbb{R}^{n})$.
\end{lemma}
\ni By lemma (\ref{OseenDef}) and inclusions 
$$
L^{1}_{loc} * L^{1}_{uloc} \subset L^{1}_{uloc}, \qquad WL^{\infty}*L^{\infty},
$$
it turns out that $\mathbb{P}(\nabla \ \cdot)$ can be defined on the space $(L^{1}_{uloc}(\Rn))^{n \times n}$. Now we focus on some properties of convolutions with the Oseen Kernel, so we consider:
$$
\frac{1}{\Delta}\partial_{i}\partial_{j}e^{t\Delta}.
$$
It holds the following:
\begin{lemma}[\cite{Lem}]\label{OseenKernelTheorem}
Let $1\leq i,j \leq n$. The operator $\frac{1}{\Delta}\partial_{i}\partial_{j}e^{t\Delta}$ is a convolution operator $O_{i,j}(t)*f_{j}$, with $O_{i,j}(t) \in (C^{\infty}(\Rn))^{n\times n}$, the homogeneity:
$$
O_{i,j}(t) = \frac{1}{t^{\frac{n}{2}}}O_{i,j}\left( \frac{x}{\sqrt{t}} \right),
$$
and the decay:
$$
(1+|x|)^{n+|\eta|}\partial^{\eta}O_{i,j} \in (L^{\infty}(\Rn))^{n\times n},
$$
for each multi-index $\eta$.
\end{lemma}
\ni This is the main technical tool we need in order to study the properties of 
$ e^{t\Delta}\mathbb{P}(\nabla \ \cdot \ ) $ that acts on the tensor $u\otimes u$ through
$$
(e^{t\Delta} \mathbb{P}(\nabla \cdot (u \otimes u)))_{i}=
e^{t\Delta} \partial_{j}(u \otimes u)_{i,j} 
- e^{t\Delta}\frac{1}{\Delta} \partial_{i}\partial_{j}\partial_{k} (u \otimes u)_{j,k}.
$$ 
It holds the following:
\begin{proposition}[\cite{Lem}]\label{OseenDecayFinale}
Let $1\leq i,j,k \leq n$. The operator $e^{t\Delta} \mathbb{P} (\nabla \ \cdot \ )$ is a convolution operator $K_{i,j,k}(t)*f_{j,k}$, with $K_{i,j,k}(t) \in (C^{\infty}(\Rn))^{n\times n}$, the homogeneity:
$$
K_{i,j,k}(t) = \frac{1}{t^{\frac{n+1}{2}}}K_{i,j,k}\left( \frac{x}{\sqrt{t}} \right),
$$
and the decay:
$$
(1+|x|)^{n+ 1 +|\eta|}\partial^{\eta}K_{i,j,k} \in (L^{\infty}(\Rn))^{n\times n},
$$
for each multi-index $\eta$.
\end{proposition}
\ni We finish by stating an useful equivalence result:
\begin{theorem}[\cite{Lem}]\label{DiffvsInt}
Let $u \in \cap_{s<T} \left( L^{2}_{t}L^{2}_{uloc,x}(0,s) \times \Rn \right)$. Then the following are equivalent:
\begin{enumerate}
\item $u$ is a weak solution of:
\begin{equation}
\left \{
\begin{array}{rcccl}
\partial_{t}u + \mathbb{P} \nabla \cdot (u \otimes u)  & = & \Delta u & \quad \mbox{in}& \quad [0,T)\times \mathbb{R}^{n}  \\
\nabla \cdot u & = & 0 & \quad \mbox{in} & \quad [0,T)\times \mathbb{R}^{n} \\
u & = & u_{0} & \quad \mbox{in}& \quad \{0\} \times \mathbb{R}^{n}. 
\end{array}\right.
\end{equation}
\item $u$ solves the integral problem:
\begin{equation}
\left \{
\begin{array}{rclcl}
 u & = & e^{t\Delta}u_{0} + \int_{0}^{t}e^{(t-s)\Delta}\mathbb{P} \nabla \cdot (u \otimes u) ds &\quad \mbox{in} \quad & [0,T)\times \mathbb{R}^{n}  \\
\nabla \cdot u & = & 0 &\quad \mbox{in} \quad & [0,T)\times \mathbb{R}^{n}. 
\end{array}\right.
\end{equation}
\end{enumerate}
\end{theorem}

\section{The Leray-Hopf solutions}

\ni The modern theory of Navier-Stokes equation starts with the Leray's work \cite{Ler} in which global existence of weak solutions of (\ref{CauchyNS}) for $L^{2}$ initial data is proved. Such weak solutions have also physical meaning because they respond to the energy dissipation. On the other hand the existence theorem is by compactness and neither regularity nor uniqeness have been proved in the general case. In this section we will briefly sketch the ideas of the Leay's theory. We start by definition of weak solution in the context of \cite{Ler}
 \begin{definition}[Leray's solutions]
The pair $(u,p)$ is a weak Leray solution of the Navier-Stokes system (\ref{CauchyNS}) in $[0,T) \times \mathbb{R}^{n}$ if the following holds:

\begin{enumerate}
  \item Exist some constants $E_{0}, E_{1}$ such that:
      \begin{equation} \label{EnergyBoundness}
      \int_{\mathbb{R}^{n}} |u(t, \cdot)|^{2} \ dx \leq E_{0},
      \end{equation}      
      for almost every $t \in (0,T)$, and 
      \begin{equation}\label{EnergyBoundness2}
      \int_{0}^{T}\int_{\mathbb{R}^{n}} |\nabla u |^{2} \ dxdt \leq E_{1};
      \end{equation}
   \item (u,p) satisfy (\ref{CauchyNS}) in the sense of distributions in $[0,T) \times \mathbb{R}^{n}$, that's
  \begin{equation}\label{DistribSense}
  \int_{0}^{T} \int_{\Rn} (\partial_{t} \phi + (u \cdot \nabla) \phi) u \ dxdt + 
  \int_{\Rn} u_{0} \phi (x,0) \ dx  = \int_{0}^{T} \int_{\Rn} (\nabla \phi \cdot \nabla) u \ dxdt, \end{equation}   
   for each $\phi \in C^{\infty}_{c}([0,T) \times \Rn)$ with $\nabla \cdot \phi =0$ and
   \begin{equation}\label{DistrSense2}
   \int_{0}^{T} \int_{\Rn} u \cdot \nabla \phi \ dxdt=0, \quad
     \int_{0}^{T} \int_{\Rn}  p \Delta \phi + \sum_{i,j=1}^{n} u_{i}u_{j} \partial_{i}\partial_{j} \phi =0,
   \end{equation}
for each $\phi \in C^{\infty}_{c}([0,T) \times \Rn)$.
   \item $u$ satisfy the energy inequality: 
       \begin{equation} \label{Energy} 
       \int_{\mathbb{R}^{n}} |u(t, \cdot)|^{2} + 2\int_{0}^{t}\int_{\mathbb{R}^{n}} |\nabla u |^{2} \ dxdt  \leq \int_{\mathbb{R}^{n}} |u_{0}|^{2} 
       \end{equation}
       for each $t \in (0,T)$.
\end{enumerate}
\end{definition}
\ni Condition ($\ref{Energy}$) is expression of the dissipation of kinetic energy (the first term of the sum) caused by the frictions (the second term). It can be justified by multiplication of (\ref{CauchyNS}) with $2\phi u$ and integration by parts.

\ni It is well know that Leray weak solutions are weakly continuous (see \cite{Temam}), i.e.
\begin{equation}\label{WC}
\lim_{t \rightarrow s} \int_{\mathbb{R}^{n}} u(t,x)w(x) \ dx = \int_{\mathbb{R}^{n}} u(s,x) w(x) \ dx
\end{equation}
for all $w \in L^{2}(\mathbb{R}^{n})$, 
and so, if $u_{0} \in L^{2}(\mathbb{R}^{n}) $ then
\begin{equation}\label{WCin0}
\lim_{t\rightarrow 0} \int_{\mathbb{R}^{n}} u(t,x)w(x) \ dx = \int_{\mathbb{R}^{n}} u_{0}(x) w(x) \ dx,
\end{equation}
for all $w \in L^{2}(\mathbb{R}^{n})$. This is how $u$ attend its initial datum. In \cite{Ler} is proved the existence of a weak solution $u \in \Rpiu \times \mathbb{R}^{n}$ of (\ref{CauchyNS}) for every $u_{0} \in L^{2}_{\sigma}(\mathbb{R}^{n})$. If we set the problem in $[0,T]\times \Omega$, $\Omega \subset \mathbb{R}^{n}$ open and bounded, and we require zero Dirichlet condition on $[0,T] \times \partial \Omega$, an analogous result is due to Hopf \cite{Hopf}.
Let then state the precise Leray's result:
\begin{theorem}[\cite{Ler}]\label{Leray'sTheorem}
Let $u_{0} \in L^{2}_{\sigma}(\mathbb{R}^{n})$. There exist a weak solution $u \in L^{\infty}(\mathbb{R}^{+}; L^{2}_{\sigma}(\mathbb{R}^{n})) \cap L^{2}(\mathbb{R}^{+}; \dot{H}^{1}(\mathbb{R}^{n}))$ of the Cauchy problem (\ref{CauchyNS}) in $\mathbb{R}^{+} \times \mathbb{R}^{n}$. Then $u$ weakly attend its initial datum, i.e. 
$$
\lim_{t \rightarrow 0} \int_{\Rn} (u(t,x)-u_{0,x})w(x) \ dx =0, \qquad \forall w \in L^{2}(\Rn).
$$
Moreover the energy 
inequality holds:
\begin{equation}\label{energy*}
\int_{\mathbb{R}^{n}} |u(t, \cdot)|^{2} + 2\int_{0}^{T}\int_{\mathbb{R}^{n}} |\nabla u |^{2} \ dxdt  \leq \int_{\mathbb{R}^{n}} |u_{0}|^{2}, \qquad \forall t \in \mathbb{R}^{+}. 
\end{equation}  
\end{theorem}
\begin{remark}
The choice $u_{0}\in L^{2}_{\sigma}$ is of course the more natural and phisically relevant for the problem. In such a way all initial data with bounded energy are covered. Anyway, as noticed, this generality leads to a poor wellposendess theory in which neither uniqeness nor persistence of regularity are guaranteed. Instead, as will seen in the next section, well posedness can be achieved if we restrict to small (in a suitable sense) initial data.
\end{remark}
\begin{proof}
We just sketched the proof. Details can be found widely in literature. A possible way to get the Leray's theorem is consider the family of regularized systems:
\begin{equation}\label{CauchyNSepsilon}
\left \{
\begin{array}{rcccl}
\partial_{t}u^{\varepsilon} + (u^{\varepsilon}*\rho^{\varepsilon} \cdot \nabla) u^{\varepsilon} +\nabla p  & = & \Delta u^{\varepsilon} & \quad \mbox{in}& \quad \Rpiu \times \mathbb{R}^{n}  \\
\nabla \cdot u^{\varepsilon} & = & 0 & \quad \mbox{in} & \quad \Rpiu \times \mathbb{R}^{n} \\
u^{\varepsilon} & = & u_{0} & \quad \mbox{in}& \quad \{0\} \times \mathbb{R}^{n}, 
\end{array}\right.
\end{equation}
where $\rho^{\varepsilon}$ is a standard mollifier of size $\varepsilon$, that's:
$$
\rho \in C^{\infty}_{c}(\Rn), \qquad \rho^{\varepsilon}=\frac{1}{\varepsilon^{n}}\rho(\frac{x}{\varepsilon}).
$$ 
Now if $u_{0}\in L^{2}_{\sigma}(\Rn)$ exists for each $\varepsilon$ a unique global (and smooth in space) solution $u^{\varepsilon}$ of problem (\ref{CauchyNSepsilon}). Furthermore the functions $u^{\varepsilon}$ satisfies the energy inequality:
\begin{equation}\label{EnergYBoundEpsilon}
\int_{\mathbb{R}^{n}} |u^{\varepsilon}(t, \cdot)|^{2} + 2 \int_{0}^{t}\int_{\mathbb{R}^{n}} |\nabla u^{\varepsilon} |^{2} \ dxdt  \leq \int_{\mathbb{R}^{n}} |u_{0}|^{2}, \quad \forall t > 0,  
\end{equation}
uniformly in $\varepsilon$. This follows simply by taking the scalar product of equation (\ref{CauchyNSepsilon}) with $u$ and by integrating by parts\footnote{By integration by parts we get $\int_{0}^{t}\int_{\Rn} \sum_{i,j=1}^{n} (u^{\varepsilon}*\rho^{\varepsilon})  (\partial_{j} u_{i}^{\varepsilon}) u^{\varepsilon}_{i}, \int_{0}^{t}\int_{\Rn}  \sum_{i=1}^{n}(\partial_{i} p)  u_{i}^{\varepsilon} =0$.}:
\begin{eqnarray}\nonumber
0 &=&\int_{0}^{t}\int_{\Rn} \partial_{t}u^{\varepsilon}\cdot u^{\varepsilon} + (u_{j}^{\varepsilon}*\rho^{\varepsilon} \cdot \nabla) u^{\varepsilon} \cdot u^{\varepsilon} +\nabla p \cdot u^{\varepsilon} -\Delta u^{\varepsilon} \cdot u^{\varepsilon} \ dxdt \\ \nonumber
&=& \int_{0}^{t}\int_{\Rn} \sum_{i=1}^{n}(\partial_{t}u_{i}^{\varepsilon}) u_{i}^{\varepsilon} +
\sum_{i,j=1}^{n} (u^{\varepsilon}*\rho^{\varepsilon})  (\partial_{j} u_{i}^{\varepsilon}) u^{\varepsilon}_{i} \\ \nonumber
&+& \sum_{i=1}^{n}(\partial_{i} p)  u_{i}^{\varepsilon} -(\sum_{i,j=1}^{n}(\partial_{j}\partial_{j} u_{i}^{\varepsilon})  u_{i}^{\varepsilon} \ dxdt \\ \nonumber
&=& \int_{0}^{t}\int_{\Rn} \frac{1}{2}\partial_{t}(\sum_{i=1}^{n} u_{i}^{\varepsilon})^{2} \ dxdt + \int_{0}^{t}\int_{\Rn} 
\sum_{i,j=1}^{n}(\partial_{j} u_{i})^{2} \ dxdt \\ \nonumber
&=& \int_{0}^{t}\int_{\Rn} \frac{1}{2}|u^{\varepsilon}|^{2} \ dxdt + \int_{0}^{t}\int_{\Rn} 
|\nabla u|^{2} \ dxdt. \nonumber
\end{eqnarray}
\ni Equation (\ref{EnergYBoundEpsilon}) allows to recover $u$ by compactness from the sequence $u^{\varepsilon}$. The solutions $u^{\varepsilon}$ are built up as a Picard sequences related to the integral formulation of (\ref{CauchyNSepsilon}), i.e.
\begin{equation}\label{CauchyNSIntEpsilon}\nonumber
\left \{
\begin{array}{rclcl}
 u^{\varepsilon} & = & e^{t\Delta}u_{0} + \int_{0}^{t}e^{(t-s)\Delta}\mathbb{P} \nabla \cdot ((u^{\varepsilon}*\rho^{\varepsilon}) \otimes u^{\varepsilon}) ds &\quad \mbox{in} \quad & \Rpiu \times \mathbb{R}^{n}  \\
\nabla \cdot u^{\varepsilon} & = & 0 &\quad \mbox{in} \quad & \Rpiu \times \mathbb{R}^{n}. 
\end{array}\right.
\end{equation}
So $u^{\varepsilon}$ is the limit of:
\begin{equation}\label{PicardForLerayEpsilon}
\begin{array}{rcl}
u^{\varepsilon}_{1} & = &e^{t\Delta}u_{0} \\
u^{\varepsilon}_{2} & = & e^{t \Delta}u_{0} + \int_{0}^{t}e^{(t-s)\Delta}\mathbb{P}\nabla \cdot ((u^{\varepsilon}_{1}*\rho^{\varepsilon}) \otimes u^{\varepsilon}_{1})(s)ds  \\
u^{\varepsilon}_{n} &=& e^{t\Delta}u_{0} + \int_{0}^{t}e^{(t-s)\Delta}\mathbb{P}\nabla \cdot ((u^{\varepsilon}_{n-1}*\rho^{\varepsilon}) \otimes u^{\varepsilon}_{n-1})(s)ds.
 \end{array}
\end{equation}
\end{proof}
\ni Notice finally that     
$$
u \in \cap_{s<T} \left( L^{2}_{t}L^{2}_{uloc,x}(0,s) \times \Rn \right), \quad \forall t \in \Rpiu,
$$
because of (\ref{energy*}); so by theorem \ref{DiffvsInt} we have the integral representation:
$$
u  =  e^{t\Delta}u_{0} + \int_{0}^{t}e^{(t-s)\Delta}\mathbb{P} \nabla \cdot (u \otimes u) ds.
$$

\section{Regularity criteria}

\ni As seen in theorem (\ref{Leray'sTheorem}) weak global solutions of problem (\ref{CauchyNS}) are always available for initial data $u_{0}$ with bounded energy ($u_{0} \in L^{2}_{\sigma}(\Rn)$). It is not known, on the other hand, if such solutions are unique, neither if they preserve the regularity of $u_{0}$. In this section we focus on partial regularity criterions. The philosophy of such a criterions is that we can infer about the uniqeness or regularity of weak solutions by the knowledge of some a priori information on the solution itself. In general a partial regularity criterion is an assertion of the following type:

\

\ni {\em Let $u$ a Leray's solution of (\ref{CauchyNS}). Let furthermore assume some additional a priori properties (tipically these are boundedness conditions) about $u$, then $u$ is the unique Leray solution of (\ref{CauchyNS}). Furthermore $u$ is $C^{\infty}$ in the space variables at each time $t>0$.}

\

\ni The first result of this kind goes back to Serrin \cite{Ser}. The author consider a little different setting. He works with weak solutions in open regions $\Omega \subseteq \mathbb{R}^{+} \times \Rn$, that is a couple $(u,p)$ $u$ such that:
\begin{equation}\label{WeakSolutionsInOmega}
  \int \int (\partial_{t} \phi + (u \cdot \nabla) \phi) u -  (\nabla \phi \cdot \nabla) u = 0, \end{equation}   
   for each $\phi \in C^{\infty}_{c}(\Omega)$ with $\nabla \cdot \phi =0$ and
   \begin{equation}\label{WeakSolutionsInOmegaDiv}
   \int \int u \cdot \nabla \phi = 0, \quad
     \int \int  p \Delta \phi + \sum_{i,j=1}^{n} u_{i}u_{j} \partial_{i}\partial_{j} \phi =0,
   \end{equation}
for each $\phi \in C^{\infty}_{c}(\Omega)$.
Conditions (\ref{WeakSolutionsInOmega}, \ref{WeakSolutionsInOmegaDiv}) are formally justified by taking the scalar product of the equations with the vector field $\phi$ and then by integrating by parts. 
\begin{remark}
Let $\Omega = (0,T) \times \Omega'$ where $\Omega'$ is an open subset of $\Rn$. Let then $\Psi$ a scalar function such that:
$$
\Delta \Psi =0 \quad \mbox{in} \quad \mathbb{R}^{n},
$$
and $a:(0,T)\rightarrow \mathbb{R}$ is integrable in $(0,T)$. Then it's easy to check that:
$$
u(t,x)= a(t)\Psi(x)
$$
is a weak solutions of (\ref{WeakSolutionsInOmega}).
\end{remark}
\ni The remark shows that regularity in space and time have to be analyzed in different ways\footnote{This is physically expectable by the incompressibility of fluid. A little change of fluid velocity localized in space causes instantly a global change of fluid velocity, so the time derivatives of the velocity are expected to be more singular.}. While it is reasonable that little a priori assumptions on $u$ are sufficient to get space regularity, stronger assumptions should be necessary in order to get time regularity. In this spirit the following holds:
\begin{theorem}[\cite{Ser}, \cite{Struwe}, \cite{Sohr}]\label{SerrinRegularity}
Let $u$ a weak solution of the Navier-Stokes equation in the open space-time region $\Omega$. Then define:
$$
\Omega_{t}=\{ t \} \times \Rn \cap \Omega.
$$
If $u \in L^{\infty}_{t}L^{2}_{\Omega_{t}}$, $\nabla u \in L^{2}_{t}L^{2}_{\Omega_{t}}$ and furthermore:
\begin{equation}\label{SerrinScaling}
\int_{0}^{+\infty} \| u \|^{s}_{L^{p}_{x}(\Omega_{t})}(t) \ dt < +\infty, \quad \frac{2}{s} + \frac{n}{p} \leq 1;
\end{equation}
then $u \in C^{\infty}$ in the space variables at each times $t$ such that $\Omega_{t} \neq 0$. Assume in addition that:
\begin{equation}\label{SerrinTimeDerivatives}
\partial_{t} u \in L^{p}_{t}L^{2}_{\Omega_{t}}, \qquad p \geq 1.
\end{equation}
Then $\partial_{x_{i}}u(t,x)$ are absolute continuous in time and exists a differentiable function $p(t,x)$ such that:
$$
\partial_{t} u - \Delta u + ( u \cdot \nabla ) u = - \nabla p
$$
almost everywhere in $\Omega$.
\end{theorem}  
\ni Actually Serrin has proved the theorem just in the case $\frac{2}{s} + \frac{n}{p} < 1$. The endpoint case has been fixed in \cite{Sohr}, \cite{Struwe}.
The critical relation $\frac{2}{s} + \frac{n}{p}=1$ follows by requiring $L^{s}_{t}L^{p}_{x}$ invariance under the scaling:
$$
\lambda \rightarrow \lambda u (\lambda^{2}t, \lambda x).
$$  
The regularity and the uniqeness of weak solutions are strictly related. A good example is the fact that under condition (\ref{SerrinScaling}) in $(0,T)\times \Rn$ also uniqeness is easily achieved. This is again due to Serrin (see \cite{Lem}).
\begin{lemma}[J.Serrin]\label{SerrinUniqness}
Let $u_{0} \in L^{2}_{\sigma}(\Rn)$ and $u$ a Leray solution of (\ref{CauchyNS}). If furthermore
\begin{equation}
\int_{0}^{T} \| u \|^{s}_{L^{p}_{x}(\Rn)}(t) \ dt < +\infty, \quad p \in (n,+\infty], \quad
\frac{2}{s} + \frac{n}{p} =1,
\end{equation}
then $u$ is unique in $[0,T)$.
\end{lemma}
\ni Another fundamental regularity result is due to Caffarelli, Kohn and Nirenberg \cite{CKN}. We will again state it without the proof, that's quite difficult, and can be found in the original paper or, in a simplified version, in \cite{Lin}. The authors give a local regularity criterion for system (\ref{CauchyNS}). A local criterion is an assertion of the following type:

\

\ni {\em Let $u$ a Leray solution of (\ref{CauchyNS}) and $(s,y)$ a fixed point in the space-time $\mathbb{R}^{+}\times \Rn$. If $u$ satisfies some a priori boundedness condition near (in a sense to be specified) the point $(s,y)$, then $u$ is $C^{\infty}$ in the space variables in the point $(s,y)$.}

\
  
\ni To state the criterion in \cite{CKN} we need some preliminaries:
\begin{definition}
The parabolic cylinder $Q(s,y,r)$ with top centered  in $(s,y)$ is the set:
$$
Q(s,y,r) = B(y,r) \times (s,s-r^{2}),
$$
where $B(y,r)$ is the $n$-dimensional ball of radius $r$ centred in $y$. The set $Q$ is important in the study of parial regularity of (\ref{CauchyNS}) because is invariant under the scaling: $(s,y) \rightarrow (\lambda^{2}s,\lambda x)$.
\end{definition}
\ni We need also consider a different definition of weak solution than the Leray's one. 
\begin{definition}[Suitable solutions, \cite{CKN}]\label{SuitSolDef}
Let $\Omega$ a open subset of $\mathbb{R}^{+}\times \Rn$ and:
$$
\Omega_{t} = \{ t \} \times \Rn \cap \Omega.
$$
The pair $(u,p)$ is a suitable weak solution of the Navier-Stokes equation if:
\begin{enumerate}
  \item $p \in L^{\frac{n+2}{4}}(\Omega)$;
  \item Exist some constants $E_{0}, E_{1}$ such that:
      \begin{equation} \label{GenEnergyBoundness}
      \int_{\Omega_{t}} |u(t, \cdot)|^{2} \ dx \leq E_{0},
      \end{equation}      
      for almost every $t$ such that $\Omega_{t}\neq 0$, and 
      \begin{equation}\label{GenEnergyBoundness2}
      \int\int_{\Omega} |\nabla u |^{2} \ dxdt \leq E_{1};
      \end{equation}
   \item (u,p) satisfy (\ref{CauchyNS}) in the sense of distributions in $\Omega$;
   \item A generalized version of the energy inequality holds: 
       \begin{equation} \label{GenEnergy} 
       2\int\int |\nabla u|^{2} \phi \leq \int\int |u|^{2}(\phi_{t} + \Delta \phi) +
       (|u|^{2} +2p)u \cdot \nabla \phi 
       \end{equation}
       for each function $\phi \in C^{\infty}_{c}(\Omega)$.
       
\end{enumerate}
\end{definition}
\ni The inequality (\ref{GenEnergy}) can be formally justified by taking the scalar product of the equation (\ref{CauchyNS}) against the vector field $\phi u$ and then by integrating by parts.
It is straightforward to check that Leray solutions are also suitable solutions. Furthermore definition the (\ref{SuitSolDef}) is meaningful, infact:
\begin{theorem}[\cite{CKN}]
Let $u_{0}\in L^{2}_{\sigma}(\Omega)$, with $\Omega$ an open subset of $\Rn$. Then there exist a pair: 
$$(u,p): (\Rpiu \times \Omega, \Rpiu \times \Omega) \rightarrow (\Rn, \mathbb{R}),$$ 
such that $(u,p)$ is a suitable solution of the Navier-Stokes equation in $\Rpiu \times \Omega$. Furthermore $u$ attains $u_{0}$ as initial datum in the following sense:  	
$$
\int_{\Omega}u(t,x) \phi(x) \ dx \rightarrow \int_{\Omega} u_{0}(x)\phi(x) \ dx, \qquad
\mbox{as} \quad t\rightarrow 0,
$$
for each $\mathbb{\phi} \in L^{2}_{\Omega}$.
\end{theorem}
Now we are ready to state the fundamental local regularity criterion:
\begin{lemma}[\cite{CKN}]\label{CKNLemma}
Let $n\geq 3$. Let then $\Omega$ an open subset of $\Rpiu\times\Rn$ and the pair $(u,p)$ a suitable solution of the Navier-Stokes equation in $\Omega$. Let then $(s,y)$ some point in $\Omega$. There exist an absolute constant $\varepsilon$ such that if:
\begin{equation}\label{CKNINtegrability}
\limsup_{r\rightarrow 0} \frac{1}{r^{n-2}} \int \int_{Q^{*}(r,s,y)} |\nabla u|^{2} \leq \varepsilon,
\end{equation}
where $Q^{*}(r,s,y) = Q(r,s + \frac{1}{8}r^{2}, y)$. Then $u$ is regular ($C^{\infty}$ in the space variables) in a neighborhood of $(s,y)$.
\end{lemma} 
\ni In the following we will use a little different formulation of the lemma, that's more convenient in order to work with weighted spaces:
\begin{lemma}[\cite{CKN},\cite{YongZhou}]\label{WeightedCKNLemma}
Let $u_{0} \in L^{2}_{\sigma}(\Rn)$ and $u$ a Leray solution of (\ref{CauchyNS}). Let then 
\begin{equation}\label{Integrability}
\int_{0}^{T} \int_{\Rn}|x|^{2-n} |\nabla u|^{2} \ dt dx < +\infty;
\end{equation}
then $u$ is regular ($C^{\infty}$ in the space variables) in the segment $(0,T)\times \{ 0 \}$.  
\end{lemma}
\begin{proof}
Let $t \in (0.T)$, then
\begin{eqnarray}
& & \limsup_{r\rightarrow 0} \frac{1}{r^{n-2}}\int\int_{Q(r,t,0)} 
|\nabla u|^{2}(t,x) \ dxdt    \nonumber\\ 
&=& \limsup_{r\rightarrow 0} \frac{1}{r^{n-2}}\int_{t-r^{2}}^{t}\int_{B(x,r)} 
|\nabla u|^{2}(t,x) \ dxdt  \nonumber   \\
& \leq & \limsup_{r\rightarrow 0} \int_{t-r^{2}}^{t}\int_{B(x,r)} 
|x|^{2-n}|\nabla u|^{2}(t,x) \ dxdt \nonumber   \\
& \leq & \limsup_{r\rightarrow 0}  \int_{t-r^{2}}^{t}\int_{\Rn} 
|x|^{2-n}|\nabla u|^{2}(t,x) \ dxdt  =0, \nonumber 
\end{eqnarray}
where the continuity property of the integral (\ref{Integrability}) is used.
\end{proof}
\ni As suggested by this version of the Caffarelli-Kohn-Nirenberg lemma, a local boundedness condition, for instance (\ref{CKNINtegrability}), can be replaced by imposing boundedness in a suitable weighted $L^{p}$ space. We follow this point of view in the next of the thesis. A complete set of partial regularity criteria can be obtained, infact, by working with weighted norms, as shown in \cite{YongZhou}. Before stating this results we give some notations:
\begin{notation}
Let $\alpha \in \mathbb{R}$, $p,s \in [1,+\infty)$. Let then $f:\Rn \rightarrow \mathbb{R}$, $F: \Rpiu \times \Rn \rightarrow \mathbb{R}$. We will say:
\begin{itemize}
\item $f \in L^{p}_{|x|^{\alpha p}dx}$ if 
$$\Big( \int_{\Rn} |x|^{\alpha p} |f(x)|^{p} \ dx \Big)^{\frac{1}{p}} < +\infty,$$ 
and we denote this norm with $\|\cdot\|_{L^{p}_{|x|^{\alpha p}dx}}$, or with $\||x|^{\alpha}\cdot\|_{L^{p}_{x}}$;
\item $f \in L^{\infty}_{|x|^{\alpha}dx}$ if 
$$\sup_{x\in \Rn}  |x|^{\alpha} |f(x)| < +\infty,$$ 
and we denote this norm with $\|\cdot\|_{L^{\infty}_{|x|^{\alpha}dx}}$, or with $\||x|^{\alpha}\cdot\|_{L^{\infty}_{x}}$;
\item $F \in L^{s}_{t}L^{p}_{|x|^{\alpha p}dx}$ if 
$$\Big( \int_{\Rpiu}\Big| \int_{\Rn} |x|^{\alpha p} |F(t,x)|^{p} \ dx \Big|^{\frac{s}{p}} \ dt  \Big)^{\frac{1}{s}} < +\infty,$$ 
and we denote this norm with 
$\|\cdot\|_{L^{s}_{t}L^{p}_{|x|^{\alpha p}dx}}$, or with $\||x|^{\alpha}\cdot\|_{L^{s}_{t}L^{p}_{x}}$;
\item $F \in L^{s}_{T}L^{p}_{|x|^{\alpha p}dx}$ if 
$$\Big( \int_{0}^{T}\Big| \int_{\Rn} |x|^{\alpha p} |F(t,x)|^{p} \ dx \Big|^{\frac{s}{p}} \ dt  \Big)^{\frac{1}{s}} < +\infty,$$ 
and we denote this norm with $\|\cdot\|_{L^{s}_{T}L^{p}_{|x|^{\alpha p}dx}}$, or with $\||x|^{\alpha}\cdot\|_{L^{s}_{T}L^{p}_{x}}$.
\end{itemize}
\end{notation}
\ni We give similar definitions for $$L^{\infty}_{t}L^{p}_{|x|^{\alpha p}dx}, \quad L^{\infty}_{T}L^{p}_{|x|^{\alpha p}dx}, \quad L^{s}_{t}L^{\infty}_{|x|^{\alpha}dx}, \quad L^{s}_{T}L^{\infty}_{|x|^{\alpha}dx}.$$

\noindent As told partial regularity criterions in weighted Lebesgue spaces hold:
\begin{theorem}[\cite{YongZhou}]\label{YZTheorem}
Let $n \geq 3$ and $u_{0}$ a divergence free vector field such that $u_{0} \in H^{2}(\Rn)$ and:
\begin{equation}\label{Weightedu0bound}
\| |x|^{1-\frac{n}{2}} u_{0}\|_{L^{2}_{x}} < +\infty.
\end{equation}
If $u$ is a weak Leray's solution of (\ref{CauchyNS})  such that:
\begin{equation}\label{YZuBound1}
\||x|^{\alpha} u \|_{L^{s}_{T}L^{p}_{x}} < +\infty,
\end{equation}
with
\begin{equation}\label{YZCondition1}
\frac{2}{s}+ \frac{n}{p} = 1-\alpha, \quad \frac{2}{1-\alpha} < s < +\infty,
\quad \frac{n}{1-\alpha} < p < +\infty, \quad -1 \leq \alpha < +1;
\end{equation}
or
\begin{equation}\label{YZuBound2}
\||x|^{\alpha} u \|_{L^{2/(1-\alpha)}_{T}L^{\infty}_{x}} < +\infty, \quad -1 < \alpha < +1;
\end{equation}
or
\begin{equation}\label{YZuBound3}
\sup_{t \in (0,T)} \||x|^{\alpha} u \|_{L^{n/(1-\alpha)}_{x}} = \varepsilon, \quad -1 \leq \alpha \leq +1;
\end{equation}
with $\varepsilon$ sufficiently small; then actually $u$ is regular ($C^{\infty}$ in space variables) in the segment $(0,T) \times \{ 0 \}$.
\end{theorem}

\ni Of course regularity in $(0,T) \times \bar{x}$ is achieved if the weights and the norms are centered in $\bar{x}$ instead of in the origin.
This is infact a slightly different formulation of the theorem. In the original one the author gets global regularity by requiring
$$
\sup_{\bar{x} \in \Rn}  \||x-\bar{x}|^{1-\frac{n}{2}} u_{0}\|_{L^{2}_{x}} < +\infty,
$$
and
$$
\sup_{\bar{x} \in Rn} \||x-\bar{x}|^{\alpha} u \|_{L^{s}_{T}L^{p}_{x}} < +\infty.
$$
Such a formulation is more useful for our purposes.
\ni This is a local regularity criterion. The weight $|x|^{\alpha}$ localizes the norms near to the origin and provides regularity just in the points $(0,T)\times \{ 0 \}$. We will show how it is actually the case just for negative values of $\alpha$. If $\alpha \geq 0$ then global regularity can be achieved. We will show that it is also the case if, by taking $\alpha < 0$, we assume a suitable amount of angular integrability of the solution.

\begin{remark}
Notice that:
\begin{enumerate}
\item The first equation in (\ref{YZCondition1}) is the critical relation coming out by requiring $L^{s}_{t}L^{p}_{|x|^{\alpha p}dx}$ invariance under the scaling:
$$
u^{\lambda}: u(t,x) \rightarrow \lambda u (\lambda^{2}t,\lambda x);
$$
\item The estimates (\ref{YZuBound2}), (\ref{YZuBound3}) are the endpoint version of (\ref{YZuBound1}), obtained by setting $(s,p)$ equal to $(2/(1-\alpha), \infty)$ and $(\infty, 3/1-\alpha )$ respectively. These are consistent with the scaling relation (\ref{YZCondition1});
\item Condition (\ref{YZuBound3}) in the case $\alpha = 1$ becomes a smallness condition on the norm $\||x|^{\alpha} u \|_{L^{\infty}_{T}L^{\infty}_{x}}$ , which implies that the possible behaviour of the strong solution
can be $|x|^{-1+\varepsilon}$ at the neighbourhood of the origin for small $\varepsilon$. This recovers one of the
main results proved in \cite{CKN}, \cite{Tian} for suitable weak solutions;
\item Of course can be set $T=+\infty$ to get regularity for all times;
\item The range of values for $\alpha$ does not depend on the dimension.
\end{enumerate}
\end{remark}
\ni Then an analogous of theorem (\ref{YZTheorem}) holds by assuming a priori boundedness of $\nabla u$ in weighted Lebesgue spaces:
\begin{theorem}[\cite{YongZhou}]\label{YZDerTheorem}
Let $n\geq 3$ and $u_{0}$ a divergence free vector field such that $u_{0} \in H^{2}(\Rn)$ and:
\begin{equation}\label{Weightedu0boundDer}
\||x|^{1-\frac{n}{2}} u_{0} \|_{L^{2}_{x}} < +\infty.
\end{equation}
If $u$ is a weak Leray solution of (\ref{CauchyNS}) such that:
\begin{equation}\label{YZDeruBound1}
\||x|^{\alpha} \nabla u \|_{L^{s}_{T}L^{p}_{x}} < +\infty,
\end{equation}
with
\begin{equation}\label{YZDerCondition1}
\frac{2}{s}+ \frac{n}{p} = 2-\alpha, \quad 1 < s < +\infty,
\quad \frac{n}{2 -\alpha} < p < +\infty, \quad -1 \leq \alpha < +2;
\end{equation}
or
\begin{equation}\label{YZDeruBound2}
\||x|^{\alpha} \nabla u \|_{L^{2/(2-\alpha)}_{T}L^{\infty}_{x}} < +\infty, \quad 0 < \alpha \leq +2;
\end{equation}
or
\begin{equation}\label{YZDeruBound3}
\sup_{t \in (0,T)} \||x|^{\alpha} \nabla u \|_{L^{n/(2-\alpha)}_{x}} = \varepsilon, \quad -1 \leq \alpha < +2;
\end{equation}
with $\varepsilon$ sufficiently small; then actually $u$ is regular ($C^{\infty}$ in space variables) on the segment $(0,T) \times \{ 0 \}$.
\end{theorem}
\ni We give again similar remarks.
\begin{remark}
Notice that:
\begin{enumerate}
\item The first equation in (\ref{YZDerCondition1}) is assumed again to ensure invariance with respect to the scaling:
$$
u^{\lambda}: u(t,x) \rightarrow \lambda u(\lambda^{2} t,\lambda x),
$$
infact:
\begin{eqnarray} 
\| |x|^{\alpha} \nabla u^{\lambda}\|_{\Lsp} &=& 
\| |x|^{\alpha} \lambda^{2}(\nabla u)(\lambda^{2}t,\lambda x)\|_{\Lsp} \nonumber \\
&=& \lambda^{2-\alpha-\frac{2}{s}-\frac{n}{p}}\| |x|^{\alpha}u\|_{\Lsp}; \nonumber
\end{eqnarray}
\item Conditions $1<s$, $0 < \alpha$ in (\ref{YZDerCondition1}),(\ref{YZDeruBound2}) seems to be artificial with respect to the more natural $\frac{2}{2-\alpha} \leq s$, $-1 \leq \alpha$, but they are necessay in order to work with $\Lsp$ spaces with $s,p \geq 1$;
\item By setting $\alpha =\frac{1}{2}$, $s=p=2$ a similar result to lemma (\ref{CKNLemma}) is achieved;
\item A global regularity result is again achieved by setting $T=+\infty$.
\item The range of values for $\alpha$ does not depend again on the dimension.
\end{enumerate}
\end{remark}

\section{Well posedness with small data}

\ni A different approach to the existence and uniqeness of the solutions of (\ref{CauchyNS}) consists in considering small initial data. Such a big restriction provides complete well posedness, i.e. regularity, uniqeness and decay of solutions. Solutions with small initial data have been deeply investigated since by \cite{Kat}, and in \cite{Tat} the sharp case seems to be covered. The key point of the small data theory is that the nonlinear term $(u\cdot \nabla)u$ is negligible with respect to the others, so the equation (\ref{CauchyNS}) can be interpreted as a perturbed heat equation. This point of view suggests to perform a fixed point algorithm around the Heat propagator. Let's start by the Duhamel's representation:
\begin{equation}\label{IntegralCauchyNS}
\left \{
\begin{array}{rcll}
u & = & e^{t \Delta}u_{0} + \int_{0}^{t}e^{(t-s)\Delta}\mathbb{P}\nabla \cdot (u \otimes u)(s)ds &  \mbox{in} \quad [0,T)\times \mathbb{R}^{n}  \\
\nabla \cdot u & = & 0  & \mbox{in} \quad [0,T) \times \mathbb{R}^{n}.
 \end{array}\right.
\end{equation}
Where $\mathbb{P}$ is the Leray projection as defined in (\ref{LerProj}). So the solution is the sum of the linear propagator $e^{t\Delta}$ and a bilinear term $B(u,u)= \int_{0}^{t}e^{(t-s)\Delta}\mathbb{P}\nabla \cdot (u \otimes u)(s)ds$. Then the Picard iteration can be performed:
\begin{equation}\label{PicardForKato}
\begin{array}{rcl}
u_{1} & = &e^{t\Delta}u_{0} \\
u_{2} & = & e^{t \Delta}u_{0} + \int_{0}^{t}e^{(t-s)\Delta}\mathbb{P}\nabla \cdot (u_{1} \otimes u_{1})(s)ds  \\
u_{n} &=& e^{t\Delta}u_{0} + \int_{0}^{t}e^{(t-s)\Delta}\mathbb{P}\nabla \cdot (u_{n-1} \otimes u_{n-1})(s)ds.
 \end{array}
\end{equation}
It is easy to show the following:
\begin{theorem}[\cite{Lem}]\label{Picard}
Let $X_{T}$ a Banach space of functions defined on $[0,T]\times \mathbb{R}^{n}$ such that the bilinear form:
$$
B(u,v) = \int_{0}^{t}e^{(t-s)\Delta}\mathbb{P}\nabla \cdot (u \otimes v)(s)ds
$$
is bounded by $X_{T} \times X_{T}$ to $X_{T}$. Let then $X_{0} \subset \mathbb{S}'(\mathbb{R}^{n})$ such that
$$
\|e^{t\Delta} f\|_{X_{T}} \leq C_{X_{0},X_{T}}\|f\|_{X_{0}} \qquad \forall t \in (0,T].
$$
Under this assumptions $u_{n}$ is a Cauchy sequence that converges to a solution of the integral problem (\ref{IntegralCauchyNS}).
\end{theorem}

\ni Usually (see again \cite{Lem}) $X_{T}$ is called an admissible path space, while $X_{0}$ is called an adapted space. Furthermore under suitable smallness assumption on $\|u_{0}\|_{X_{0}}$  in theorem \ref{Picard} can be set $T=+\infty$ to get a global existence result.
Many adapted spaces have been considered in literature since by \cite{Kat} where $u_{0} \in L^{n}$. This result have been generalized to the homogeneous Sobolev space $\dot{H}^{\frac{n}{p}-1}$ in \cite{Kat2,Kat3}, to Morrey spaces in \cite{Morrey}, to Besov spaces in \cite{Gall} and of course a lot of alternative references are possible. The biggest space in which Picard iteration is possible seems to be $BMO^{-1}$, see \cite{Tat}. For instance the following continuous embeddings of adapted spaces holds:
$$
\dot{H}^{\frac{n}{p}-1} \subset L^{n} \subset
\dot{B}^{-1+\frac{n}{p}}_{p,p < \infty,\infty} \subset BMO^{-1}.
$$ 
Even if, as mentioned, several choices are possible, we will work basically with weighted $L^{p}$ spaces. 
\begin{remark}
Global well posedness with small data forces the adapted space $X_{0}$ to be invariant under the scaling $\lambda \rightarrow u_{0}^{\lambda}= \lambda u_{0}(\lambda x)$, i.e
$$
\|u_{0}^{\lambda}\|_{X_{0}} = \|u_{0}\|_{X_{0}}, \qquad \forall \lambda \in \mathbb{R}^{+}.
$$
This easily follows by the similarity propetry of equation (\ref{CauchyNS}). For instance, if we restrict to the $L^{p}$ case the right-scaling adapted space is $L^{n}(\mathbb{R}^{n})$. Let infact $u_{0} \in L^{p}(\mathbb{R}^{n})$ with $p > n$. Of course scaling
\begin{equation}
\begin{array}{rcl}
u^{\lambda}(t,x) & = & \lambda u(\lambda^{2}t, \lambda x) \qquad \lambda > 0 \\
u_{0}^{\lambda}(x) & = & \lambda u_{0}(\lambda x) 
\end{array}
\end{equation}
leads to a one-parameter family of solutions of (\ref{CauchyNS}). If furthermore a global well posedness result with small data is achieved, by setting $\lambda \rightarrow +\infty$ it follows well posedness with arbitrarily large initial data. The same is in the case $p<n$,  by setting $\lambda \rightarrow 0$.
\end{remark}

\ni The first result in small data teory goes back to \cite{Kat}:
\begin{theorem}\label{SmallKato}
Let $u_{0}$ a divergence free vector field on $\mathbb{R}^{n}$. Exists $\varepsilon >0$ such that if 
$$
\|u_{0}\|_{L^{n}(\mathbb{R}^{n})} < \varepsilon,
$$
then there is a unique solution $u : \mathbb{R}^{+}\times \mathbb{R}^{n} \rightarrow \Rn$ of the integral problem (\ref{CauchyNSInt}). Furthermore $u$ has the decay:
\begin{equation}\label{KatSolDecay}
\|u(t,\cdot)\|_{L^{q}(\mathbb{R}^{n})} \leq \frac{c_{q}}{t^{(1- \frac{n}{q})/2}}\|u_{0}\|_{L^{n}(\mathbb{R}^{n})}, \qquad t > 0;
\end{equation}
\begin{equation}\label{KatDerSolDecay}
\|\partial u(t,\cdot)\|_{L^{q}(\mathbb{R}^{n})} \leq \frac{c_{q}}{t^{(2- \frac{n}{q})/2}}\|u_{0}\|_{L^{n}(\mathbb{R}^{n})}, \qquad t > 0
\end{equation}
\end{theorem}
for all $q\in[n,+\infty]$. Solutions $u$ also obey to the bound:
\begin{equation}\label{KatSolBound}
\|u\|_{\Lrq} < +\infty, \quad \frac{2}{r} + \frac{n}{q} = 1, \quad n \leq q \leq \frac{n^{2}}{n-2}.
\end{equation}
This is of course the simplest small data result achievable. As we mentioned the optimal case seems to be covered by Koch and Tataru in \cite{Tat}. To state their result we need some more definitions:
\begin{definition}[$BMO(\mathbb{R}^{n})$]
A tempered distribution $u$ belongs to the space $BMO(\mathbb{R}^{n})$ if
\begin{equation}\label{BMONorm}
\Big(\sup_{x,R>0} \frac{2}{|B(x,R)|}\int_{0}^{R}\int_{B(x,R)} t |\nabla(e^{t\Delta} u)|^{2}(t,y) \ dydt\Big)^{\frac{1}{2}} < \infty.
\end{equation}
\end{definition}
\ni This is the Carleson charaterization of $BMO(\Rn)$. Other equivalent definitions are in \cite{Stein93-a}. The square root in (\ref{BMONorm}) is taken because in such a way th quantity is a seminorm\footnote{Of course (\ref{BMONorm}) of a constant function is zero.}.
\begin{remark}
Let $w$ the solution of the heat equation
$$
\left \{
\begin{array}{rcl}
\partial_{t}w - \Delta w & = & 0 \quad \mbox{in} \quad \Rpiu \times \mathbb{R}^{n}  \\
w & = & v \quad \mbox{in} \quad \{0\} \times \mathbb{R}^{n},
\end{array}\right.
$$
of course 
$$
\|v\|_{BMO(\Rn)} = \Big(\sup_{x,R>0} \frac{1}{|B(x,R)|}\int_{0}^{R^{2}}\int_{B(x,R)} | \nabla w|^{2}(t,y) \ dydt\Big)^{\frac{1}{2}},
$$
where we used the symbol $\|\cdot\|_{BMO(\Rn)}$ even if the quantity is a seminorm.
\end{remark}
\ni Then the space $BMO^{-1}(\Rn)$ is defined by:
\begin{definition}[$BMO^{-1}(\Rn)$]\label{BMO^-1Def}
A tempered disrtibution $v$ belongs to $BMO^{-1}(\Rn)$ if:
\begin{equation}\label{BMO^-1Norm}
\| v \|_{BMO^{-1}(\Rn)} = \Big(\sup_{x,R>0} \frac{1}{|B(x,R)|}\int_{0}^{R^{2}}\int_{B(x,R)} |e^{t\Delta} v|^{2}(t,y) \ dydt\Big)^{\frac{1}{2}} < \infty.
\end{equation}
In such a case $\|\cdot\|_{BMO^{-1}(\Rn)}$ is actually a norm.
\end{definition} 
\ni It is easy to show that if $v$ is a vector field such that $v_{i} \in BMO(\Rn)$ for each $i$, then $\nabla \cdot v \in BMO^{-1}(\Rn)$. 
The converse is also true, as stated by the following theorem that establishes the precise relationship between the two spaces:
\begin{theorem}[\cite{Lem}]
Let $u$ a tempred distribution. Then $u \in BMO^{-1}(\Rn)$ if and only if there exist $v_{i} \in BMO(\Rn)$ such that:
$$
u= \sum_{i=1}^{n} \partial_{i} v_{i}.
$$
\end{theorem}
\ni We also define the parabolic cylinder\footnote{This is a central object in the study of regulariry properties of (\ref{CauchyNS}) because it obeys to the scaling: $(t,x)\rightarrow (\lambda^{2}t,\lambda x)$; we will came back later on this topics.} centred in $x$ and with radius $R$:
\begin{equation}\label{ParCyl}
Q(x,R)=B(x,R)\times (0,R^{2}),
\end{equation} 
and introduce the adapted space $X$ by:
\begin{definition}[\cite{Tat}]
\begin{equation}
\|u\|_{X}= \sup_{t>0}t^{\frac{1}{2}}\|u(t)\|_{L^{\infty}(\Rn)} + \Big(\sup_{x,R>0} \frac{1}{|B(x,R)|}\int_{Q(x,R)} |u|^{2}(t,y) \ dydt\Big)^{\frac{1}{2}}.
\end{equation}
\end{definition}
\ni We are ready to state the theorem:
\begin{theorem}[\cite{Tat}]\label{TatTheo}
Let $u_{0}$ a divergence free vector field on $\mathbb{R}^{n}$. Exists $\varepsilon >0$ such that if 
$$
\|u_{0}\|_{BMO^{-1}(\Rn)} < \varepsilon,
$$
then there is a unique solution $u : \mathbb{R}^{+}\times \mathbb{R}^{n} \rightarrow \Rn$ of the integral problem (\ref{CauchyNSInt}). Furthermore $u \in X$.
\end{theorem}

\ni This result works with a translation invariant adapted spaces, and, as mentioned, there is a wide literature on the topics. On the other hand even if $X_{0}$ is not translation invariant local regularity results are still available, but this cases have not been so deeply investigated. In the following we focus in particular on weighted spaces with power weights of the kind $|x|^{\alpha}$. So the translation invariance is broken but we still have to require invariance with respect to the scaling centered in the origin
$$
\lambda \rightarrow \lambda u_{0}(\lambda x);
$$   
In this way the most simple spaces available are endowed with the norms
$$
\| |x|^{\alpha}  \cdot \|_{L^{p}(\Rn)}, \qquad \mbox{with} \quad \alpha = 1- \frac{n}{p}.
$$
In the case $p=2$ a very interesting result has been obtained in \cite{CKN}
\begin{theorem}[Caffarelli-Kohn-Nirenberg]\label{CKNSmallData}
Let $u_{0} \in L^{2}_{\sigma}(\Rn)$ and $u$ a suitable weak solution of (\ref{CauchyNS}). There exists an absolute constant $\varepsilon_{0} > 0$ such that if
$$
\| |x|^{1- n/2} u_{0} \|_{L^{2}(\Rn)} = \varepsilon < \varepsilon_{0},
$$ 
then $u$ is regular ($C^{\infty}$ in space variables) in the interior of the parabola
$$
\Pi = \left\{ (t,x) \quad \mbox{t.c.} \quad t > \frac{|x|^{2}}{\varepsilon_{0}- \varepsilon} \right\}.
$$
\end{theorem}

\begin{remark}
Of course by translations the analogous result holds if one consider small data in the norm $\| |x-\bar{x}|^{\alpha} \cdot\|_{L^{p}(\Rn)}$, for a fixed $\bar{x} \in \Rn$. 
\end{remark}

\ni We will come back on this result in the next chapter and we will show how it depends on the amount of angular integrability of $u_{0}$. By applying the technology developed in the first chapter we suggest how to quantify precisely the gain of regularity that angular integrability provides.

\chapter{Results in weighted setting with angular integrability}

\ni We apply the technology developed in Chapter \ref{SectInequality} to the study of the well posedness and regularity of (\ref{CauchyNS}). We closely look at these problems in the setting of weighted $L^{p}$ spaces with different integrability in radial and angular directions. More precisely we consider well posedness of (\ref{CauchyNS}) with $u_{0} \in L^{p}_{|x|^{\alpha p}d|x|}L^{\p}_{\theta}$ that's defined by
\begin{equation}\label{MixedLpSect4}
\begin{array}{lcl} 
 L^{p}_{|x|^{\alpha p}d|x|}L^{\p}_{\theta}   
 &=& \Big\{ f \in L^{1}_{loc}(\Rn) \quad \mbox{s.t.} \\
 && 
  \left(
    \int_{0}^{+\infty}
    \|f(\rho\ \cdot\ )\|^{p}_{L^{\p}(\mathbb{S}^{n-1})}
    \rho^{\alpha p + n-1} d \rho
  \right)^{\frac1p} < +\infty \Big\};
\end{array}
\end{equation} 
it is immediate to show that the above quantity is a norm and we use the notation 
$$ \|\cdot\|_{L^{p}_{|x|^{\alpha p}d|x|}L^{\p}_{\theta}}, \quad \mbox{or} \quad \||x|^{\alpha}\cdot\|_{\Lpptilde}.
$$
This space is not translation invariant, as the classical adapted spaces are, but as shown in theorem \ref{CKNSmallData} local regularity results are still achieveble in this setting.
\ni Then we are interested in regularity criteria in the context of spaces:
\begin{equation}\label{MixedLpSect4bis}
\begin{array}{lcl} 
L^{s}_{T}L^{p}_{|x|^{\alpha p}d|x|}L^{\p}_{\theta}  
 &=& \Big\{ u \in L^{1}_{loc}( (0,T) \times \Rn) \quad \mbox{s.t.}   \\
 & &  \left(
    \int_{0}^{T} \left( \int_{0}^{+\infty}
    \|u(t, \rho\ \cdot\ )\|^{p}_{L^{\p}(\mathbb{S}^{n-1})}
    \rho^{\alpha p + n-1} d \rho \right)^{\frac{s}{p}} dt
  \right)^{\frac1s} < +\infty \Big\},
\end{array}
\end{equation}
and again we denote this norms with  
$$
\|\cdot\|_{L^{s}_{T}L^{p}_{|x|^{\alpha p}d|x|}L^{\p}_{\theta}}, \quad  \mbox{or} \quad \||x|^{\alpha}\cdot\|_{\Lpptilde}.
$$
We use also $\|\cdot\|_{L^{s}_{t}L^{p}_{|x|^{\alpha p}d|x|}L^{\p}_{\theta}}$ and $\||x|^{\alpha}\cdot\|_{L^{s}_{t}\Lpptilde}$ if $T=+\infty$. Of course if we restrict to radially symmetric functions $u_{0}=u_{0}(|x|)$ (or $u=u(t,|x|)$) the norms reduce to the classical ones. Now it is well known that the problem (\ref{CauchyNS}) is very much simpler in a symmetric setting, and stronger results are achievable. The idea is try to recover some of this improvements in the case of initial data or solutions with merely higher angular integrability. The functional spaces well suited to this purpose are indeed (\ref{MixedLpSect4}), (\ref{MixedLpSect4bis}). We find encouraging results by considering $L^{p}_{|x|^{\alpha p}d|x|}L^{\p}_{\theta}$ with $\alpha <0$ and large value of $\p$. The idea is that while the weight $|x|^{\alpha}$ localizes the norm near to the origin, the radial $L^{\p}$ integrability, with large $\p$, provides a bound for large $|x|$. This heuristic will be more precisely formulated later, and will be useful to interpret all the following results.   
Of course by translations can be considered power weights centred in some $\bar{x} \neq 0$. In  this case all the norms have to be translated in $\bar{x}$. We give some precise definitions about the spaces we are going to use.
\begin{notation}
Let $\alpha \in \mathbb{R}$, $p,s \in [1,+\infty)$. Let then $f:\Rn \rightarrow \mathbb{R}$, $F: \Rpiu \times \Rn \rightarrow \mathbb{R}$. We will say:
\begin{itemize}
\item $f \in L^{p}_{|x|^{\alpha p}d|x|}L^{\p}_{\theta}$ if 
$$\Big( \int_{\Rpiu} \|f( \rho \cdot \ )\|^{p}_{L^{\p}_{(\mathbb{S}^{n-1})}}\rho^{\alpha p +n-1}  \ d \rho \Big)^{\frac{1}{p}} < +\infty,$$ 
and we denote this norm with $\|\cdot\|_{L^{p}_{|x|^{\alpha p}d|x|}L^{\p}_{\theta}}$, or $\||x|^{\alpha}\cdot\|_{L^{p}_{|x|}L^{\p}_{\theta}}$;
\item $f \in L^{\infty}_{|x|^{\alpha}d|x|}L^{\p}_{\theta}$ if 
$$\sup_{ \rho > 0 }  \rho^{\alpha} \|f( \rho \ \cdot \ )\|_{L^{\p}(\mathbb{S}^{n-1})} < +\infty,$$ 
and we denote this norm with $\|\cdot\|_{L^{\infty}_{|x|^{\alpha}d|x|}L^{\p}_{\theta}}$, or $\||x|^{\alpha}\cdot\|_{L^{\infty}_{|x|}L^{\p}_{\theta}}$;
\item $F \in L^{s}_{t}L^{p}_{|x|^{\alpha p}d|x|}L^{\p}_{\theta}$ if 
$$\Big( \int_{\Rpiu} \Big| \int_{\Rpiu} \|F(t, \rho \ \cdot \ )\|^{p}_{L^{\p}(\mathbb{S}^{n-1})} \rho^{\alpha p +n-1}  \ d \rho \Big|^{\frac{s}{p}} \ dt  \Big)^{\frac{1}{s}} < +\infty,$$ 
and we denote this norm with 
$\|\cdot\|_{L^{s}_{t}L^{p}_{|x|^{\alpha p}d|x|}L^{\p}_{\theta}}$, or $\||x|^{\alpha}\cdot\|_{L^{s}_{t}L^{p}_{|x|}L^{\p}_{\theta}}$;
\item $F \in L^{s}_{T}L^{p}_{|x|^{\alpha p}d|x|}L^{\p}_{\theta}$ if 
$$\Big( \int_{0}^{T}\Big| \int_{\Rpiu} \|F(t,\rho \ \cdot \ )\| \rho^{\alpha p + n-1}  \ d \rho \Big|^{\frac{s}{p}} \ dt  \Big)^{\frac{1}{s}} < +\infty,$$ 
and we denote this norm with $\|\cdot\|_{L^{s}_{T}L^{p}_{|x|^{\alpha p}d|x|}L^{\p}_{\theta}}$, or $\||x|^{\alpha}\cdot\|_{L^{s}_{T}L^{p}_{|x|}L^{\p}_{\theta}}$.
\end{itemize}
\end{notation}
We give similar definitions for $$L^{\infty}_{t}L^{p}_{|x|^{\alpha p}d|x|}L^{\p}_{\theta}, \quad L^{\infty}_{T}L^{p}_{|x|^{\alpha p}d|x|}L^{\p}_{\theta}, \quad L^{s}_{t}L^{\infty}_{|x|^{\alpha}d|x|}L^{\p}_{\theta}, \quad L^{s}_{T}L^{\infty}_{|x|^{\alpha}d|x|}L^{\p}_{\theta}.$$

\section{Decay estimates for convolutions with heat and Oseen kernels}

\ni The most important technical tools we need are weighted decay estimates for convolutions with the heat and Oseen kernels. Even if results of this kind already exist in literature we cover a larger set of weights by considering $L^{p}_{|x|}L^{\p}_{\theta}$ spaces. In particular we show that infact the higher angular integrability allows to consider a larger set of weights. Corollary (\ref{cor:nonhom}) is the main ingredient in the proofs.

\noindent To give a more compact notation it's convenient to define the quantities:
\begin{equation}\label{LambdaDef}
\Lambda (\alpha, p, \p) = \alpha + \frac{n-1}{p} - \frac{n-1}{\p}
\end{equation}
\begin{equation}\label{OmegaDef}
\Omega (\alpha,p,s) = \alpha + \frac np + \frac 2s.
\end{equation}
We will use also the notation $\Lambda_{a}$ instead of $\Lambda(\alpha,p,\p)$, when the values of $p,\p$ will be clear by the context. 
Let's start by the punctual decay: 

\begin{proposition}\label{PDecayCor}
Let $n\geq 2$, $1 < p \leq q < +\infty$ and $1 \leq \p \leq \q \leq +\infty$. Assume further that $\alpha, \beta$ satisfy the set of conditions 
 \begin{equation}\label{eq:condDL(Heat)}
    \beta > -\frac nq,\qquad \alpha<\frac{n}{p'}, \qquad
    \Lambda (\alpha,p,\p) \geq \Lambda (\beta,q,\q).
  \end{equation}
Then the following estimates hold:
\begin{enumerate}
         \item \begin{equation}\label{PHeatDer}
     \||x|^{\beta} \partial^{\eta} e^{t\Delta}u_{0}\|_{L^{q}_{|x|}L^{\q}_{\theta}} 
     \leq \frac{c_{\eta}}{t^{(|\eta| + \frac{n}{p}-\frac{n}{q} + \alpha-\beta)/2}} 
     \| |x|^{\alpha} u_{0}\|_{L^{p}_{|x|}L^{\p}_{\theta}}, \qquad t>0,
     \end{equation}
     provided that $|\eta| + \frac{n}{p}-\frac{n}{q} + \alpha-\beta \geq0$,
      \item \begin{equation}\label{PGHeatDer}
       \||x|^{\beta} \partial^{\eta} e^{t\Delta} \mathbb{P} \nabla \cdot F\|_{L^{q}_{x}L^{\q}_{\theta}} 
       \leq \frac{d_{\eta}}{t^{(1 + |\eta| + \frac{n}{p} -\frac{n}{q} +\alpha -\beta)/2}} \| |x|^{\alpha} F\|_{L^{p}_{|x|}L^{\p}_{\theta}}, \qquad t>0,
         \end{equation} 
provided that $1+ |\eta| + \frac{n}{p}-\frac{n}{q} + \alpha-\beta > 0$.
\end{enumerate}
for each multi index $\eta$, so in particular:
\begin{enumerate}
\item \begin{equation}\label{PHeat}
     \||x|^{\beta} e^{t\Delta}u_{0}\|_{L^{q}_{|x|}L^{\q}_{\theta}} 
     \leq \frac{c_{0}}{t^{(\frac{n}{p}-\frac{n}{q} + \alpha-\beta)/2}} 
     \| |x|^{\alpha} u_{0}\|_{L^{p}_{|x|}L^{\p}_{\theta}}, \qquad t>0,
     \end{equation}
provided that $\frac{n}{p}-\frac{n}{q} + \alpha-\beta \geq0$,
    \item \begin{equation}\label{PGHeat}
       \||x|^{\beta} e^{t\Delta} \mathbb{P} \nabla \cdot F\|_{L^{q}_{x}L^{\q}_{\theta}} 
       \leq \frac{d_{0}}{t^{(1 + \frac{n}{p} -\frac{n}{q} +\alpha -\beta)/2}} \| |x|^{\alpha} F\|_{L^{p}_{|x|}L^{\p}_{\theta}}, \qquad t>0,
         \end{equation} 
provided that $1+ \frac{n}{p}-\frac{n}{q} + \alpha-\beta > 0$. The range of admissible $p,q$ indices can be relaxed to $1\leq p \leq q \leq +\infty$ provided that $\Lambda (\alpha,p,\p) > \Lambda (\beta,q,\q)$. 
\end{enumerate}
\end{proposition}
\begin{proof}
The proof follows by proposition (\ref{cor:nonhom}) and scaling considerations. At first notice:
\begin{equation}\label{HeatScaling}
e^{t\Delta} \phi = S_{\sqrt{t}}e^{\Delta}S_{1/\sqrt{t}} \phi,
\end{equation}
where $S_{\lambda}$ is defined by:
\begin{equation}
(S_{\lambda}\phi)(x)= \phi(\frac{x}{\lambda}).
\end{equation}
Then also notice:
\begin{equation}
\|\partial^{\eta} |x|^{\beta} S_{\lambda} \phi \|_{\Lqqtilde} = \lambda^{\frac{n}{q} +\beta -|\eta|}\||x|^{\beta}\phi\|_{\Lqqtilde}.
\end{equation}
For each $u_{0} \in L^{p}_{|x|^{\alpha p}d|x|}L^{\p}_{\theta}$ and $t\in \Rpiu$ the function $e^{\Delta} S_{1/ \sqrt{t}}$ is in the Schwartz class, so we can apply proposition (\ref{cor:nonhom}) with arbitrarily high values of $\lambda$ and condition (\ref{eq:condabg}) is trivially satisfied.
We get:
\begin{eqnarray}
\||x|^{\beta} \partial^{\eta} e^{t\Delta} u_{0}\|_{\Lqqtilde} &=& 
\| |x|^{\beta}\partial^{\eta} S_{\sqrt{t}}e^{\Delta}S_{1/ \sqrt{t}} u_{0}\|_{\Lqqtilde} \nonumber \\
&=& t^{(\frac{n}{q} + \beta - |\eta|)/2} \||x|^{\beta} (\partial^{\eta} e^{\Delta})S_{1/\sqrt{t}} u_{0}\|_{\Lqqtilde} \nonumber \\
&\leq & \frac{c_{\eta}}{t^{(-\frac{n}{q} -\beta + |\eta |)/2}}\||x|^{\alpha}S_{1/ \sqrt{t}}u_{0}\|_{\Lpptilde} \nonumber \\
&=& \frac{c_{\eta}}{t^{(|\eta| + \frac{n}{p}-\frac{n}{q}+\alpha -\beta)/2}}\||x|^{\alpha} u_{0}\|_{\Lpptilde}, \nonumber
\end{eqnarray}
where we used the condition (\ref{eq:condDL}), i.e.
$$ 
\Lambda(\alpha,p,\p) \geq \Lambda(\beta,q,\q).
$$
The proof of (\ref{PGHeatDer}) is similar, but we have to work with the operator $e^{t\Delta} \mathbb{P} \nabla \cdot \ $. As seen in proposition \ref{OseenDecayFinale} it is a convolution with a kernel $K$ such that
\begin{equation}\label{OseenInProofScaling}
K_{j,k,m}(t) = K_{j,k,m} \left( \frac{x}{\sqrt{t}} \right),
\end{equation}
and 
\begin{equation}\label{OseenInProof}
(1+|x|)^{1 + n + |\mu |}\partial^{ \mu} K_{j,k,m} \in L^{\infty}(\Rn),
\end{equation}
for each multi index $\mu$. By the scaling (\ref{OseenInProofScaling}) follows:
\begin{equation}\label{OssenScaling}
K(t) * \phi =\frac{1}{\sqrt{t}} S_{\sqrt{t}} K * S_{1/\sqrt{t}} \phi.
\end{equation}
So we get
\begin{eqnarray}
\||x|^{\beta} \partial^{\eta} e^{t\Delta} \mathbb{P} \nabla \cdot F\|_{\Lqqtilde} &=&  
\| |x|^{\beta}  \partial^{\eta}K(t) * F\|_{\Lqqtilde}
\nonumber \\
&=& \frac{1}{\sqrt{t}}
\| |x|^{\beta}\partial^{\eta} S_{\sqrt{t}} K * S_{1/ \sqrt{t}} F\|_{\Lqqtilde}                     \nonumber \\
&=& \frac{1}{\sqrt{t}}t^{(\frac{n}{q} + \beta  - |\eta|)/2} \||x|^{\beta} (\partial^{\eta} K) * S_{1/ \sqrt{t}} F\|_{\Lqqtilde} \nonumber \\
&\leq & \frac{d_{\eta}}{t^{(-\frac{n}{q} -\beta +1 +  |\eta |)/2}}\||x|^{\alpha}S_{1/ \sqrt{t}} F\|_{\Lpptilde} \nonumber \\
&=& \frac{d_{\eta}}{t^{(1+ |\eta| + \frac{n}{p}-\frac{n}{q}+\alpha -\beta)/2}}\||x|^{\alpha} F\|_{\Lpptilde}, \nonumber
\end{eqnarray}
provided that $\Lambda_{\alpha} \geq \Lambda_{\beta}$. Notice that the optimal choice of $\gamma$ allowed by (\ref{OseenInProof}) is $\gamma = 1 + n + |\eta|$ and leads to
$$\alpha + \beta + 1 + n + |\eta| > n \Big( 1+\frac{1}{q}- \frac{1}{p} \Big) \Rightarrow 
1+ |\eta| + \frac{n}{p}-\frac{n}{q}+\alpha -\beta > 0.$$
\end{proof}
\noindent It's remarkable that the restriction $\Lambda_{\alpha} \geq \Lambda_{\beta}$ can be removed if we localize the estimate in a space-time parabola above the origin. The size of the parabola will depend by the values of the difference $\Lambda_{\alpha} - \Lambda_{\beta}$ and increases as $\Lambda_{\alpha}- \Lambda_{\beta} \rightarrow 0^{-}$. In the limit case $\Lambda_{\alpha}=\Lambda_{\beta}$ we infact recover proposition \ref{PDecayCor}.
\begin{proposition}\label{LocalPDecay}
Let $n\geq 2$, $1 < p \leq q < +\infty$ and $1 \leq \p \leq \q \leq +\infty$. Assume further that $\alpha, \beta$ satisfy the set of conditions 
 \begin{equation}\label{eq:condDL(Heat)Loc}
    \beta > -\frac nq,\qquad \alpha<\frac{n}{p'}, \qquad
    \Lambda (\alpha,p,\p) < \Lambda (\beta,q,\q),
    \end{equation}
and define:
$$
\Lambda_{\alpha,\beta}= \Lambda (\alpha,p,\p) - \Lambda (\beta,q,\q).
$$    
Let then $\Pi(R)$ the space-time parabola:
$$
\Pi(R) = \Big\{ \frac{|x|}{\sqrt{t}} < R \Big\};
$$   
then the following estimates hold:
\begin{enumerate}
         \item \begin{equation}\label{PHeatDerLoc}
     \| \mathbbm{1}_{\Pi(R)} |x|^{\beta} \partial^{\eta} e^{t\Delta}u_{0}\|_{L^{q}_{|x|}L^{\q}_{\theta}} 
     \leq \frac{c_{\eta}R^{-\Lambda_{\alpha,\beta}}}{t^{(|\eta| + \frac{n}{p}-\frac{n}{q} + \alpha-\beta)/2}} 
     \| |x|^{\alpha} u_{0}\|_{L^{p}_{|x|}L^{\p}_{\theta}}, \qquad t>0,
     \end{equation}
     provided that $|\eta| + \frac{n}{p}-\frac{n}{q} + \alpha-\beta \geq0$,
      \item \begin{equation}\label{PGHeatDerLoc}
       \| \mathbbm{1}_{\Pi(R)} |x|^{\beta} \partial^{\eta} e^{t\Delta} \mathbb{P} \nabla \cdot F\|_{L^{q}_{x}L^{\q}_{\theta}} 
       \leq \frac{d_{\eta}R^{-\Lambda_{\alpha,\beta}}}{t^{(1 + |\eta| + \frac{n}{p} -\frac{n}{q} +\alpha -\beta)/2}} \| |x|^{\alpha} F\|_{L^{p}_{|x|}L^{\p}_{\theta}}, \qquad t>0,
         \end{equation} 
provided that $1+ |\eta| + \frac{n}{p}-\frac{n}{q} + \alpha-\beta > 0$.
\end{enumerate}
for each muti index $\eta$, so in particular:
\begin{enumerate}
\item \begin{equation}\label{PHeatLoc}
     \| \mathbbm{1}_{\Pi(R)} |x|^{\beta} e^{t\Delta}u_{0}\|_{L^{q}_{|x|}L^{\q}_{\theta}} 
     \leq \frac{c_{0}R^{-\Lambda_{\alpha,\beta}}}{t^{(\frac{n}{p}-\frac{n}{q} + \alpha-\beta)/2}} 
     \|  |x|^{\alpha} u_{0}\|_{L^{p}_{|x|}L^{\p}_{\theta}}, \qquad t>0,
     \end{equation}
provided that $\frac{n}{p}-\frac{n}{q} + \alpha-\beta \geq0$,
    \item \begin{equation}\label{PGHeatLoc}
       \| \mathbbm{1}_{\Pi(R)} |x|^{\beta} e^{t\Delta} \mathbb{P} \nabla \cdot F\|_{L^{q}_{x}L^{\q}_{\theta}} 
       \leq \frac{d_{0}R^{-\Lambda_{\alpha,\beta}}}{t^{(1 + \frac{n}{p} -\frac{n}{q} +\alpha -\beta)/2}} \| |x|^{\alpha} F\|_{L^{p}_{|x|}L^{\p}_{\theta}}, \qquad t>0,
         \end{equation} 
provided that $1+ \frac{n}{p}-\frac{n}{q} + \alpha-\beta > 0$.
\end{enumerate}
These estimates can also be differently formulated by setting $R^{-\Lambda_{\alpha,\beta}}= K$. For instance (\ref{PHeatDerLoc}) becomes:
\begin{equation}\label{PHeatDerLoc2}
     \| \mathbbm{1}_{\Pi(K^{-1/\Lambda_{\alpha,\beta}})} |x|^{\beta} \partial^{\eta} e^{t\Delta}u_{0}\|_{L^{q}_{|x|}L^{\q}_{\theta}} 
     \leq \frac{c_{\eta}K}{t^{(|\eta| + \frac{n}{p}-\frac{n}{q} + \alpha-\beta)/2}} 
     \| |x|^{\alpha} u_{0}\|_{L^{p}_{|x|}L^{\p}_{\theta}}, \qquad t>0,
     \end{equation}
and similarly for the other estimates.
\end{proposition}
\begin{proof}
We use simply $\Lambda$ instead of $\Lambda_{\alpha,\beta}$. Of course:
$$
\Lambda < 0 \quad \Rightarrow \quad  R^{-\Lambda} \Big| \frac{x}{\sqrt{t}} \Big|^{\Lambda} \geq 1, \quad
\mbox{if} \quad (t,x)\in \Pi(R).
$$
Then we get, as in (\ref{PDecayCor}):
\begin{eqnarray}
\|\mathbbm{1}_{\Pi(R)} |x|^{\beta} \partial^{\eta} e^{t\Delta} u_{0}\|_{\Lqqtilde} &=& 
\|\mathbbm{1}_{\Pi(R)} |x|^{\beta}\partial^{\eta} S_{\sqrt{t}}e^{\Delta}S_{1/ \sqrt{t}} u_{0}\|_{\Lqqtilde} \nonumber \\
&\leq & \frac{R^{-\Lambda}}{t^{\Lambda/2}}\| |x|^{\beta +\Lambda}\partial^{\eta} S_{\sqrt{t}}e^{\Delta}S_{1/ \sqrt{t}} u_{0}\|_{\Lqqtilde} \nonumber \\
&=& \frac{R^{-\Lambda}}{t^{\Lambda /2}} t^{(\frac{n}{q} + \beta + \Lambda - |\eta|)/2} \||x|^{\beta +\Lambda} e^{\Delta}S_{1/\sqrt{t}} u_{0}\|_{\Lqqtilde} \nonumber \\
&\leq & \frac{c_{\eta}}{t^{(-\frac{n}{q} -\beta + |\eta |)/2}}\||x|^{\alpha}S_{1/ \sqrt{t}}u_{0}\|_{\Lpptilde} \nonumber \\
&=& \frac{c_{\eta}}{t^{(|\eta| + \frac{n}{p}-\frac{n}{q}+\alpha -\beta)/2}}\||x|^{\alpha} u_{0}\|_{\Lpptilde}, \nonumber
\end{eqnarray}
where the indexes relationships are consistent because:
$$
\Lambda_{\alpha} \geq \Lambda(\Lambda_{\alpha,\beta} +\beta,p,\p) = \Lambda(\Lambda_{\alpha}-\Lambda_{\beta} + \beta,p,\p)= \Lambda_{\alpha}.
$$
The other inequalities can be proved in the same way. 
\end{proof}
\begin{remark}
We precisely observed that the inequalities hold with an additional factor $R^{-\Lambda}$ after localization in a space-time parabola. Notice that this factor goes to $1$ as $\Lambda \rightarrow 0^{-}$. To get a uniformly in $\Lambda$ constant, it is instead necessary to restrict the size of the parabola: if we chose the constant equal to $K$, we need to restrict to:
$$
\left| \frac{x}{\sqrt{t}} \right| \leq K^{-\frac{1}{\Lambda}}.  
$$ 
Here as $\Lambda \rightarrow 0^{-}$ the parabola fills the whole space-time.
\end{remark}

\noindent By the time decay also integral estimates can be easily obtained:

\begin{proposition}\label{IDecayCor}
Let $n\geq 2$, $1 < p \leq q < \frac{np}{(|\eta| + \alpha-\beta)p + n-2}$, $r \in (1, +\infty)$ and $ 1 \leq \p \leq \q \leq +\infty$. Assume further that $\alpha, \beta$ satisfy the set of conditions 
 \begin{equation}\label{eq:condDL(IHeat)}
    \beta > -\frac nq,\qquad \alpha<\frac{n}{p'}.
   \end{equation}
The following estimates hold:
\begin{equation}\label{IHeat}
     \| |x|^{\beta} \partial^{\eta} e^{t\Delta}u_{0}\|_{L^{r}_{t}L^{q}_{|x|}L^{\q}_{\theta}} 
     \leq c_{\eta} 
     \| |x|^{\alpha} u_{0}\|_{L^{p}_{|x|}L^{\p}_{\theta}}, \qquad t>0,
     \end{equation}
for each multi index $\eta$, provided that:
     \begin{equation}\label{OmegaScaling}
  |\eta| + \Omega(\alpha, p, \infty) = \Omega (\beta,q,r), 
  \quad \Lambda (\alpha,p,\p) \geq \Lambda (\beta,q,\q).
  \end{equation}
  And the localized version
\begin{equation}\label{IHeatLoc}
     \| \mathbbm{1}_{\Pi(R)} |x|^{\beta} \partial^{\eta} e^{t\Delta}u_{0}\|_{L^{r}_{t}L^{q}_{|x|}L^{\q}_{\theta}} 
     \leq c_{\eta} R^{-\Lambda_{\alpha,\beta}} 
     \| |x|^{\alpha} u_{0}\|_{L^{p}_{|x|}L^{\p}_{\theta}}, \qquad t>0,
     \end{equation}
for each multi index $\eta$, provided that:
     \begin{equation}\label{OmegaScalingLoc}
  |\eta| + \Omega(\alpha, p, \infty) = \Omega (\beta,q,r), 
  \quad \Lambda_{\alpha,\beta} = \Lambda (\alpha,p,\p) - \Lambda (\beta,q,\q) < 0,  
  \end{equation}
  where we remember the definition of $\Pi(R)$: 
$$
\Pi(R) = \Big\{ \frac{|x|}{\sqrt{t}} < R \Big\}.
$$   
   \end{proposition}
\begin{proof}
By the punctual decay:
 \item \begin{equation}
     \||x|^{\beta} \partial^{\eta} e^{t\Delta}u_{0}\|_{L^{q}_{|x|}L^{\q}_{\theta}} 
     \leq \frac{c_{\eta}}{t^{(|\eta| + \frac{n}{p}-\frac{n}{q} + \alpha-\beta)/2}} 
     \| |x|^{\alpha} u_{0}\|_{L^{p}_{|x|}L^{\p}_{\theta}}, \qquad t>0,
     \end{equation}
follows that $\partial^{\eta}e^{t\Delta}u_{0}$ is bounded in the Lorentz space $L^{r,\infty}(\mathbb{R}^{+}; L^{q}_{|x|^{\beta q}d|x|}L^{\q}_{\theta})$ if $|\eta| + \Omega(\alpha, p, \infty) = \Omega (\beta,q,r)$. Infact:
\begin{eqnarray} 
\|\||x|^{\beta}\partial^{\eta}e^{t\Delta}u_{0}\|_{\Lqqtilde}\|_{L^{r,\infty}_{t}} &\leq &
   c_{\eta}  \left\| \frac{1}{t^{(|\eta| + \frac{n}{p}-\frac{n}{q} + \alpha-\beta)/2}} 
     \| |x|^{\alpha} u_{0}\|_{L^{p}_{|x|}L^{\p}_{\theta}}\right\|_{L^{r,\infty}_{t}} \nonumber \\  
&\leq & c_{\eta} \left\|  \frac{1}{t^{(|\eta| + \frac{n}{p}-\frac{n}{q} + \alpha-\beta)/2}} \right\|_{L^{r,\infty}}
        \|u_{0}\|_{\Lpptilde}  \nonumber \\
        &\leq & c_{\eta} \|u_{0}\|_{\Lpptilde}, \nonumber
\end{eqnarray}     
provided that:
 $$
 (|\eta| + \frac{n}{p} -\frac{n}{q} +\alpha -\beta)/2 = \frac{1}{r} \Rightarrow |\eta| + \Omega(\alpha, p, \infty) = \Omega (\beta,q,r).
 $$
Let now consider a couple $(\alpha_{0},\beta_{0}, p_{0},\p_{0},q_{0},\q_{0},r_{0})$, $(\alpha_{1},\beta_{1}, p_{1},\p_{1},q_{1},\q_{1},r_{1})$ such that the assumptions of theorem are satisfied, and in particular (\ref{eq:condDL(IHeat)}, \ref{OmegaScaling}) holds. Then we have the couples of linear operators:
\begin{equation}
    \partial^{\eta} e^{t\Delta} : 
    \begin{array}{ccc}
     \Lpptildealphazero & \longrightarrow & L^{r_{0},\infty}_{t} \Lqqtildebetazero \\
     & & \\
     & & \\
     \Lpptildealphauno & \longrightarrow & L^{r_{1},\infty}_{t} \Lqqtildebetauno .  
      \end{array}
\end{equation}
if we fix $\xi \in (0,1)$ we can perform real interpolation of operators with parameters $(\xi, r)$, provided that:
\begin{equation}\label{rConstr}
p_{\xi} < r_{\xi},
\end{equation}
where
$$
\frac{1}{p_{\xi}} = (1- \xi)\frac{1}{p_{0}} + \frac{1}{\xi p_{1}},
$$  
in the same way are defined $q_{\xi},r_{\xi},\q_{\xi},\p_{\xi}$, while
$$
\alpha_{\xi}= (1-\xi)\alpha_{0} +\xi \alpha_{1},
$$
and the same for $\beta_{\xi}$. So we get the bounded operators:
$$
\partial^{\eta}e^{t\Delta} u_{0} :
$$
$$
 \Lpptildealphaxi \rightarrow \Big(L^{r_{0},\infty}_{t} \Lqqtildebetazero, L^{r_{1},\infty}_{t} \Lqqtildebetauno \Big)_{\xi,r} = L^{r}_{t} \Lqqtildebetaxi. 
$$
It is now straightforward to check that indeces $(\alpha_{\xi},\beta_{\xi}, \mbox{etc})$ satisfy the relations (\ref{eq:condDL(IHeat)}, \ref{OmegaScaling}) and other assumptions. Furthetmore constrains (\ref{OmegaScaling}) and (\ref{rConstr}) are equivalent to $q_{\xi} < \frac{np_{\xi}}{(|\eta| + \alpha_{\xi} - \beta_{\xi}) p_{\xi} + n-2}$. Of course this method misses the endpoints $p,q=1$.
The estimates (\ref{IHeatLoc}) are  proved in the same way by using the localized punctual decay. 

\end{proof}

\noindent The estimates of the Duhamel term needs no interpolation notions:
\begin{proposition}\label{DDecayCor}
Let $n\geq 2$, $1 < p \leq 2q < +\infty$, $1 < s \leq 2r < +\infty $ and $1 \leq p \leq 2q \leq +\infty$. Assume further that $\alpha, \beta$ satisfy the set of conditions 
 \begin{equation}\label{eq:condDL(DHeat)}
    \beta > -\frac nq,\qquad \alpha<\frac{n}{p'},\qquad
   2 \Lambda (\alpha,p,\p) \geq \Lambda (\beta,q,\q)
  \end{equation}
Then the following estimates holds:
   \begin{equation}\label{DGHeat}
       \left\||x|^{\beta} \partial^{\eta} \int_{0}^{t} e^{(t-s)\Delta} \mathbbm{P} \nabla \cdot 
       (u \otimes u) \ ds \right\|_{L^{r}_{t}L^{q}_{x}L^{\q}_{\theta}} 
       \leq  d_{\eta} \| |x|^{\alpha} 
       u \|^{2}_{L^{s}_{t}L^{p}_{|x|}L^{\p}_{\theta}}, 
         \end{equation} 
for each multi index $\eta$, provided that:
       \begin{equation}\label{DGOmegaScaling}
  2 \Omega(\alpha, p, s) = \Omega (\beta,q,r) + 1 - |\eta |.
       \end{equation}
In particular holds:
   \begin{equation}\label{DGHeatq=p}
       \left\||x|^{\beta} \partial^{\eta} \int_{0}^{t} e^{(t-s)\Delta} \mathbbm{P} \nabla 
       \cdot (u \otimes u) \  ds\right\|_{L^{r}_{t}L^{q}_{x}L^{\q}_{\theta}} 
       \leq  d_{\eta} \| |x|^{\beta} 
       u \|^{2}_{L^{r}_{t}L^{q}_{|x|}L^{\q}_{\theta}}, 
         \end{equation} 
for each multi index $\eta$, provided that:
         \begin{equation}\label{DGOmegaScalingq=p}
          \frac 2r + \frac nq   = 1 - \beta - |\eta |, 
         \qquad \Lambda(\beta,q,\q) \geq 0.
         \end{equation}
The range of admissible $p,q$ indices can be relaxed to $1\leq p \leq q \leq +\infty$ provided that $2\Lambda (\alpha,p,\p) > \Lambda (\beta,q,\q)$.         
\end{proposition}

\begin{proof}
By Minkowsky inequality and estimates (\ref{PHeatDer}):
\begin{eqnarray}
  & &\left\||x|^{\beta} \partial^{\eta} \int_{0}^{t} e^{(t-s)\Delta}\mathbb{P} 
   \nabla \cdot F(x,s) \ ds \right\|_{\Lrqqtilde}   \nonumber \\ 
  & \leq &\left\| \int_{\Rpiu} \| |x|^{\beta} \partial^{\eta} e^{(t-s)\Delta}\mathbb{P} 
   \nabla \cdot F \|_{\Lqqtilde} \ ds\right\|_{L^{r}_{t}}  \nonumber \\
  &\leq & d_{\eta} \left\| \int_{\mathbb{\Rpiu}} 
   \frac{1}{(t-s)^{(1 + |\eta | 
   +\frac{n}{p_{0}} -\frac{n}{q} +\alpha_{0} -\beta)/2}}\||x|^{\alpha_0}F \|_{\L^{q_0}_{|x|}L^{\q_0}_{\theta}}   
   \ ds  \right\|_{L^{r}_{t}},   \nonumber 
\end{eqnarray}
provided that:
\begin{equation}\label{DuamhCondRemark}
\p_0 \leq \q, \quad p_0 \leq q \quad 1 + |\eta | + \frac{n}{p_{0}} - \frac{n}{q_{0}} + \alpha_{0} - \beta > 0, \quad \Lambda_{\alpha_0} \geq \Lambda_{\beta}.
\end{equation}
Let then
\begin{equation}\label{WYScalingInCor}
1+ \frac{1}{r} = \frac{1}{s_{0}} + \frac{1}{k},
\end{equation}
and use the Young inequality in the following Lorentz Spaces:
$$
\| \cdot \|_{L^{r}} \leq \| \cdot \|_{L^{s_{0}}} \| \cdot \|_{L^{k,\infty}},
$$
that's allowed if $1<r,s_{0},k <+\infty$. It is assured by assumptions $1< r,s < +\infty$. We get
\begin{eqnarray}
  & &\left\||x|^{\beta} \partial^{\eta} \int_{0}^{t} e^{(t-s)\Delta}\mathbb{P} 
   \nabla \cdot F(x,s) \ ds \right\|_{\Lrqqtilde}   \nonumber \\ 
   &\leq & d_{\eta} \||x|^{\alpha_0} F\|_{L^{s_{0}}_{t}L^{p_{0}}_{|x|}L^{\p_{0}}_{\theta}} 
   \left\|\int_{\mathbb{\Rpiu}} \frac{dt}{t^{(1+|\eta | +\frac{n}{p_{0}} 
   -\frac{n}{q_{0}} +\alpha_{0}-\beta)/2}}\right\|_{L^{k,\infty}_{t}}   \nonumber \\
   & \leq &\ d_{\eta} \||x|^{\alpha_0} 
 F\|_{L^{s_{0}}_{t}L^{p_{0}}_{|x|}L^{\p_{0}}_{\theta}}, \nonumber 
\end{eqnarray}
provided that:
\begin{equation}\label{DuhamProvvScal}
p_0 \leq q \quad (1 + |\eta | + \frac{n}{p_{0}} - \frac{n}{q_{0}} + \alpha_{0} - \beta)/2=\frac{1}{k}, \quad \Lambda_{\alpha_0} \geq \Lambda_{\beta},
\end{equation}
because of
$$
\left\| \int_{\Rpiu} \frac{dt}{t^{1/k}}\right\|_{L^{k,\infty}_{t}} =1.
$$
Conditions (\ref{DuhamProvvScal}), (\ref{WYScalingInCor}) lead to:
\begin{equation}\label{DuhamProvvScal2}
\Omega(\alpha_{0}, p_{0}, s_{0}) = 1-|\eta | +\Omega(\beta,q,r).
\end{equation}
We now specify $F=u \otimes u$:
\begin{eqnarray}
  & &\left\||x|^{\beta} \partial^{\eta} \int_{0}^{t} e^{(t-s)\Delta}\mathbb{P} 
   \nabla \cdot (u\otimes u)(x,s)  \ ds \right\|_{\Lrqqtilde}   \nonumber \\ 
  &\leq &  c_{\eta} \||x|^{\alpha_{0}} 
   |u|^{2}\|_{L^{s_{0}}_{t}L^{p_{0}}_{|x|}L^{\p_{0}}_{\theta}}  \nonumber \\
 &\leq &  c_{\eta} \||x|^{\alpha_{0}/2} 
   |u|\|^{2}_{L^{2s_{0}}_{t}L^{2p_{0}}_{|x|}L^{2\p_{0}}_{\theta}} \\
 & \leq  &c_{\eta} \||x|^{\alpha}u\|^{2}_{\Lspptilde}, \nonumber
  \end{eqnarray}
where we have set 
\begin{equation}\label{DuhamPivot}
(\alpha_{0}/2, 2s_{0},2p_{0},2\p_{0})=(\alpha,s,p,\p).
\end{equation}
Notice that $2\Omega(\alpha,s,p)= \Omega(\alpha_{0},s_{0},p_{0})$, so (\ref{DuhamProvvScal2}) and (\ref{DuhamPivot}) lead to (\ref{DGOmegaScaling})
$$
2\Omega(\alpha,p,s) = \Omega(\beta,q,r) +1 - |\eta |;
$$
while condition (\ref{DuhamProvvScal}) leads to
$$
2\Lambda_{\alpha} \geq \Lambda_{\beta}.
$$
Finally notice that (\ref{WYScalingInCor}) and (\ref{DuhamProvvScal}) imply the relationships:
$$
r \geq s_{0}=s/2, \qquad q \geq p_{0}=p/2, \qquad \q \geq \p_{0}=\p/2.
$$
These conditions are furthermore consistent with the choice $(\alpha,s,p,\p)=(\beta,r,q,\q)$; in such a way we recover inequality (\ref{DGHeatq=p}):
$$
\left\||x|^{\beta} \partial^{\eta} \int_{0}^{t} e^{(t-s)\Delta} \mathbbm{P} \nabla 
       \cdot (u \otimes u) \  ds \right\|_{L^{r}_{t}L^{q}_{x}L^{\q}_{\theta}} 
       \leq  d_{\eta} \| |x|^{\beta} 
       u \|^{2}_{L^{r}_{t}L^{q}_{|x|}L^{\q}_{\theta}},
$$
provided that
$$
\Omega(\beta,q,r) = 1-|\eta|, \qquad \Lambda(\beta,q,\q) \geq 0.
$$

\end{proof}

\section{Regularity criteria in weighted spaces with angular integrability}



As told the technology developed until now is well suited to study the regularity property of weak solutions of (\ref{CauchyNS}) in weighted spaces. In particular we focus on theorem (\ref{YZTheorem}) in which is shown the regularity of a weak solution $u$ in the segment $(0,T)\times \{ 0 \}$ provided that is satisfied the a priori bound:

\begin{equation}\label{moreOrLess}
\||x|^{\alpha} u \|_{L^{s}_{T}L^{p}_{x}} \leq +\infty, 
\quad \frac{2}{s} + \frac{n}{p} = 1- \alpha.
\end{equation} 
Of course such norms are invariant with respect the natural scaling of (\ref{CauchyNS}) centred in the origin:

\begin{equation*}
\lambda \rightarrow \lambda u (\lambda^{2} t, \lambda x).
\end{equation*}

\ni Now we are interested in what happens if we consider more or less angular integrability in (\ref{moreOrLess}).
We have results in two different directions:

\ni If we consider weights $|x|^{\alpha}$ with $\alpha < 0$:
\begin{itemize}
\item  we get regularity in the segment $(0,T)\times \{ 0 \}$ as in (\ref{YZTheorem}), but by requiring boundedness in $L^{s}_{t}L^{p}_{|x|^{\alpha p}d|x|}L^{\p}_{\theta}$ with $\p < p$; 
\item we actually get global regularity (in $(0,T)\times \Rn$) by requiring boundedness in $L^{s}_{t}L^{p}_{|x|^{\alpha p}d|x|}L^{\p}_{\theta}$ with $\p > p$ and $\p$ large enough.  
\end{itemize}
In the case $|x|^{\alpha}$, $\alpha \geq 0$ we show that:
\begin{itemize}
\item assumptions in (\ref{YZTheorem}) actually lead to global regularity (in $(0,T)\times \Rn$);
\item global regularity is achieved also if $\p < p$ with $\p$ large enough; precisely $\p_{G} \leq \p < p$, with $\p_{G}$ depending on the other parameters;
\item regularity in $(0,T)\times \{ 0 \}$ is achieved in the range $\p_{L} \leq \p < \p_{G}$, with $\p_{L}$ depending on the other parameters. 
\end{itemize}
The notation $\p_{L},\p_{G}$ would remind to $\p_{Local}, \p_{Global}$.
Such a scheme is explained by the following heuristic: if $\alpha < 0$ the weight 
$|x|^{\alpha}$ localizes in some sense the norm near to the origin, so local regularity is expectable, also for $\p < p$, but some boundedness condition at infinity are necessary in order to get global regularity. Such a condition is actually high $L^{\p}$ integrability in the angular direction. If we consider instead $|x|^{\alpha}$ with $\alpha \geq 0$, the weight suffices to guarantee also boundedness for large $|x|$, furthermore the integrability in the angular direction can be relaxed to small values of $\p$. In this spirit we prove two extensions of theorem (\ref{YZTheorem}). In the first under the hypothesis of higher angular integrability we get global regularity. In the second we get regularity in the segment $(0,T)\times \{ 0 \}$, as in (\ref{YZTheorem}), but with weaker angular integrability. It is convenient to introduce the quantity:
\begin{equation}\label{PtildeDef}
\p_{L} = \left \{
\begin{array}{lcr}
 \frac{2(n-1)p}{(2 \alpha +1)p +2(n-1)}  &  \mbox{if} & -\frac{1}{2} \leq \alpha < 0  \\
 && \\
 \frac{2(n-1)p}{p +2(n-1)}  &  \mbox{if} & 0 \leq \alpha < 1, 
\end{array}\right.
\end{equation}
\begin{equation}\label{PtildeGDef}
\p_{G} = \left \{
\begin{array}{lcr}
 \max \left( 4, \frac{(n-1)p}{\alpha p + n-1} \right)  &  \mbox{if} & \frac{1-n}{2} \leq \alpha < 0  \\
 && \\
  \frac{(n-1)p}{\alpha p + n-1}    &  \mbox{if} & 0 \leq \alpha \leq \frac{1}{2}, 
\end{array}\right.
\end{equation}

\ni as told $\p_{L}, \p_{G}$ would remind to $\p_{Local}, \p_{Global}$. This quantities are infact sufficiently high angular integrability in order to to get respectively local (in $(0,t)\times \{ 0 \}$) and global regularity. Notice that
$$
\p_{L} < \p_{G}, \quad \mbox{if} \quad \alpha <  1/2, \qquad
\p_{L} = \p_{G} \quad \mbox{if} \quad \alpha = 1/2;
$$
and
\begin{equation}\label{pLpGalphaPositivo}
\p_{L} < p < \p_{G}, \qquad \mbox{if} \quad \alpha <0, 
\end{equation}
\begin{equation}\label{pLpGalphaNegativo}
\p_{L} < \p_{G} < p, \qquad \mbox{if} \quad \alpha > 0;
\end{equation}
The range of $\alpha$ in which $\p_{L}, \p_{G}$ have been defined is justified in the following proofs.

\

\ni The following global regularity criterion holds:
\begin{theorem}\label{OurYZTheorem}
Let $n \geq 3$, $u_{0} \in L^{2}_{\sigma}(\Rn)$ and $u$ a weak Leray solution of (\ref{CauchyNS}).

\

\ni If ${\bf \alpha \in [(1-n)/2,0)}$, $\alpha_{0} \in [(2-n)/2, 2/(2+n))$, $\frac{n}{1- \alpha} < p \leq \frac{1-n}{\alpha}$, and 
   $$  
   \| |x|^{\alpha_{0}} u_{0}\|_{L^{p_{0}}_{|x|}L^{\p_{0}}_{\theta}} < +\infty
   $$ 
with:
\begin{equation}\label{OurYZCond0}
\alpha_{0} = 1- \frac{n}{p_{0}}, \quad 
\p_{0} \leq \frac{\p_{G}}{2},
\end{equation}
\begin{equation}\label{OurYZCond0Complicata}
\left \{
\begin{array}{lcr}
2 \leq p_{0} \leq \p_{G}/2   &  \mbox{if} & \p_{G} \leq 2n   \\
2 \leq p_{0} \leq \p_{G}/2, \quad p_{0} < \frac{2\p_{G}}{\p_{G} - 2n}    &  \mbox{if} & \p_{G} > 2n; 
\end{array}\right.
\end{equation}
and
\begin{equation}\label{OurYZuBound1}
\| |x|^{\alpha} u \|_{L^{s}_{T}L^{p}_{|x|}L^{\p}_{\theta}} < +\infty,
\end{equation}
with:
\begin{equation}\label{OurYZCondition1}
\frac{2}{s}+ \frac{n}{p} = 1-\alpha,
\end{equation} 
\begin{equation}\label{OurYZCondition2} 
\frac{2}{1-\alpha} < s < +\infty,
\end{equation}
\begin{equation}\label{OurYZCondition3}
\p \geq \p_{G} = \max \left(4, \frac{(n-1)p}{\alpha p +n -1}\right);
\end{equation}
then actually $u$ is regular ($C^{\infty}$ in space variables) in $(0,T) \times \Rn$.

\

\ni The same holds if ${\bf \alpha \in [0,1/2]}$, $\alpha_{0} \in [(2-n)/2, 2/(2+n))$, $\max \left( 4, \frac{n}{1 - \alpha} \right) < p < +\infty$, or $p=4$, and  
$$  
\| |x|^{\alpha_{0}} u_{0}\|_{L^{p_{0}}_{|x|}L^{\p_{0}}_{\theta}} < +\infty
$$ 
with:
\begin{equation}\label{OurYZCond0bis}
\alpha_{0} = 1- \frac{n}{p_{0}}, \quad 
\p_{0} \leq \frac{p}{2},
\end{equation}
\begin{equation}\label{OurYZCond0ComplicataBis}
\left \{
\begin{array}{lcr}
2 \leq p_{0} \leq p/2   &  \mbox{if} & p \leq 2n   \\
2 \leq p_{0} \leq p/2, \quad p_{0} < \frac{2 p}{p - 2n}    &  \mbox{if} & p > 2n; 
\end{array}\right.
\end{equation}
and 
\begin{equation}\label{OurYZuBound1bis}
\| |x|^{\alpha} u \|_{L^{s}_{T}L^{p}_{|x|}L^{\p}_{\theta}} < +\infty,
\end{equation}
with:
\begin{equation}\label{OurYZCondition1bis}
\frac{2}{s}+ \frac{n}{p} = 1-\alpha,
\end{equation} 
\begin{equation}\label{OurYZCondition2bis} 
\frac{2}{1-\alpha} < s < +\infty,
\end{equation}
\begin{equation}\label{OurYZCondition3bis}
\p \geq \p_{G} = \frac{(n-1)p}{\alpha p +n -1}.
\end{equation}
\end{theorem}
\begin{remark} 
By relations (\ref{pLpGalphaPositivo}, \ref{pLpGalphaNegativo}) it turns out that in the case of localizing weights ($|x|^{\alpha}, \alpha <0$) additional angular integrability is requested in order to get global regularity. We will come back on this point also in the next section. On the other hand the weights $|x|^{\alpha}, \alpha > 0$, by providing an additional boundedness at infinity, allows to get global regularity with less angular integrability.  
\end{remark}

\ni Of course by translations $u$ is regular provided that the weights and the norms are centered in any $\bar{x}$.

\begin{remark}
We give some remarks about the indexes:

\begin{itemize}
\item The first conditions in (\ref{OurYZCond0}, \ref{OurYZCond0bis}) and the conditions (\ref{OurYZCondition1}, \ref{OurYZCondition1bis}) follow by requiring invariance with respect to the natural scalings for respectively $u_{0}$ and $u$, i.e. 
$$
\lambda \rightarrow \lambda u_{0}(\lambda x)
$$ 
$$
\lambda \rightarrow \lambda u(\lambda^{2}, \lambda x);
$$
\item Our method misses the endpoints $s=+\infty$, $p=+\infty$. The last one is recovered just in case $\alpha \geq 0$, if we assume $\p > \p_{G}$. This is because of the use of proposition \ref{DDecayCor};
\item Of course we can set $T=+\infty$ to get regularity for all times;
\item If $n > 3$ we get a gain in the negative values of $\alpha$ with respect to the theorem \ref{YZTheorem}. We have infact $\frac{1-n}{2} \leq \alpha$ instead of $-1 \leq \alpha$. This is also more satisfactory because exbits a dependence on the dimension. We have, on the other hand, a loss on the positive values of $\alpha$, that's $\alpha \leq \frac{1}{2}$ instead of $\alpha < 1$;
\item It is interesting to notice that no direct correlation between the radial and angular integrability of the initial datum have to be required. 
\end{itemize}
\end{remark}

\

\ni The assumptions on the initial datum are a bit complicated, and have to be considered as merely technical. For instance can be considered $u_{0}$ in the Schwartz class without a real loss in the main contents of the theorem. In this way the formulation becomes simpler. The real information is about the angular integrability of the solution requested in order to get regularity, so it is the hypotesis $p \geq \p_{G}$. It comes out by requiring $\Lambda (\alpha,p,\p) \geq 0$ in order to apply inequality (\ref{DGHeat}) with simply $L^{r}_{t}L^{q}_{x}$ spaces on the left side.
 
\

\begin{proof}
We want to use the regularity criterion (\ref{SerrinRegularity}), so we need to show:
\begin{equation}\label{OurYZCondition1Proof}
\|u\|_{L^{r}_{T}L^{q}_{x}} < +\infty, \quad \mbox{with} \quad \frac{2}{r} +\frac{n}{q} =1.
\end{equation}
Let's start by the integral representation
\begin{equation}\nonumber
u  =  e^{t \Delta}u_{0} + \int_{0}^{t}e^{(t-s)\Delta}\mathbb{P}\nabla \cdot (u \otimes u)(s)ds. 
\end{equation}

\ni In order to get (\ref{OurYZCondition1Proof}) we distinguish the cases $\alpha \in [(1-n)/2,0)$ and $\alpha \in [0,1/2]$.
\ni \subsection*{Case ${\bf \alpha \in [(1-n)/2,0)}$} 
\begin{eqnarray}\nonumber
\|u\|_{L^{r}_{T}L^{q}_{x}} &\leq & \|e^{t\Delta}u_{0}\|_{L^{r}_{T}L^{q}_{x}}+
\left\|\int_{0}^{t}e^{(t-s)\Delta}\mathbb{P}\nabla \cdot (u \otimes u)(s) \ ds \right\|_{L^{r}_{T}L^{q}_{x}} \\ \nonumber
&=& I + II.
\end{eqnarray}
By the scaling assumption and the proposition (\ref{IDecayCor}) we get:
\begin{equation}\label{I}
I \leq c_{0} \||x|^{\alpha_{0}}u_{0}\|_{L^{p_{0}}_{|x|}L^{\p_{0}}_{\theta}}, 
\end{equation}
provided that 
\begin{equation}\label{I2}
  p_{0} \leq q < \frac{np_{0}}{p_{0} - 2}, \quad \p_{0} \leq q,
   \quad \Lambda_{\alpha_{0}, p_{0}, \p_{0}} \geq 0.
\end{equation}
Actually the condition $\Lambda_{\alpha_{0}, p_{0}, \p_{0}} \geq 0$ is not necessary in order to prove the theorem. We assume it for now to avoid some technicalities in the proof. We will show at the end how it can be removed.
To bound $II$ we use proposition (\ref{DDecayCor}) and scaling, so
\begin{equation}\nonumber
II \leq  d_{0} \||x|^{\alpha}u\|^{2}_{L^{s}_{T}\Lpptilde}, 
\end{equation}
provided that
\begin{equation}\label{OurYZCond2Proof}
\Lambda(\alpha,p, \p) \geq 0,
\end{equation}
\begin{equation}\label{OurYZCond3Proof}
p/2, \ \p /2 \leq q, \quad s/2 \leq r.
\end{equation}
Condition (\ref{OurYZCond2Proof}) is ensured by
$$
\p \geq \frac{(n-1)p}{\alpha p +n -1}.
$$
Notice also that this condition, the scaling relation and $\alpha <0$ imply $\frac{n}{1-\alpha} < p \leq \frac{1-n}{\alpha}$. So the widest range for $p$ is atteined as $\alpha \rightarrow 0^{-}$.
The last point is to show that a couple $(r,q)$ such that (\ref{OurYZCond3Proof}) is consistent with $\frac{2}{r} + \frac{n}{q}=1$ actually exists. We choose $q = \p_{G} = \max \left( 4, \frac{(n-1)p}{\alpha p + n-1} \right)$. 
This is allowed by $(1-n)/2 \leq \alpha$; infact if $\frac{(n-1)p}{\alpha p + n-1} \geq 4$ we get
$$
\frac{2}{r} = 1- \frac{n}{q} = 1- \frac{2n \alpha}{n-1} + \frac{2n}{p} \Rightarrow 
\frac{2}{r} - \frac{4}{s} = \frac{1-n -2 \alpha}{n-1},
$$
so
$$
(1-n)/2 \leq \alpha \Rightarrow s/2 \leq r; 
$$
while if $\frac{(n-1)p}{\alpha p + n-1} < 4$ we get also $p < 4$ and
$$
\frac{2}{r} = 1- \frac{n}{4}, \ \frac{4}{s} > 2- \frac{2n}{p} \Rightarrow
\frac{2}{r} - \frac{4}{s} < -1 +\frac{2n}{p} - \frac{n}{4} < 0.   
$$

\ni Finally the condition (\ref{I2}) becomes 
$$
p_{0} \leq \frac{\p_{G}}{2} < \frac{np_{0}}{p_{0}-2},
$$
that by a straightforward calculation leads to (\ref{OurYZCond0Complicata}), $\alpha_{0} \in [(2-n)/2, 2/(2+n))$.

\subsection*{Case ${\bf \alpha \in [0,1/2]}$} 
The only difference is in the choice of $(r,q)$. Here we set $q = p/2$. In such a way the condition (\ref{OurYZCond3Proof}) is ensured by $\alpha \leq 1/2$, infact:
$$
\frac{2}{r} = 1- \frac{2n}{p} \Rightarrow \frac{2}{r} - \frac{4}{s} = -1 + 2 \alpha,
$$  
so
$$
\alpha \leq 1/2 \Rightarrow s/2 \leq r.
$$
Notice that now we haven't the restriction $p\leq \frac{1-n}{\alpha}$.
Then the condition (\ref{I2}) becomes
$$
p_{0} \leq \frac{q}{2} < \frac{np_{0}}{p_{0}},
$$
that by a straightforward calculation leads to (\ref{OurYZCond0ComplicataBis}), $\alpha_{0} \in [(2-n)/2, 2/(2+n))$ and $\max \left( 4, \frac{n}{1 - \alpha} \right) < p$ or $p=4$.

\ni We show how the condition $\Lambda_{\alpha_{0}, p_{0}, \p_{0}} \geq 0$ can be removed. Let call it simply $\Lambda$ and suppose $\Lambda < 0 $.  We still can use the localized estimate (\ref{IHeatLoc}) to get the bound
$$
\|\mathbbm{1}_{\Pi(R)} u \|_{L^{r}_{T}L^{q}_{x}} \leq 
R^{-\Lambda} c_{0} \||x|^{\alpha_{0}} u_{0} \|_{L^{p_{0}}_{|x|}L^{\p_{0}}_{\theta}} +
d_{0} \||x|^{\alpha} u \|_{L^{s}_{T}L^{p}_{|x|}L^{\p}_{\theta}}
$$
where $\Pi(R)$ is the parabola
$$
\Pi(R) = \left\{ (t,x) \in \Rpiu \times \Rn \quad \mbox{s.t.} \quad
\left| \frac{x}{\sqrt{t}}\right| \leq R \right\}.
$$
So regularity is achieved in $\Pi(R)$ and, as $R \rightarrow +\infty$, in the whole space-time.

\end{proof}

\ni Then the following local regularity criterion holds:

\begin{theorem}\label{OurYZTheoremLoc}
Let $n \geq 3$, $u_{0} \in L^{2}_{\sigma}(\Rn) \cap H^{1}(\Rn) \cap L^{2}_{|x|^{2-n}dx}$. Let $u$ a weak Leray's solution of (\ref{CauchyNS}).

\

\ni If ${\bf \alpha \in [-1/2,0)}$, $\alpha_{0} \in \left[ 1-n, \frac{2-n}{2+n}\right)$, $\max \left(2, \frac{n}{1- \alpha} \right) < p < +\infty$ or $p=2$, and  
   $$  
   \| |x|^{\alpha_{0}} u_{0}\|_{L^{p_{0}}_{|x|}L^{\p_{0}}_{\theta}} < +\infty
   $$ 
with:
\begin{equation}\label{OurYZCond0Loc}
\alpha_{0} = 1- \frac{n}{p_{0}}, \quad 
\p_{0} \leq \frac{p}{2},
\quad \Lambda_{\alpha_{0}, p_{0}, \p_{0}} \geq 0;
\end{equation}
\begin{equation}\label{OurYZCond0ComplicataLoc}
\left \{
\begin{array}{lcr}
1 \leq p_{0} \leq p/2   &  \mbox{if} & p \leq n   \\
1 \leq p_{0} \leq p/2, \quad p_{0} < \frac{p}{p - n}    &  \mbox{if} & p > n; 
\end{array}\right.
\end{equation}
and
\begin{equation}\label{OurYZuBound1Loc}
\| |x|^{\alpha} u \|_{L^{s}_{T}L^{p}_{|x|}L^{\p}_{\theta}} < +\infty,
\end{equation}
with:
\begin{equation}\label{OurYZCondition1Loc}
\frac{2}{s}+ \frac{n}{p} = 1-\alpha,
\end{equation} 
\begin{equation}\label{OurYZCondition2Loc} 
\frac{2}{1-\alpha} < s < +\infty, 
\end{equation}
\begin{equation}\label{OurYZCondition3Loc}
\p \geq \p_{L} = \frac{2(n-1)p}{(2\alpha +1)p + 2(n-1)};
\end{equation}
then actually $u$ is and regular ($C^{\infty}$ in space variables) in the segment $(0,T) \times \{ 0 \}$.

\

\ni The same holds if ${\bf \alpha \in [0,1)}$, $\alpha_{0} \in \left[ 1-(1-\alpha)n, 1- (1-\alpha)\frac{2n}{2+n} \right)$, $\max \left( 2, \frac{n}{1 - \alpha} \right) < p < +\infty$ or $p=2$, and  
$$  
\| |x|^{\alpha_{0}} u_{0}\|_{L^{p_{0}}_{|x|}L^{\p_{0}}_{\theta}} < +\infty
$$ 
with:
\begin{equation}\label{OurYZCond0bisLoc}
\alpha_{0} = 1- \frac{n}{p_{0}}, \quad 
\p_{0} \leq \frac{p}{2}, \quad
\Lambda_{\alpha_{0},p_{0},\p_{0}} \geq 0,
\end{equation}
\begin{equation}\label{OurYZCond0ComplicataBisLoc}
\frac{1}{1-\alpha} \leq p_{0} \leq \frac{p}{2}, \quad p_{0} < \frac{p}{(1-\alpha)p -n};
\end{equation}
and 
\begin{equation}\label{OurYZuBound1bisLoc}
\| |x|^{\alpha} u \|_{L^{s}_{T}L^{p}_{|x|}L^{\p}_{\theta}} < +\infty,
\end{equation}
with:
\begin{equation}\label{OurYZCondition1bisLoc}
\frac{2}{s}+ \frac{n}{p} = 1-\alpha,
\end{equation} 
\begin{equation}\label{OurYZCondition2bisLoc} 
\frac{2}{1-\alpha} < s < +\infty, 
\end{equation}
\begin{equation}\label{OurYZCondition3bisLoc}
\p \geq \p_{L} = \frac{2(n-1)p}{p + 2(n-1)}.
\end{equation}
\end{theorem}

\ni Of course by translations $u$ is regular in the segment $(0,T) \times \{ \bar{x} \}$, provided that all the norms and the weights are centered in $\bar{x}$.

\begin{remark}
We give again some remarks about the indexes:

\begin{itemize}
\item The first conditions in (\ref{OurYZCond0Loc}, \ref{OurYZCond0bisLoc}) and the conditions (\ref{OurYZCondition1Loc}, \ref{OurYZCondition1bisLoc}) follow by requiring invariance with respect to the natural scalings for respectively $u_{0}$ and $u$, i.e. 
$$
\lambda \rightarrow \lambda u_{0}(\lambda x)
$$ 
$$
\lambda \rightarrow \lambda u(\lambda^{2}, \lambda x);
$$
\item Of course we can set $T=+\infty$ to get regularity for all times;
\item We get a loss in the negative values of $\alpha$ with respect to the theorem \ref{YZTheorem}. We have infact $-\frac{1}{2} \leq \alpha$ instead of $-1 \leq \alpha$.
\end{itemize}
\end{remark}

\

\ni We make again complicated assumptions on the initial datum, but the main point is the angular integrability of the solution requested in order to get regularity in $(0,T)\times \{ 0 \}$, i.e. the hypotesis $p \geq \p_{L}$. It comes out by requiring $\Lambda (\alpha,p,\p) \geq \beta$ with $\beta \in [-1, 1)$. This is because in such a way (\ref{DGHeat}) allows to apply directly the theorem \ref{YZTheorem}. 

\

\begin{proof}
Here we want to use directly the theorem (\ref{YZTheorem}), so we need to show:
\begin{equation}\label{OurYZCondition1ProofLoc}
\| |x|^{\beta} u \|_{L^{r}_{T}L^{q}_{x}} < +\infty, \quad \mbox{with} \quad \frac{2}{r} +\frac{n}{q} =1 - \beta.
\end{equation}
Let's start by the integral representation
\begin{equation}\nonumber
u  =  e^{t \Delta}u_{0} + \int_{0}^{t}e^{(t-s)\Delta}\mathbb{P}\nabla \cdot (u \otimes u)(s)ds.  
\end{equation}

\ni In order to get (\ref{OurYZCondition1ProofLoc}) we distinguish the case $\alpha \in [-1/2,0)$ and $\alpha \in [0,1)$.
\ni \subsection*{Case ${\bf \alpha \in [-1/2,0)}$} 
\begin{eqnarray}\nonumber
\| |x|^{\beta} u \|_{L^{r}_{T}L^{q}_{x}} &\leq & \| |x|^{\beta} e^{t\Delta}u_{0}\|_{L^{r}_{T}L^{q}_{x}} \\ \nonumber
&+& \left\| |x|^{\beta} \int_{0}^{t}e^{(t-s)\Delta}\mathbb{P}\nabla \cdot (u
\otimes u)(s) \ ds \right\|_{L^{r}_{T}L^{q}_{x}} \\ \nonumber
&=& I + II.
\end{eqnarray}
By the scaling assumption and corollary (\ref{IDecayCor}) we get:
\begin{equation}\label{ILoc}
I \leq c_{0} \||x|^{\alpha_{0}}u_{0}\|_{L^{p_{0}}_{|x|}L^{\p_{0}}_{\theta}}
\end{equation}
provided that 
\begin{equation}\label{I2Loc}
\quad p_{0} \leq q < \frac{np_{0}}{(\alpha_{0} - \beta)p +n-2}, \quad \p_{0} \leq q, \quad \Lambda_{\alpha_{0}, p_{0}, \p_{0}} \geq 0.
\end{equation}
To bound $II$ we use proposition (\ref{DDecayCor}) and scaling, so 
\begin{equation}\nonumber
II \leq   d_{0} \||x|^{\alpha} u \|^{2}_{L^{s}_{T}\Lpptilde}, 
\end{equation}
provided that
\begin{equation}\label{OurYZCond2ProofLoc}
2 \Lambda(\alpha,p, \p) \geq \beta,
\end{equation}
\begin{equation}\label{OurYZCond3ProofLoc}
p/2, \ \p /2 \leq q, \quad s/2 \leq r.
\end{equation}
Condition (\ref{OurYZCond2ProofLoc}) is ensured by
\begin{equation}\label{pLocInProofLoc}
\p \geq \frac{2(n-1)}{(2\alpha -\beta)p + 2(n-1)}.
\end{equation}
The last point is to show tha a couple $(\beta,r,q)$ such that (\ref{OurYZCond3ProofLoc}) is consistent with the scaling relation actually exists. 

\ni Because we are using theorem (\ref{YZTheorem}), we need to restrict to $-1\leq \beta$ and, in order to get the lower possible value for $\p$, we actually choose $\beta =-1$. In such a way the condition (\ref{pLocInProofLoc}) becomes the (\ref{OurYZCondition3Loc}). With this choice of $\beta$ we have 
$$
\p \leq p \quad \mbox{if} \quad -1/2 \leq \alpha,
$$
that's infact the range of $\alpha$ we have restricted on.
Then we choose $q=p/2$, so by the scaling relations we get
$$
\frac{2}{r} - \frac{4}{s} = 2 \alpha -2 \leq 0,
$$
that's consistent with $s/2 \leq r$.
Of course because of the choice $q= p/2$ and the scaling we have to require
$$
\max \left( 2, \frac{n}{1 - \alpha} \right) < p, \qquad \mbox{or} \quad p=2.
$$
Then the condition (\ref{I2Loc}) becomes  
$$
\quad p_{0} \leq q < \frac{np_{0}}{2p_{0} - 2},
$$ 
that by a straightforward calculation leads to (\ref{OurYZCond0ComplicataLoc}), $\alpha_{0} \in \left[ 1-n, \frac{2-n}{2+n}  \right)$.

\subsection*{Case ${\bf \alpha \in [0,1)}$} 
The only difference is again in the choice of $(\beta, r, q)$. 
Because of $\alpha \geq 0$ we can reach smaller values for $\p$, and we do it by requiring $2\alpha - \beta =1$ in (\ref{pLocInProofLoc}), in such a way
$$
\p \geq \frac{2(n-1)p}{p + 2(n-1)}.
$$ 
More precisely we choose 
$$
(\beta, r, q) = (2\alpha -1, s/2, p/2).
$$
A simple calculation shows that this choice is consistent with the scaling relation. Now by the condition (\ref{I2Loc}) and scaling we have
$$
p_{0} \leq q < \frac{np_{0}}{(2 - 2\alpha)p_{0} -2},
$$
that by a straightforward calculation leads to (\ref{OurYZCond0ComplicataBisLoc}), $\alpha_{0} \in \left[ 1-(1-\alpha)n, 1-(1-\alpha)\frac{2n}{2+n} \right)$.   
\end{proof}

\section{Well posedness with small data in weighted $L^p$ spaces}\label{smallDataMixed}

In chapter \ref{chap2} we have introduced the fundamental results and ideas of the small data theory for the solutions of the Navier Stokes equation. In this section we came back on this topic in order to get results in weighted Lebesgue spaces with angular integrability. In particular we will focus on theorem \ref{CKNSmallData}. As observed the small data theory is very well understood in the case of translation invariant adapted spaces, but not so much is known without this hypothesis. We will consider basically small data $u_{0}$ in the weighted space $L^{p}_{|x|^{\alpha p}d|x|}L^{\p}_{\theta}$ that's endowed with the norm
$$
\left( \int_{0}^{+\infty}
    \|f( \rho \ \cdot\ )\|^{p}_{L^{\p}(\mathbb{S}^{n-1})}
    \rho^{\alpha p + n-1} d \rho
  \right)^{\frac1p}.
$$
Of course by translations all the results we are going to prove still hold if the norms and the weights are centred in some point $\bar{x}$ different from the origin . In order to develop a small data theory we have then to restrict to critical spaces, so we need invariance under the scaling
$ \lambda \rightarrow \lambda u_{0}(\lambda x) $,
that, as observed before, is the right one for the initial datum of system (\ref{CauchyNS}).
This leads to the scaling relation
$$
\alpha = 1-\frac{n}{p}
$$
Notice that in this family there are the critical spaces $L^{3}$ and $L^{2}_{|x|^{2-n} dx}$ considered in theorems \ref{CKNSmallData} and \ref{CKNSmallData}. Regular strong solutions are available for small data in $L^{3}$, while smallness in $L^{2}_{|x|^{2-n} dx}$ gives only regularity localized in the interior of a space-time parabola centered on the origin. We conjecture that such behaviour is typical of the power weights $|x|^{\alpha}$ with $\alpha <0$. This is not surprising because in such case, even if the norms are defined on the whole space $\Rn$, the weights give a kind of localization near to the origin and a loss of informations at infinity occurs. Then we show that such informations can be recovered by a certain amount of angular integrability, and in this case, global regularity in space and time is availble. This is the case in the limit $\p \rightarrow \left( \frac{(n-1)p}{p-1} \right)^{-}$, as shown by the following

\begin{theorem}\label{WeighKato}
Let $n\geq 3$, $p \in [2,2+n]$,  $ \p \in [1, +\infty]$ and $\alpha, \p$ such that: 
 \begin{equation}\label{eq:condDL(Heat)1InThm}
    \alpha=1 - \frac{n}{p}, \qquad
    \p \geq \frac{(n-1)p}{p-1}.  
\end{equation}  
Let then $u_{0} \in L^{2}_{\sigma}(\Rn)$. There exists $\varepsilon > 0$, depending on all the parameter, such that if
$$
\||x|^{\alpha}u_{0}\|_{L^{p}_{|x|}L^{\p}_{\theta}} < \varepsilon
$$
then exists a unique global solution $u$ to the integral problem (\ref{IntegralCauchyNS}). Furthermore for all $p \leq q < \frac{np}{(\alpha - \beta)p + n-2}$, $\p \leq \q$, $r \in (1, +\infty)$ it holds
\begin{equation}\label{NotInfty}
\||x|^{\beta} u\|_{L^{q}_{|x|}L^{\q}_{\theta}} < 2c_{0}\varepsilon.
\end{equation}
provided that
\begin{equation}\label{eq:condDL(Heat)2InThm}
    \Lambda(\alpha, p, \p) > \Lambda (\beta,q,\q)   \qquad
    \frac{2}{r} + \frac{n}{p} = 1- \beta,
\end{equation}
and for all $p \leq q < \frac{np}{( 1 + \alpha - \beta)p + n-2}$, $\p \leq \q$, $r \in (1, +\infty)$ it holds
\begin{equation}\label{NotInftyDer}
\||x|^{\beta} \nabla u\|_{L^{q}_{|x|}L^{\q}_{\theta}} < 2c_{1}\varepsilon.
\end{equation}
provided that
\begin{equation}\label{eq:condDL(Heat)2InThmDer}
    \Lambda(\alpha, p, \p) > \Lambda (\beta,q,\q)   \qquad
    \frac{2}{r} + \frac{n}{p} = 2 - \beta;
\end{equation}
So $u$ is a regular ($C^{\infty}$ in the space variables) classical solution of (\ref{CauchyNS}). 
\end{theorem}

\begin{proof}
Let's start by the integral representation:
$$
u  =  e^{t \Delta}u_{0} + \int_{0}^{t}e^{(t-s)\Delta}\mathbb{P}\nabla \cdot (u \otimes u)(s)ds
$$
and consider the Picard sequence
$$
\begin{array}{rcl}
u_{1} & = &e^{t\Delta}u_{0} \\
u_{2} & = & e^{t \Delta}u_{0} + \int_{0}^{t}e^{(t-s)\Delta}\mathbb{P}\nabla \cdot (u_{1} \otimes u_{1})(s)ds  \\
u_{n} &=& e^{t\Delta}u_{0} + \int_{0}^{t}e^{(t-s)\Delta}\mathbb{P}\nabla \cdot (u_{n-1} \otimes u_{n-1})(s)ds.
 \end{array}
$$
We will show that $u_{n}$ is a Cauchy sequence in all the Banach spaces $L^{r}_{t}L^{q}_{|x|^{\beta q}d|x|}L^{\q}_{\theta}$ such that (\ref{eq:condDL(Heat)2InThm}) holds. Let's start noticing that it is a straightforward calculation to show that $p \leq q < \frac{np}{(\alpha- \beta)p + n-2}$, $\p \leq \q$, $r \in (1,+\infty)$, (\ref{eq:condDL(Heat)2InThm}) and (\ref{eq:condDL(Heat)1InThm}) are satisfied by a non null set of indeces. We use the weighted estimates (\ref{IHeat}), (\ref{DGHeatq=p}) with $\eta =0$ to bound by induction:
\begin{eqnarray}
\||x|^{\beta} u_{1}\|_{L^{r}_{t}L^{q}_{|x|}L^{\q}_{\theta}} & \leq & c_{0}\||x|^{\alpha}u_{0}\|_{L^{p}_{|x|}L^{\p}_{\theta}} = c_{0}\varepsilon \nonumber \\
\||x|^{\beta} u_{2}\|_{L^{r}_{t}L^{q}_{|x|}L^{\q}_{\theta}}&
\leq & c_{0}\varepsilon + d_{0} \||x|^{\beta}u_{1}\|^{2}_{L^{r}_{t}L^{q}_{|x|}L^{\q}_{\theta}} \nonumber  \\ 
&\leq & 2c_{0}\varepsilon \nonumber
\end{eqnarray}
if we take a small  $\varepsilon$ such that\footnote{At this step it could suffices to take $d_{0}c_{0}\varepsilon \leq 1$. The stronger condition $4d_{0}c_{0}\varepsilon \leq 1$ is used starting by the tirhd step.} $4d_{0}c_{0}\varepsilon \leq 1$. Then by induction
\begin{eqnarray}\label{moment}
\||x|^{\beta} u_{n}\|_{L^{r}_{t}L^{q}_{|x|}L^{\q}_{\theta}}&
\leq & c_{0}\varepsilon + d_{0} \||x|^{\beta}u_{n-1}\|^{2}_{L^{r}_{t}L^{q}_{|x|}L^{\q}_{\theta}}  \\ 
&\leq & 2c_{0}\varepsilon \nonumber
\end{eqnarray}
again because $4d_{0}c_{0}\varepsilon \leq 1$. So the sequence is well defined in $L^{r}_{t}L^{q}_{|x|^{\beta q}d|x|}L^{\q}_{\theta}$. 
Now we have to show that $u_{n}$ is a Cauchy sequence. Agin using (\ref{IHeat}), (\ref{DGHeatq=p}) we bound the differences:
\begin{eqnarray}
& & \||x|^{\beta} (u_{n} - u_{n-1})\|_{L^{r}_{t}L^{q}_{|x|}L^{\q}_{\theta}}\nonumber \\
& = & 
 \left\||x|^{\beta}\int_{0}^{t}e^{(t-s)\Delta}\mathbb{P} \nabla \cdot (u_{n-1}\otimes u_{n-1} -
 u_{n-2}\otimes u_{n-2})(s)ds\right\|_{L^{r}_{t}L^{q}_{|x|}L^{\q}_{\theta}}   \nonumber  \\  
 &\leq & \left\||x|^{\beta}\int_{0}^{t}e^{(t-s)\Delta}\mathbb{P} \nabla \cdot (u_{n-1}\otimes (u_{n-1}-u_{n-2}))(s)ds\right\|_{L^{r}_{t}L^{q}_{|x|}L^{\q}_{\theta}}   \nonumber \\
 &+& \left\||x|^{\beta}\int_{0}^{t}e^{(t-s)\Delta}\mathbb{P} \nabla \cdot (u_{n-2}\otimes(u_{n-1} \otimes u_{n-2}))(s)ds\right\|_{L^{r}_{t}L^{q}_{|x|}L^{\q}_{\theta}}    \nonumber \\
 &\leq & d_{0}(\||x|^{\beta}u_{n-1}\|_{L^{r}_{t}L^{q}_{|x|}L^{\q}_{\theta}}\||x|^{\beta}(u_{n-1} - u_{n-2})\|_{L^{r}_{t}L^{q}_{|x|}L^{\q}_{\theta}}    \nonumber \\
 & + & \||x|^{\beta}u_{n-2}\|_{L^{r}_{t}L^{q}_{|x|}L^{\q}_{\theta}}\||x|^{\beta} (u_{n-1}-u_{n-2})\|_{L^{r}_{t}L^{q}_{|x|}L^{\q}_{\theta}})  \nonumber \\
 & \leq & d_{0}c_{0}\varepsilon\||x|^{\beta} (u_{n-1} - u_{n-2})\|_{L^{r}_{t}L^{q}_{|x|}L^{\q}_{\theta}} \nonumber 
\end{eqnarray}
where we used the uniform bound (\ref{moment}). In $n-2$ steps we get:
\begin{eqnarray}
& & \||x|^{\beta} (u_{n}-u_{n-1})\|_{L^{r}_{t}L^{q}_{|x|}L^{\q}_{\theta}} \nonumber \\
 &\leq & (2d_{0}c_{0}\varepsilon)^{n-1} \||x|^{\beta}(u_{2}-u_{1})\|_{L^{r}_{t}L^{q}_{|x|}L^{\q}_{\theta}}  \nonumber \\
 &=& (2c_{0}\varepsilon)^{n-1} \left\||x|^{\beta}\int_{0}^{t}e^{(t-s)\Delta}\mathbb{P} \nabla \cdot (u_{1}\otimes u_{1})(s)ds\right\|_{L^{r}_{t}L^{q}_{|x|}L^{\q}_{\theta}}    \nonumber  \\ 
&\leq & (2d_{0}c_{0}\varepsilon)^{n}. \nonumber
\end{eqnarray}
So the differences $u_{n}-u_{n-1}$ are bounded by a gemoetric series, and easily follows that $u_{n}$ is a Cauchy sequences.

\ni The regularity of $u$ is now a direct consequence of theorem (\ref{OurYZTheorem}). In particular by setting\footnote{It is again straightforward to show that conditions $(\ref{eq:condDL(Heat)2InThm})$ are satisfited also under the restriction $\beta=0, q=\q$.} $\beta=0, q=\q$ is possible to refer to the original Serrin's result (\ref{SerrinRegularity}).  
Notice that 
$$
\Lambda(\alpha,p,\p) \geq 0 \Rightarrow \p \geq \frac{(n-1)p}{p-1},
$$
$$
p \leq q < \frac{np}{p-2} \Rightarrow p \in [2,2+n].
$$
In order to prove (\ref{NotInftyDer}) we use (\ref{IHeat}), (\ref{DGHeatq=p}) with $\eta = 1$, so

\begin{eqnarray}
\||x|^{\beta} \nabla u_{1}\|_{L^{r}_{t}L^{q}_{|x|}L^{\q}_{\theta}} & \leq & c_{1}\||x|^{\alpha}u_{0}\|_{L^{p}_{|x|}L^{\p}_{\theta}} = c_{1}\varepsilon \nonumber \\
\||x|^{\beta} \nabla u_{2}\|_{L^{r}_{t}L^{q}_{|x|}L^{\q}_{\theta}}&
\leq & c_{1}\varepsilon + d_{1} \||x|^{\beta} \nabla u_{1}\|^{2}_{L^{r}_{t}L^{q}_{|x|}L^{\q}_{\theta}} \nonumber  \\ 
&\leq & 2c_{1}\varepsilon \nonumber
\end{eqnarray}
if we take a small  $\varepsilon$ such that $4d_{1}c_{1}\varepsilon \leq 1$. Then by induction
\begin{eqnarray}\label{momentDer}
\||x|^{\beta} \nabla u_{n}\|_{L^{r}_{t}L^{q}_{|x|}L^{\q}_{\theta}}&
\leq & c_{1}\varepsilon + d_{1} \||x|^{\beta}u_{n-1}\|^{2}_{L^{r}_{t}L^{q}_{|x|}L^{\q}_{\theta}}  \\ 
&\leq & 2c_{1}\varepsilon. \nonumber
\end{eqnarray}

\end{proof}

\ni Actually this theorem is a particular case of the Koch-Tataru theorem and can be proved directly by the methods in \cite{Tat} and the estimates in proposition \ref{IDecayCor}. Anyway  we prefer a more direct proof, in particular of the bounds (\ref{NotInfty}, \ref{NotInftyDer}). Then also the regularity of $u$ is a really difficult problems for the Koch-Tataru solutionts (see \cite{Banica}), so we prefer to prove it directly in our case. 

\ni

\ni The theorem show that if we work in $L^{p}_{|x|^{\alpha p}d|x|}L^{\p}_{\theta}$ a sufficiently high angular integrability in order to recover the loss of informations for large $|x|$ is
$$
\p_{G} = \frac{(n-1)p}{p-1}.
$$
Now it is really interesting to understand what happens in the range
$$
p < \p < \p_{G}. 
$$
We have no definitive results in this direction but a clear way to pursue. We will show how this problem is connected to the behaviour of the Leray solutions close to the null solutions. In order to simplify as more as possible the notation let consider just the spaces
$$
L^{2}_{|x|^{-1}d|x|}L^{\p}_{\theta}, \quad 2 < \p < \p_{G}=4. 
$$
So we set $n=3, p=2$. This is also the most interesting case because the quantity involved are related to real physical quantities. Let now recall the theorem \ref{CKNSmallData}
\begin{theorem}[Caffarelli-Kohn-Nirenberg]\label{CKNSmallDataBis}
Let $u_{0} \in L^{2}_{\sigma}(\Rn)$ and $u$ a suitbale weak solution of (\ref{CauchyNS}). There exists an absolute constant $\varepsilon_{0} > 0$ such that if
$$
\| |x|^{-1/2} u_{0} \|_{L^{2}(\Rn)} = \varepsilon < \varepsilon_{0},
$$ 
then $u$ is regular ($C^{\infty}$ in space variables) in the interior of the parabola
\begin{equation}\label{PiLimit}
\Pi_{2} = \left\{ (t,x) \quad \mbox{t.c.} \quad t > \frac{|x|^{2}}{\varepsilon_{0}- \varepsilon} \right\}.
\end{equation}
\end{theorem}
\ni We used the notation $\Pi_{2}$ to remind the fact that th authors work with $L^{2}$ integrability in the angular variables.
What we expect is a growth in the size of $\Pi_{\p}$ for bigger values of $\p$, and the filling of the whole space-time in the limit $\p \rightarrow 4^{-}$. 
To be more precise we expect to show regularity in the set
\begin{equation}\label{Conj1}
\Pi_{\p} = \left\{ (t,x) \in \Rpiu \times \Rn \quad \mbox{s.t} \quad
t > \frac{c(\p) |x|^{2}}{\varepsilon_{0} - \varepsilon} \right\},
\end{equation}
for small data in $L^{2}_{|x|^{-1}d|x|}L^{\p}_{\theta}$, and furthermore
\begin{equation}\label{Conj2}
c(\p) \rightarrow 0, \quad \mbox{as} \quad \p \rightarrow 4^{-}. 
\end{equation}
In such a way the gap between theorem \ref{CKNSmallDataBis} and \ref{WeighKato} would be completely covered.  
It turns out that this behaviour is strictly connected with a possible improvements of theorem \ref{CKNSmallData} in the limit $\varepsilon \rightarrow 0$. This is a non trivial point and has of course an intrinsic interest. If we take the limit $\varepsilon \rightarrow 0$ in $(\ref{PiLimit})$ we get the maximal regularity set
$$
\Pi = \left\{ t > \frac{|x|^{2}}{\varepsilon_{0}} \right\}.
$$
On the other hand seems reasonable to conjecture improvements to the size of $\Pi$. A possibility is
\begin{equation}\label{PossibleImprov}
\Pi_{\varepsilon} = \left\{ t > \frac{c(\varepsilon)|x|^{2}}{\varepsilon_{0} - \varepsilon}  \right\},
\end{equation}  
with
\begin{equation}\label{ConjVare}
c(\varepsilon) \rightarrow 0, \quad \mbox{as} \quad \varepsilon \rightarrow 0.
\end{equation}
Let now show how (\ref{ConjVare}) implies (\ref{Conj1}, \ref{Conj2}). The idea is to perform a decomposition of the initial datum inspired by a similar argument in \cite{Calderon}. We split $u_{0}$ in
$$
u_{0}= v_{0} + w_{0}
$$
with
$$
\nabla \cdot v_{0}= \nabla \cdot w_{0} =0,
$$
and
$$
v_{0} \in L^{2}_{|x|^{-1}dx}, \quad w_{0} \in L^{2}_{|x|^{-1}L^{4}_{\theta}}.
$$
Moreover we require that
$$
\| |x|^{-1/2} w_{0}\|_{L^{2}_{|x|}L^{4}_{\theta}} \rightarrow 0 \quad \mbox{as}
\quad \p \rightarrow 2^{+}
$$
and
$$
\| |x|^{-1/2} v_{0}\|_{L^{2}_{x}} \rightarrow 0 \quad \mbox{as}
\quad \p \rightarrow 4^{-}
$$
In order to achieve such a decomposition we slightly modify the argument in \cite{Calderon}. Let $s >0$ and define
$$
u_{0} = u_{0,<s} + u_{0,>s},
$$
where $u_{0,<s}$ is equal to $u_{0}$ if $|u_{0}| <s$ and is zero otherwise . Then $v_{0}, w_{0}$ are defined by 
$$
v_{0} = \lim_{t \rightarrow 0} e^{t\Delta} \mathbb{P} u_{0,>s},
$$
$$
w_{0} = \lim_{t \rightarrow 0} e^{t\Delta} \mathbb{P} u_{0,<s}.
$$
It follows easily by the representation of $e^{t\Delta} \mathbb{P}$ as a convolution operator that the limits are attained respectively in $L^{2}_{|x|^{-1}}dx$ and $L^{2}_{|x|^{-1}}L^{4}_{\theta}$ and the properties $\nabla \cdot v_{0}= \nabla \cdot u_{0}$. Furthermore by a simple interpolation argument\footnote{see \cite{Calderon}.} 
\begin{equation}\label{Decomposition}
\begin{array}{lcl}
\| |x|^{-1/2} w_{0}\|_{L^{2}_{|x|}L^{4}_{\theta}} 
&\leq & C s^{1-\frac{\p}{4}} \| |x|^{-1/2} u_{0} \|^{\p/ 4}_{L^{2}_{|x|}L^{\p}_{\theta}} \\
&& \\
\| |x|^{-1/2} v_{0}\|_{L^{2}_{x}}
&\leq & C s^{1-\frac{\p}{2}} \| |x|^{-1/2} u_{0} \|^{\p /2}_{L^{2}_{|x|}L^{\p}_{\theta}},
\end{array}
\end{equation}
for each $s > 0$. Then we choose $s= \frac{\theta}{1-\theta}$, where $\theta$ is defined by
$$
\frac{1}{\p} = \frac{1-\theta}{2} + \frac{\theta}{4}.
$$
Notice also that
$$
1 - \frac{\p}{4} = \frac{1 - \theta}{2 - \theta}, \qquad 1- \frac{\p}{2} = - \frac{\theta}{1-\theta}.
$$
To simplify the notation we set 
$$
A_{\theta} = C \left( \frac{\theta}{1-\theta} \right)^{\frac{1- \theta}{2 - \theta}}, \qquad 
B_{\theta} = C \left( \frac{\theta}{1-\theta} \right)^{- \frac{\theta}{2-\theta}},
$$
in such a way (\ref{Decomposition}) becomes
\begin{equation}\label{DecompositionBis}
\begin{array}{lcl}
\| |x|^{-1/2} w_{0}\|_{L^{2}_{|x|}L^{4}_{\theta}} 
&\leq & A_{\theta} \| |x|^{-1/2} u_{0} \|^{\p/4}_{L^{2}_{|x|}L^{\p}_{\theta}} \\
&& \\
\| |x|^{-1/2} v_{0}\|_{L^{2}_{x}}
&\leq & B_{\theta} \| |x|^{-1/2} u_{0} \|^{\p /2}_{L^{2}_{|x|}L^{\p}_{\theta}}.
\end{array}
\end{equation}
A straightforward calculation shows that $A_{\theta}$ is an increasing function of $\theta$ and
\begin{equation}
\lim_{\theta \rightarrow 0} A_{0}=0, \quad \lim_{\theta \rightarrow 1} A_{\theta}=1;
\end{equation}
while $B_{\theta}$ is a decreasing function of $\theta$ such that
\begin{equation}\label{ultima}
\lim_{\theta \rightarrow 0} B_{0}=1, \quad \lim_{\theta \rightarrow 1} B_{\theta}=0.
\end{equation}

\ni Then we consider the Cauchy problems
\begin{equation}\label{CauchyNSForW}
\left \{
\begin{array}{rcl}
\partial_{t}w + (w \cdot \nabla) w +\nabla p_{w}  & = & \Delta w   \\
\nabla \cdot w & = & 0   \\
w & = & w_{0}, 
\end{array}\right.
\end{equation}
and
\begin{equation}\label{CauchyNSForv}
\left \{
\begin{array}{rcl}
\partial_{t}v + (v \cdot \nabla) v + 
(v \cdot \nabla) w + (w \cdot \nabla) v
 +\nabla p_{v}  & = & \Delta v   \\
\nabla \cdot v & = & 0  \\
v & = & v_{0}. 
\end{array}\right.
\end{equation}
Of course $u=v+w$ and the pressure $p_{v}, p_{w}$ can be still recovered by $v,w$ through
$$
p_{w} = C \left( \sum_{i,j =1}^{n} R_{i}R_{j}(w_{i}w_{j}) \right),
$$
and
$$
p_{v} = C \left( \sum_{i,j =1}^{n} R_{i}R_{j}(v_{i}v_{j})
+2 \sum_{i,j =1}^{n} R_{i}R_{j}(v_{i}w_{j})  \right).
$$
At first we notice that the global regularity of $w$ is ensured by theorem \ref{WeighKato} provided that 
$$
A_{\theta} \| u_{0} \|^{\p/4}_{L^{2}_{|x|}L^{\p}_{\theta}} < \varepsilon_{4}.
$$
We used the notation $\varepsilon_{4}$ to emphasize the fact that the smallness condition is about the $L^{2}_{|x|^{-1}d|x|}L^{4}_{\theta}$ norm of $w_{0}$. 
Then by (\ref{NotInfty}, \ref{NotInftyDer}) we also get the bound
\begin{equation}\label{BoundOnW}
\begin{array}{lcl}
\| |x|^{\beta} w \|_{L^{r}_{t}L^{q}_{x}} &\leq & c_{0} \varepsilon_{4}, 
\quad \frac{2}{r} + \frac{3}{q} = 1-\beta \\
&& \\
\| |x|^{\beta} \nabla w \|_{L^{r}_{t}L^{q}_{x}} &\leq & c_{0} \varepsilon_{4}, 
\quad \frac{2}{r} + \frac{3}{q} = 2-\beta. 
\end{array}
\end{equation}
These bounds are necessary in order to handle the terms $(v \cdot \nabla) w, (w \cdot \nabla) v$ in (\ref{CauchyNSForv}). Then if we are able to prove\footnote{Actually we conjecture (\ref{PossibleImprov}, \ref{ConjVare}) for the perturbed system (\ref{CauchyNSForv}). Anyway the additional terms are easily handled by using (\ref{BoundOnW}).} (\ref{PossibleImprov}, \ref{ConjVare}) we get the following regularity set for $v$:
$$
\Pi = \left\{ (t,x) \quad \mbox{t.c.} \quad t > \frac{c(\varepsilon)|x|^{2}}{\varepsilon_{0}- \varepsilon} \right\},
$$
with $\varepsilon = \| |x|^{-1/2}v_{0}\|_{L^{2}_{x}} \leq B_{\theta} \| |x|^{-1/2} u_{0} \|^{\p /2}_{L^{2}_{|x|}L^{\p}_{\theta}}$ sufficiently small. Now by using (\ref{ultima}) we get $c(\varepsilon) \rightarrow 0$ as $ \p \rightarrow 4^{-}$ (or $\theta \rightarrow 1$).

\chapter*{Outlooks and remarks}

The consequences of angular integrability in Sobolev embeddings and PDEs have been considered by many authors. We have focused basically on the applications to the Navier-Stokes equation, but, as mentioned in section \ref{SobEmb} (Chapter \ref{SectInequality}), this point of view is natural and useful in the context of Srichartz estimates for the wave and Schr\"odinger equations on $\Rn$. A comprehensive reference about its application to the wave equation is \cite{JiangWangYu10-a}. 

\ni The consequences of higer angular integrability have been explored also in te context of the Dirac equation, see \cite{DanconaCacciafesta11-a}, \cite{MachiharaNakamuraNakanishi05-a}.

\ni We have in mind to conclue by suggesting some additional consequences of propositions \ref{IDecayCor}, \ref{DDecayCor}. As we have seen until now the key point in global regularity results is the request
\begin{equation}\label{LambdaOtlooks}
\Lambda (\alpha, p, \p) \geq 0
\end{equation}   
for a solution $u$ bounded in
\begin{equation}\label{BoundOutlooks}
\| |x|^{\alpha} u \|_{L^{s}_{t}L^{p}_{|x|}L^{\p}_{\theta}}, \qquad 
\frac{2}{s} + \frac{n}{p} = 1 - \alpha. 
\end{equation}
Then the regularity and uniqeness problems are strictly connected. We have given as example the theorem \ref{SerrinUniqness} in which the uniqueness of weak solutions bounded in 
\begin{equation}\nonumber
\| u \|_{L^{s}_{t}L^{p}_{x}} < +\infty
\end{equation}
is proved. In the same spirit by applying Sobolev embeddings in $L^{p}_{|x|^{\alpha}d|x|^{\alpha}}L^{\p}_{\theta}$ spaces under condition \ref{LambdaOtlooks}, we get uniqueness of weak solutions bounded in the norms \ref{BoundOutlooks}.

\ni Anyway we have basically focused on theorems \ref{OurYZTheorem}, \ref{OurYZTheoremLoc} to show the difference between local and global regularity results. These theorems can be further  extended in some directions. In particular a larger set of indeces can be covered about the $L^{p}$ integrability. The restriction on $p$ from below, i.e.
$$
\frac{n}{1-\alpha} < p, \qquad 4 \leq p \quad \mbox{or} \quad 2 \leq p,
$$    
can be relaxed by using respectively the local regularity criteria in \cite{CKN}, \cite{Kukavica}. We omit the details.

\chapter*{Acknowledgments}

The author would like to thank prof. Piero D'Ancona for the work done together and constant suggestions and encouragement.

\end{document}